\documentclass[12pt]{amsart}

\usepackage{amsmath, amssymb, amsfonts}
\usepackage{graphicx, overpic}
\usepackage{fullpage}
\usepackage{color}
\usepackage[usenames,dvipsnames,svgnames,table]{xcolor}
\usepackage{enumerate}

\theoremstyle{plain}
\newtheorem{thm}{Theorem}[section]
\newtheorem{lemma}[thm]{Lemma}
\newtheorem{prop}[thm]{Proposition}
\newtheorem{cor}[thm]{Corollary}

\numberwithin{equation}{section}
\numberwithin{figure}{section}

\theoremstyle{definition}
\newtheorem{definition}[thm]{Definition}
\newtheorem{example}[thm]{Example}
\newtheorem{remark}[thm]{Remark}

\newtheorem{obs}[thm]{Observation}

\newtheorem{assumption}[thm]{Standing Assumptions}

\newcommand{\G}{\Gamma}

\newcommand{\R}{\mathbb{R}}
\newcommand{\N}{\mathbb{N}}
\renewcommand{\H}{\mathbb{H}}
\newcommand{\Cayley}[1]{\mathcal {C}_#1}
\newcommand{\CAT}{\operatorname{CAT}}
\newcommand{\cC}{\mathcal{C}}

\newcommand{\val}{\mathrm{val}}
\newcommand{\into}{\hookrightarrow} % injection

\newcommand{\cA}{\mathcal{A_{\approx}}}
\newcommand{\cAplus}{\mathcal{A}^+}
\newcommand{\cB}{\mathcal{B}}
\newcommand{\cT}{\mathcal T}

\newcommand{\hide}[1]{}

\newcommand{\card}{\operatorname{card}}

\def\polhk#1{\setbox0=\hbox{#1}{\ooalign{\hidewidth
    \lower1.0ex\hbox{$\,\lhook$}\hidewidth\crcr\unhbox0}}}
\newcommand{\Swiatkowski}{\'Swi{\polhk{a}}tkowski}

\title{Bowditch's JSJ tree and the quasi-isometry classification of certain Coxeter groups}
\author{Pallavi Dani and Anne Thomas}
\date{\today}
\thanks{The first author 
was partially supported by NSF Grant No.~DMS-1207868 and this
work was supported by a grant from the Simons Foundation (\#426932, Pallavi Dani). 
This research of the second author was supported in part by ARC Grant No.~DP110100440.  The second author was supported in part by an Australian Postdoctoral Fellowship.}

\begin{document}
\maketitle

\begin{abstract}  Bowditch's JSJ tree for splittings over $2$-ended subgroups is a quasi-isometry invariant for $1$-ended hyperbolic groups which are not cocompact Fuchsian~\cite{bowditch}.  Our main result gives an explicit, computable ``visual" construction of this tree for certain hyperbolic right-angled Coxeter groups.  As an application of our construction we identify a large class of such groups for which the JSJ tree, and hence the visual boundary, is a complete quasi-isometry invariant, and thus the quasi-isometry problem is decidable.  We also give a direct proof of the fact that among the Coxeter groups we consider, the cocompact Fuchsian groups form a rigid quasi-isometry class. In Appendix~\ref{app:Kn}, written jointly with Christopher Cashen, we show that the JSJ tree is not a complete quasi-isometry invariant for the entire class of Coxeter groups we consider.
\end{abstract}

%%%%%%%%%%%%%%%%%%%%%%%
%%%%%%%%%%%%%%%%%%%%%%%
\section{Introduction}
%%%%%%%%%%%%%%%%%%%%%%%
%%%%%%%%%%%%%%%%%%%%%%%

In this paper we use Bowditch's JSJ tree to investigate the quasi-isometry classification of certain right-angled Coxeter groups.  After describing our results, we review previous work on this classification, about which surprisingly little is known.
 
Given a finite simplicial graph $\G$, the associated \emph{right-angled Coxeter group} $W_\G$ has generating set $S$ equal to  the vertices of $\G$, and relations $s^2 = 1$ for all $s \in S$, and $st = ts$ whenever $s,t \in S$ are adjacent vertices.   We call $\G$  the \emph{defining graph} of $W_\G$.  We assume throughout that~$W_\G$ is infinite, or equivalently that $\G$ is not a complete graph.  We also assume that $W_\G$ is \emph{$2$-dimensional}, that is, that $\G$ has no triangles.  This is equivalent to the Davis complex for $W_\G$ being $2$-dimensional (see Section~\ref{sec:davis} for background on the Davis complex).

Recall that a \emph{special subgroup} of $W_\G$ is one generated by a subset of the generating set $S$.  For each $T \subseteq S$, the special subgroup $\langle T\rangle$ is also a right-angled Coxeter group, whose defining graph has vertex set $T$ and edge set consisting of all edges of $\G$ which have both endpoints in $T$.   By Proposition 8.8.2 of~\cite{davis-book}, if $W_\G$ has infinitely many ends then $W_\G$ has a nontrivial decomposition as a finite tree of groups, in which each vertex group is a finite or $1$-ended special subgroup and each edge group is a finite special subgroup.  It then follows from Theorems 0.3 and 0.4 of Papasoglu and Whyte~\cite{papasoglu-whyte} that for $W_\G$ and $W_\Lambda$ with infinitely many ends, $W_\Gamma$ is quasi-isometric to $W_\Lambda$ if and only if their respective tree-of-groups decompositions have the same set of quasi-isometry types of $1$-ended vertex groups (without multiplicities).
 In light of these results, we focus on the case that $W_\G$ is $1$-ended in this paper. 

JSJ decompositions of groups were introduced by Sela~\cite{sela} and Rips and Sela~\cite{rips-sela}; see Guirardel--Levitt~\cite{guirardel-levitt} for an overview of the theory.
In~\cite{bowditch}, Bowditch gives a construction of a canonical JSJ splitting over $2$-ended subgroups of any $1$-ended hyperbolic group $G$ which is not cocompact Fuchsian.  (A finitely generated group is \emph{Fuchsian} if it is non-elementary and acts properly discontinuously on the hyperbolic plane.)  The JSJ tree $\cT_G$ for this splitting is a quasi-isometry invariant, as it is defined in terms of the local cut point structure of the boundary.  The group $G$ acts on its JSJ  tree $\cT_G$ with finite quotient.  We recall the construction of Bowditch's JSJ tree in Section~\ref{sec:bowditch}.

Our main result gives an explicit construction of Bowditch's JSJ tree $\cT_{W_\G}$ for certain $W_\G$; see Theorem~\ref{thm:main} below.  We restrict our attention to the class of 2-dimensional, $1$-ended, hyperbolic right-angled Coxeter groups $W_\G$, which are not cocompact Fuchsian and which admit a non-trivial splitting over a $2$-ended subgroup (if there is no such splitting then $\cT_{W_\G}$ consists of a single vertex).  Each of these assumptions corresponds to certain properties of the (triangle-free) graph $\G$, as follows.  First, by a special case of~\cite[Theorem 8.7.2]{davis-book}, the group $W_\G$ is $1$-ended if and only if $\G$ is connected and has no separating vertices or edges.  Also, by a special case of~\cite[Corollary 12.6.3]{davis-book}, the group $W_\G$ is hyperbolic if and only if $\G$ has no embedded cycles of length four.  

A Coxeter group $W$ (not necessarily right-angled) is cocompact Fuchsian if and only if $W$ is either a hyperbolic polygon reflection group, or the direct product of a hyperbolic polygon reflection group with a finite Coxeter group.  This result will not surprise experts, but as we could not find it stated explicitly in the literature, we record a proof in Appendix~\ref{sec:fuchsian-all}, as 
Theorem~\ref{thm:fuchsian-all}.  For $\G$ triangle-free, this implies that  $W_\G$ is cocompact Fuchsian if and only if $\G$ is a cycle of length $\ge 5$.  

As for splittings, Papasoglu showed in~\cite{papasoglu} that among 
$1$-ended, finitely presented groups that are not commensurable to surface groups, having a splitting over a $2$-ended subgroup is a quasi-isometry invariant.   
In~\cite{mihalik-Tschantz} Mihalik and Tschantz 
characterised the Coxeter groups which admit such splittings in terms of their defining graphs.  
A pair of vertices $\{a, b\}$ in $\G$ is called a \emph{cut pair} if $\{a,b\}$ separates $\G$, meaning that $\G \setminus \{a,b\}$ has at least two components, each of which contains at least one vertex.
In our triangle-free setting, the main result of~\cite{mihalik-Tschantz} states that $W_\G$ splits over a $2$-ended subgroup if and only if there is a cut pair $\{a,b\}$ in $\G$ (see Section~\ref{sec:splittings} for more details).

Correspondingly, we restrict to graphs $\G$ satisfying the following conditions. 

\begin{assumption} \label{assumptions}
The graph $\G$: 
\begin{enumerate}
\item has no triangles ($W_\G$ is 2-dimensional);
\item  is connected and has no separating vertices or edges ($W_\G$ is $1$-ended);
\item has no squares ($W_\G$ is hyperbolic);
\item is not a cycle of length $\ge 5$ ($W_\G$ is not cocompact Fuchsian); and
\item has a cut pair of vertices $\{a, b\}$ ($W_\G$ splits over a $2$-ended subgroup).
\end{enumerate}
\end{assumption}

The main result of this paper is a ``visual'' description of Bowditch's JSJ tree for such $W_\G$: 
\begin{thm}\label{thm:main}
For $\G$ satisfying Standing Assumptions~\ref{assumptions}, the graph $\G$ visually determines Bowditch's JSJ tree $\cT_\G = \cT_{W_\G}$, in the sense that the $W_\G$-orbits of vertices and edges of $\cT_\G$ are  in bijection with subsets of vertices of the defining graph $\G$ which satisfy certain explicit graph-theoretic conditions.  
Moreover, the stabilisers of the vertices and edges of $\cT_\G$ are conjugates of certain special subgroups of $W_\G$. 
\end{thm}

We prove Theorem~\ref{thm:main} in Section~\ref{sec:JSJ}, and an explicit statement of this theorem appears as  
Theorem~\ref{thm:JSJ}.  A consequence of our description is:

\begin{cor}\label{cor:algorithm}
For $\G$ satisfying Standing Assumptions~\ref{assumptions}, there exists an algorithm to compute the JSJ tree of the right-angled Coxeter group $W_\G$.
\end{cor}

In order to prove Theorem~\ref{thm:main}, we relate the separation properties of subsets of vertices of~$\G$, of geodesics in the Davis complex, and of endpoints of geodesics in the visual boundary~$\partial W_\G$.  We are able to use a few arguments from Lafont's papers~\cite{lafont-3d, lafont} in our proofs, but substantial additional work is required.  The vertices of Bowditch's JSJ tree $\cT_\G$ are of three types.  The first two types are defined in terms of subsets of $\partial W_\G$, and include the so-called $\approx$-pairs and $\sim$-classes.  The vertices of the third type are called stars, and are defined as subsets of vertices of the first two types.  As we explain further in 
Remark~\ref{rem:lafont}, the $\approx$-pairs and certain $\sim$-classes in our setting correspond to key structures in the spaces investigated by Lafont in~\cite{lafont-3d,lafont}.  Similar spaces have also been considered by Crisp--Paoluzzi~\cite{crisp-paoluzzi}, Malone~\cite{malone} and Stark~\cite{stark}.  However the situation for $\sim$-classes is more varied in our setting.  Moreover, there is no analogue of stars 
in~\cite{lafont-3d, lafont} and these other works.   The subtleties of $\sim$-classes and the identification of stars are the most difficult parts of the proof of Theorem~\ref{thm:main}.

Our requirement that the defining graph have no triangles (condition (1) in Standing 
Assumptions~\eqref{assumptions}) is a simplifying assumption.  
For instance, this assumption makes it easier to characterise when $W_\G$ is $1$-ended, hyperbolic and not cocompact Fuchsian, and when $W_\G$ admits a splitting over a $2$-ended subgroup (conditions (2)--(5) in Standing Assumptions~\eqref{assumptions}).  These in turn lead to relatively straightforward descriptions of the JSJ tree in the triangle-free case.  

Although we anticipate that many of our arguments can be adapted to the higher-dimensional setting, that is, where the graph $\G$ does contain triangles, such a generalisation will not be immediate.  As explained in Remark~\ref{rem:approx}, the identification of the $\approx$-pair vertices of the JSJ tree, and the determination of their valence, would become quite delicate.     
Modifying the characterisations of the $\sim$-class vertices and the stars would require even more substantial work.  For example, one condition in our description of the $\sim$-class vertices 
(in Proposition~\ref{prop:sim}) involves disallowing certain subgraphs which are subdivided copies of $K_4$, where $K_4$ is the complete graph on $4$ vertices.  If $\G$ were not triangle-free, this description would have to distinguish between $K_4$'s that ``matter'' and those that don't.   Our proofs also rely on a series of graph-theoretic preliminary results in Section~\ref{sec:sim}, whose statements would become more complicated in the presence of triangles.

We give some applications of our main result in Section~\ref{sec:applications}.  First, we show that the JSJ tree is a complete invariant for a large subclass of groups.  Let $\mathcal G$ be the class of  graphs satisfying Standing Assumptions~\ref{assumptions} which have no induced subgraphs which are subdivided copies of~$K_4$.

\begin{thm}\label{thm:complete-inv}
Bowditch's JSJ tree is a complete quasi-isometry invariant for $W_\G$ with $\G \in \mathcal G$. More precisely, given $\Lambda, \G \in \mathcal G$, the groups $W_\G$ and $W_\Lambda$ are quasi-isometric if and only if there is a type-preserving isomorphism from $\cT_\G$ to $\cT_\Lambda$.
\end{thm}

By Theorem 3.9 of Cashen and Martin~\cite{cashen-martin}, the JSJ tree is a complete quasi-isometry invariant for any class of groups in which the JSJ tree has no stars.  Using this, in Section~\ref{sec:k4} we prove Theorem~\ref{thm:complete-inv} by characterising the groups $W_\G$ for which $\cT_\G$ has no stars as exactly those for which the corresponding graph $\G$ has no subdivided $K_4$ subgraphs.  

In Appendix~\ref{app:Kn}, which is joint work with Cashen, we use techniques from Cashen and Macura~\cite{cashen-macura} to show  (in Theorem~\ref{thm:Kn})  that if $\G_{n+1}$ is a sufficiently subdivided copy of the complete graph on $n+1$ vertices, then for $n \geq 3$ the groups $W_{\G_{n+1}}$ all have the same JSJ tree, but are pairwise not quasi-isometric.  Hence the JSJ tree is not a  complete quasi-isometry invariant in general (i.e.~for all graphs satisfying Standing Assumptions~\ref{assumptions}).

Combining Corollary~\ref{cor:algorithm} with Theorem~\ref{thm:complete-inv}, we immediately obtain the following computability result.  This was pointed out to us by an anonymous referee and by Murray Elder.

\begin{cor}\label{cor:decidable}  
The quasi-isometry problem is decidable for the class of right-angled Coxeter groups $W_\G$ with $\G \in \mathcal G$.
\end{cor}

As a consequence of Bowditch's construction of the JSJ tree and Theorem~\ref{thm:complete-inv}, we also obtain that the visual boundary is a complete invariant for the same class of groups.

\begin{cor}\label{cor:complete-inv}  Let $\G,\Lambda \in \mathcal G$ (where $\mathcal G$ is as defined above the statement of Theorem~\ref{thm:complete-inv}). Then $W_\G$ and $W_\Lambda$ are quasi-isometric if and only if $\partial W_\G$ and $\partial W_\Lambda$ are homeomorphic.
\end{cor}

The visual boundaries of hyperbolic groups have been investigated by \Swiatkowski\ and coauthors in a series of papers~\cite{martin-swiatkowski,swiatkowski-manifolds,swiatkowski-metric,swiatkowski-dense,swiatkowski-sierpinski}.  In particular, in~\cite{martin-swiatkowski} Martin and \Swiatkowski\ consider visual boundaries of fundamental groups of graphs of groups with finite edge groups.  Suppose now that $W_\G$ and $W_\Lambda$ are $2$-dimensional hyperbolic right-angled Coxeter groups with  infinitely many ends.  When applied to the above-mentioned tree-of-groups decompositions for $W_\G$ and $W_\Lambda$, the results of~\cite{martin-swiatkowski} imply that $\partial W_\G$ and $\partial W_\Lambda$ are homeomorphic if and only if the corresponding sets of homeomorphism types of visual boundaries of $1$-ended vertex groups are the same.  Our results thus have applications to determining the homeomorphism type of $\partial W_\G$ when $W_\G$ has infinitely many ends.

The work of Tukia~\cite{tukia}, Gabai~\cite{gabai}, and Casson--Jungreis~\cite{casson-jungreis} shows that  the cocompact Fuchsian groups form a rigid quasi-isometry class, i.e., any finitely generated group which is quasi-isometric to a Fuchsian group is cocompact Fuchsian.  
Our last application gives a direct proof of this fact for $2$-dimensional right-angled Coxeter groups.

\begin{thm}\label{thm:qi-rigid}  Any $2$-dimensional right-angled Coxeter group which is quasi-isometric to a cocompact Fuchsian group is cocompact Fuchsian.
\end{thm}

\noindent We give a precise statement as Theorem~\ref{thm:fuchsian-racg} and prove this result in Section~\ref{sec:qi}.  

We now discuss several quasi-isometry invariants that have previously been considered for right-angled Coxeter groups.  In~\cite{dani-thomas}, we investigated \emph{divergence}.  Up to an equivalence relation which identifies polynomials of the same degree, the rate of divergence is a quasi-isometry invariant~\cite{gersten-quadratic}.  We characterised those $W_\G$ with linear and quadratic divergence by properties of their defining graphs and showed that for every positive integer $d$, there is a $W_{\G_d}$ with divergence polynomial of degree $d$.   
Our characterisations were generalised to the non-triangle-free setting by 
Behrstock, Falgas-Ravry, Hagen and Susse~\cite{bfrhs} for linear divergence and by 
Levcovitz~\cite{levcovitz1} for quadratic divergence.  Recently Tran~\cite{tran} used a more refined quasi-isometry invariant suggested by Charney, the \emph{divergence spectrum}, to distinguish the quasi-isometry classes in a family of right-angled Coxeter groups which all have exponential divergence.

\emph{Thickness},  a quasi-isometry invariant related to divergence, was introduced by Behrstock, Dru\c{t}u and Mosher in~\cite{behrstock-drutu-mosher}.  Behrstock--Dru\c{t}u~\cite{behrstock-drutu} show that 
thickness provides a polynomial upper bound on divergence.
In~\cite{behrstock-hagen-sisto-caprace}, Behrstock, Hagen, Sisto and Caprace give an effective characterisation of the right-angled Coxeter groups $W_\G$ (not just the $2$-dimensional ones) which are thick, and show that if 
$W_
\G$ is not thick then it is hyperbolic relative to a collection of thick special subgroups, in which case its divergence is exponential.  As discussed in~\cite{behrstock-hagen-sisto-caprace}, this relative hyperbolicity result combined with results from~\cite{behrstock-drutu-mosher} and~\cite{drutu} can  be used for quasi-isometry classification.  
Recently Levcovitz~\cite{levcovitz2} has introduced a new quasi-isometry invariant for 
2-dimensional right-angled Coxeter groups, the \emph{hypergraph index}, which is closely related to the divergence and thickness of the group, and used it to produce examples of right-angled Coxeter groups for which the orders of thickness and algebraic thickness do not coincide.

By Moussong's Theorem~\cite[Theorem 12.3.3]{davis-book}, right-angled Coxeter groups are CAT(0) groups.  Although the visual boundary is not a quasi-isometry invariant for CAT(0) groups, Charney--Sultan~\cite{charney-sultan} have introduced a boundary for CAT(0) spaces, called the \emph{contracting boundary}, which is a quasi-isometry invariant.  In Section 5 of~\cite{charney-sultan}, the contracting boundary is used to distinguish the quasi-isometry classes of two $1$-ended non-hyperbolic right-angled Coxeter groups.  

Finally, we remark that although every right-angled Artin group is a finite-index subgroup of some  right-angled Coxeter group~\cite{davis-jan}, there exist right-angled Coxeter groups which are not quasi-isometric to any right-angled Artin group.  For example, there are no $1$-ended hyperbolic right-angled Artin groups, since a right-angled Artin group is hyperbolic if and only if it is a free group.  In fact, by considering divergence we see that there are infinitely many quasi-isometry classes of right-angled Coxeter groups which are not the quasi-isometry class of any right-angled Artin group: in~\cite{dani-thomas} we constructed right-angled Coxeter groups with divergence polynomial of any degree, but the divergence of right-angled Artin groups can be only linear or quadratic~\cite{abddy,behrstock-charney}.  
Moreover, Behrstock~\cite{behrstock-counterexample} has recently constructed an example of a right-angled Coxeter group with quadratic divergence which is not quasi-isometric to any right-angled Artin group.
Thus while results on the quasi-isometry classification of right-angled Artin groups (for 
instance~\cite{behrstock-jan-neumann, behrstock-neumann, bestvina-kleiner-sageev-RAAG, huang}) yield some information about right-angled Coxeter groups with certain specific defining graphs, they don't readily lead to a quasi-isometry classification of right-angled Coxeter groups in general.

\subsection*{Acknowledgements}  We thank Jason Behrstock, Murray Elder, Chris Hruska, Misha Kapovich, Mike Mihalik and Kim Ruane for helpful conversations, and the 
anonymous referees for their comments.
We are grateful to Anthony Henderson for helping to support a visit by the first author to the University of Sydney, and to the University of Glasgow for travel support for the second author.  Some of the research for this paper was carried out at the MSRI Summer School in Geometric Group Theory in June 2015, and we thank MSRI for travel support for both authors.

%%%%%%%%%%%%%%%%%%%%%%%
%%%%%%%%%%%%%%%%%%%%%%%
\section{Background}\label{sec:background}
%%%%%%%%%%%%%%%%%%%%%%%
%%%%%%%%%%%%%%%%%%%%%%%

In this section, we recall the Davis complex for $W_\G$ 
(Section~\ref{sec:davis}), a splitting result due to Mihalik and Tschantz~\cite{mihalik-Tschantz} (Section~\ref{sec:splittings}), and the construction and key properties of Bowditch's JSJ tree 
(Section~\ref{sec:bowditch}).  We continue notation from the introduction and assume that the defining graph $\G$ satisfies the Standing Assumptions~\ref{assumptions}.

%%%%%%%%%%%%%%%%%%%%%%%%%%%%%%%%%%%%%%%%%%%%%%%
\subsection{The Davis complex}\label{sec:davis}
%%%%%%%%%%%%%%%%%%%%%%%%%%%%%%%%%%%%%%%%%%%%%%%

A reference for the material is this section is~\cite{davis-book}.

Let $\G'$ be the barycentric subdivision of the defining graph $\G$.  Then the \emph{chamber} $K$ may be defined as the simplicial complex obtained by coning on $\G'$.  Since $\G$ is triangle-free, $K$ is $2$-dimensional.  Denote by $\sigma_0$ the cone point of $K$.  Since all vertices of $\G$ have valence at least two, the boundary of $K$ may be identified with $\G$ (or with $\G'$).  We can then naturally  define the \emph{radial projection} of a path in $K\setminus \{\sigma_0\}$ to $\partial K$, and so obtain a path in $\G$.  Similarly, any point in $K$ can be coned to $\sigma_0$.

For each $s \in S$, the \emph{mirror (of type $s$)} of the chamber $K$ is the subset $K_s$ of $\partial K = \G'$ given by the star of the vertex $s$ in the graph $\G'$.  The mirror $K_s$ is thus the star graph of valence~$n$, where $n = \card \{ t \in S \mid st = ts, t \neq s \} \geq 2$.  Two mirrors $K_s$ and $K_t$ intersect (in a point) if and only if $st = ts$, and the boundary $\partial K$ may be viewed as the union of mirrors $\cup_{s \in S} K_s$.  For each $x \in K$ let $S(x) = \{ s \in S \mid x \in K_s \} \subset S$.
Denote by $W_{S(x)}$ the subgroup of $W$ generated by the set $S(x)$, with $W_{S(x)}$ trivial if $S(x) = \emptyset$.

The \emph{Davis complex} $\Sigma = \Sigma_\G$ is defined to be the following $2$-dimensional  simplicial complex:
\[\Sigma = W \times K / \sim\]
where $(w,x) \sim (w',x')$ if and only if $x = x'$ and $w^{-1}w' \in W_{S(x)}$. 
We define a \emph{chamber} of $\Sigma$ to be any copy of $K$ in $\Sigma$, and a \emph{panel (of type $s$)} in $\Sigma$ to be any copy of the mirror $K_s$.   We write $w\sigma_0$ or just $w$ for the image of the cone point $\sigma_0$ of the chamber $K = (e,K)$ under $w \in W$.  Two chambers in $\Sigma$ are then adjacent along a panel of type $s$ if and only if their cone points $w$ and $w'$ satisfy $w^{-1}w' = s$.  

Since $\G$ is triangle-free, we can now re-cellulate $\Sigma$ so that its vertex set is $W$, its $1$-skeleton is the Cayley graph $\Cayley\G$ of $W$ with respect to the generating set $S$, and all $2$-cells in $\Sigma$ are squares with boundary word $stst = 1$, where $s$ and $t$ are commuting generators.  We call this cellulation of $\Sigma$ the \emph{cellulation by big squares}, with the \emph{big squares} being the $2$-cells.  We also define the \emph{cellulation by small squares} of $\Sigma$ to be the first square subdivision of the cellulation by big squares, with the \emph{small squares} being the squares obtained on subdividing each big square into four.   In the cellulation by small squares, the chamber $K$ is then the ``cubical cone" on the barycentric subdivision $\G'$ of $\G$.  
 
The Davis complex $\Sigma$, with the cellulation by either big or small squares, may now be metrised so that each big square is a unit Euclidean square.  By~\cite[Theorem 12.2.1]{davis-book}, 
this piecewise Euclidean structure is $\CAT(0)$.  In this metrisation, a geodesic $\gamma$ in the 
Cayley graph $\cC = \Cayley\G$ is a geodesic in the Davis complex $\Sigma$ if and only if no two 
successive labels $a$ and $b$ of $\gamma$ are adjacent vertices in $\Gamma$.  Also, each 
induced 
$n$-cycle in $\G$, where $n \geq 5$ since $\G$ is triangle- and square-free, corresponds to a family of 
convex subcomplexes of $\Sigma$ which are planes tiled by big squares so that $n$ big squares 
meet at each vertex.  Since $n \geq 5$, each such subcomplex is quasi-isometric to the hyperbolic 
plane.

\subsection{Splittings of right-angled Coxeter groups over 2-ended subgroups}\label{sec:splittings}

Bowditch's JSJ tree has at least one edge if and only if $G$ admits a splitting over a 2-ended subgroup.  As mentioned in the introduction, Mihalik and Tschantz characterised the Coxeter groups which admit such splittings in~\cite{mihalik-Tschantz}. We now state their result in our setting.

Let $\G$ be a finite simplicial graph which is triangle-free and has no separating vertices or edges, i.e.~$\G$ satisfies (1) and (2) from Standing Assumptions~\ref{assumptions}.  It follows that if $\{a, b\}$ is a cut pair of $\G$, then $a$ and $b$ are non-adjacent, and therefore generate a $2$-ended special subgroup.  For such $\G$, the relevant result from~\cite{mihalik-Tschantz} can be stated as follows:
\begin{thm}\label{thm:mihalik-tschantz}\cite{mihalik-Tschantz}
The right-angled Coxeter group $W_\G$ defined by $\G$ as above 
splits over a $2$-ended subgroup $H$ if and only if there is a cut pair $\{a,b\}$ in $\G$; moreover, there is a unique cut pair $\{a,b\}$ so that some  conjugate of $H$ contains the special subgroup $\langle a, b \rangle$, necessarily with finite index, and $W_\G$ also splits over $\langle a, b\rangle$. 
\end{thm}

\noindent  Thus up to conjugacy and finite index, all splittings of $W_\G$ over $2$-ended subgroups are splittings over special subgroups generated by cut pairs.

%%%%%%%%%%%%%%%%%%%%%%%
\subsection{Bowditch's JSJ tree}\label{sec:bowditch}
%%%%%%%%%%%%%%%%%%%%%%%

Let $G$ be a $1$-ended hyperbolic group which is not cocompact Fuchsian (as defined in 
the introduction). 
For such $G$, Bowditch~\cite{bowditch} uses the structure of local cut points of its boundary 
$M=\partial G$ to define a canonical JSJ tree $\cT = \cT_G$ associated to $G$.  We now recall the construction and key properties of this tree $\cT$ from~\cite{bowditch}.

We begin with some terminology 
from~\cite{bowditch}.  Given $x \in M$, define the \emph{valency} $\val(x)$ of $x$ to be the number of ends of the locally compact space $M\setminus \{x\}$.  A priori $\val(x)\in \N \cup\{\infty\}$,
but Bowditch shows~\cite[Proposition 5.5]{bowditch} that if $M$ is the boundary of a $1$-ended hyperbolic group, then $\val(x)$ is finite for all $x \in M$.  The point~$x$ is a \emph{local cut point} if $\val(x)\geq 2$. 

Now given $n\in \N $, let $M(n)=\{ x \in M \mid \val(x)=n \}$ and 
$M(n+)=\{x\in M \mid \val(x) \geq n\}$.  Bowditch defines relations $\sim$ and $\approx$ on 
$M(2)$ and $M(3+)$ respectively, as in Definitions~\ref{def:sim} and \ref{def:approx} below.  
For $x, y \in M$, let $N(x,y)$ be the number of components of $M \setminus \{x, y\}$. 

\begin{definition}\label{def:sim}[The relation $\sim$]
Given $x, y \in M(2)$, let $x \sim y$ if and only if either $x=y$ or $N(x,y)=2$.  
\end{definition}
The following are some properties of $\sim$.
\begin{enumerate}
\item The relation $\sim$ is an equivalence relation on $M(2)$~\cite[Lemma~3.1]{bowditch}. 
\item Every point of $M(2)$ is in some $\sim$-class by definition.  
\item The $\sim$-equivalence classes are of two types: either they are pairs or they are infinite.  Moreover, the infinite ones are Cantor sets~\cite[Corollary~5.15 and Proposition~5.18]{bowditch}.
 
\end{enumerate}

\begin{definition}\label{def:approx}[The relation $\approx$]
Given $x, y\in M(3+)$, let $x\approx y$ if $x\neq y$ and $N(x, y)=\val(x)=\val(y)\geq 3$.  
\end{definition}
The following are some properties of $\approx$. 
\begin{enumerate}
\item If $x \approx y$ and $x \approx z$ then $y=z$~\cite[Lemma 3.8]{bowditch}.  Thus a priori, a subset of $M(3+)$ is partitioned into pairs of the form $\{x, y\}$ where $x\approx y$.  
\item Lemma 5.12 and Proposition~5.13 of~\cite{bowditch} show that in fact, all of $M(3+)$ is partitioned into 
$\approx$-pairs.  
\end{enumerate}

\begin{remark}\label{rem:allM(2+)}
By the above properties, it is evident that $M(2+)$ is the disjoint union of the $\sim$-classes and 
the $\approx$-pairs in $M$.  
\end{remark}

\subsubsection*{Outline of the construction}

The JSJ tree $\cT$ is produced as follows.  Denote by $T$ the set of of $\sim$-classes and $\approx$-pairs.

\begin{definition}\label{def:between}[Betweenness]
Given three classes $\eta, \theta, \zeta \in T$, the class $\theta $ is \emph{between} $\eta $ and~$\zeta$ if there exist points $x \in \eta$ and $y \in \zeta$, and distinct points $a, b \in \theta$, so that $x$ and $y$ are separated in $M$ by $\{a,b\}$. 
\end{definition}

\begin{remark} \label{rmk:betweenness}
By Lemma 3.18 of \cite{bowditch}, if $\eta$ and $\theta$ are distinct classes in $T$, then for all distinct pairs of points $a,b \in \theta$, the set $\eta$ is contained in a single component of $M \setminus \{a,b\}$.  Hence $\theta$ is between $\eta$ and $\zeta$ if and only if there exist distinct points $a, b \in \theta \subset M$ and distinct components $U$ and $V$ of $M \setminus \{a,b\}$ such that $\eta \subseteq U$ and $\zeta \subseteq V$.  
\end{remark}

Bowditch shows in \cite{bowditch} that the set $T$ forms a \emph{pretree} (i.e.~a set with the betweenness condition satisfying certain axioms) which is \emph{discrete} (i.e.~intervals between points are finite).  In \cite{bowditch-treelike}, Bowditch proves that every discrete pretree can be embedded in a \emph{discrete median pretree} by adding in all possible \emph{stars} of size at least 3.  

\begin{definition}\label{def:star}[Stars]
A \emph{star} is a subset $X$ of $T$ with the property that no element of $T$ is between any pair of elements in $X$, and which is maximal with respect to this property.  Thus if $\eta \in T \setminus X$, then there exist $\zeta\in X$ and $\theta \in T$ such that $\theta$ is between $\eta$ and $\zeta$.  The \emph{size} of a star is its cardinality.  
\end{definition}

\noindent Any discrete median pretree can be realised as the vertex set of a simplicial tree \cite{bowditch-treelike}. The simplicial tree obtained in this way from~$T$ is essentially the JSJ tree $\cT$, although certain additional vertices are added at the midpoints of some of the edges in order to get a cleaner statement about stabilisers.

\subsubsection*{The JSJ tree}

The properties of this tree $\cT$ are summarised in Theorems~0.1 and 5.28 of~\cite{bowditch} and are explained in more detail in Sections~3 and~5 of that paper. 
 The group $G$ acts minimally, simplicially and without edge inversions on $\cT$ with finite quotient graph.  

\subsubsection{Vertices}
The tree $\cT$ has vertices of three types: $V_1(\cT), V_2(\cT)$ and $V_3(\cT)$.
\begin{enumerate}
\item \emph{The vertices $V_1(\cT)$.} The set of Type 1 vertices consists of: 
\begin{itemize}
\item the 
$\sim$-classes in $M(2)$ consisting of exactly two elements; \item 
all the $\approx$-pairs; and \item the extra vertices mentioned at the end of the previous section, added at the midpoints of any edges between stars of size at least 3 and infinite~$\sim$-classes.  
\end{itemize}

For all $v \in V_1(\cT)$, the stabiliser $G(v)$ of~$v$ is a maximal $2$-ended subgroup of $G$.  If~$v$ comes from a $\sim$- or a $\approx$-pair $\{x,y\}$, its degree in $\cT$ is equal to $N(x,y) < \infty$.  The degree of the added vertices is $2$ since they subdivide edges.   

\medskip

\item \emph{The vertices $V_2(\cT)$.} The set of Type 2 vertices is in bijective correspondence with the collection of infinite $\sim$-classes in $M(2)$.  

The stabiliser $G(v)$ of a vertex $v \in V_2(\cT)$ is a \emph{maximal hanging Fuchsian (MHF)} subgroup of $G$.  
This means that there is a properly discontinuous action  of $G(v)$ on the hyperbolic plane~$\H^2$, without parabolics, such that the quotient $\H^2/G(v)$ is non-compact and further, there is an  
 equivariant homeomorphism from the $\sim$-class corresponding to $v$ onto the limit set of the $G$-action in $\partial \H^2$.  
The degree of a vertex in $V_2(\cT)$ is infinite. 

\medskip

\item \emph{The vertices $V_3(\cT)$.}  
The set of Type 3 vertices consists of the added stars of size at least 3 mentioned in the outline of the construction. 
The stabiliser of a vertex of this type is an infinite non-elementary group which is not a maximal hanging Fuchsian subgroup~\cite[Lemma 5.27]{bowditch}.
The degree of a vertex of this type is infinite.   From the description of the edges below, one sees that the elements in the star are in bijection with the edges incident to the corresponding vertex.  In particular, stars of size at least 3 are infinite. 
\end{enumerate}

\subsubsection{Edges}\label{sec:edges}
Given $v_1 \in V_1(\cT)$ coming from an $\approx$-pair $\zeta$, and
$v_2\in V_2(\cT)$ coming from an infinite $\sim$-class $\theta$, there is an edge between $v_1$ and $v_2$ if $\zeta \subset \bar \theta$.
Every edge of the pretree $T$ is of this form. 
The tree $\mathcal T$ has the following additional edges.  Given $v_3 \in V_3(\cT)$, there is an edge between $v_3$ and each vertex of $V_1(\cT) \cup V_2(\cT)$ in the star corresponding to $v_3$.  Finally, if there are edges between vertices of $V_3(\cT)$ and 
$V_2(\cT)$, they are subdivided by adding a valence 2 vertex in the middle.  It is evident from this construction that the two endpoints of an edge of $\cT$ are never of the same type, and that 
$\sim$-pair vertices (which are in $V_1(\cT)$) are necessarily connected to vertices in $V_3(\cT)$. 
The stabiliser of every edge is a 2-ended subgroup, and is the intersection of the stabilisers of the vertex groups incident to this edge.

\begin{remark}\label{rem:jsj-qi}(Quasi-isometry invariance)  Let $G$ and $H$ be $1$-ended hyperbolic groups which are not cocompact Fuchsian.  By the construction above if $G$ and $H$ are quasi-isometric, then since $\partial G$ and $\partial H$ are homeomorphic, the JSJ trees $\cT_G$ and  $\cT_H$ are equal as coloured trees.  That is, the JSJ tree is a quasi-isometry invariant.   
\end{remark}

%%%%%%%%%%%%%%%%%%%%%%%
%%%%%%%%%%%%%%%%%%%%%%%
\section{JSJ tree for certain right-angled Coxeter groups}\label{sec:JSJ}
%%%%%%%%%%%%%%%%%%%%%%%
%%%%%%%%%%%%%%%%%%%%%%%

In this section, we prove our main result, Theorem~\ref{thm:main}, by giving an explicit construction of the Bowditch JSJ tree $\cT = \cT_\G$ associated to
a right-angled Coxeter group $W=W_\G$ such that  $\G$ satisfies our Standing 
Assumptions~\ref{assumptions}. 
That is, we show that for each class of vertices in $\cT$, the $W$-orbits of this class are in bijection with subsets of vertices of $\G$ which satisfy certain explicit graph-theoretic properties.

Throughout this section, we continue the notation of Section~\ref{sec:background}, and we assume that $\G$ satisfies the Standing Assumptions~\ref{assumptions}.  We begin with some conventions and definitions in Section~\ref{sec:assumptions} and preliminary results in Section~\ref{sec:separation}. 
In Section~\ref{sec:approx}, we identify the $\approx$-pairs and their stabilisers. 
We construct certain $\sim$-classes in Section~\ref{sec:sim}.  In Section~\ref{sec:type 1 and 2} we show that we have identified all $\sim$-classes, and so in fact constructed all the vertices of the pre-tree $T$.  We also determine the stabilisers of $\sim$-classes. 
 In Section~\ref{sec:stars} we construct certain stars, and in Section~\ref{sec:all-stars} prove that we have identified all the stars and determine their stabilisers.  
 Finally, 
we summarise 
the description of $\mathcal T$ in terms of $\G$ in Theorem~\ref{thm:JSJ} of 
Section~\ref{sec:summary}.

%%%%%%%%%%%%%%%%%%%%%%%
\subsection{Conventions and definitions}\label{sec:assumptions}
%%%%%%%%%%%%%%%%%%%%%%%
Throughout Section~\ref{sec:JSJ}, 
we denote the Davis complex of~$W$ by $\Sigma$. 
Recall that the 1-skeleton of the cellulation of $\Sigma$ by big squares can be identified with the Cayley graph of $W$.  We use $\cC$ to denote this Cayley graph, together with a choice of base point in $\Sigma$, to be labelled by the identity $e\in W$.  
Now we identify $\partial W = \partial \Sigma$ with the set of geodesic rays in $\cC$ emanating 
from~$e$.  However we also freely think of endpoints of bi-infinite geodesics in~$\cC$ not passing through $e$ as points of $\partial W$.  Given any bi-infinite geodesic $\gamma$ in $\cC$, we denote its endpoints in $\partial W$ by $\gamma^+$ and $\gamma^-$, and put $\partial \gamma = \{ \gamma^+, \gamma^- \}$.

Recall the definition of a \emph{cut pair}  from the introduction.  A vertex of $\G$ is \emph{essential} if it has valence at least three, and a cut pair is \emph{essential} if its two vertices are essential.  A \emph{reduced path} is a path in $\G$ or in $\cC$ which does not contain any loops or any backtracking. (Given any path in $\G$ or in $\cC$, if such loops or backtracking exist, they can be cut out, eventually leading to a reduced path.)  Note that a reduced path does not have to be a shortest path.  
Given a path $\tau$ in $\G$ or $\cC$, the subpath between two vertices $p$ and $q$ on $\tau$ is denoted by $\tau_{[p,q]}$.

A geodesic $\gamma$ in $\cC$ is \emph{bicoloured (by $a$ and $b$)} if it is labelled alternately by a pair of (nonadjacent) vertices $a, b$ of $\G$.  If a geodesic $\gamma$ is bicoloured by $a$ and $b$ we may say that it is \emph{$(a,b)$-bicoloured}.

%%%%%%%%%%%%%%%%%%%%%%%%%%%%%%%%%%%%%%%%%%%%%%%%
\subsection{Preliminary results}\label{sec:separation}
%%%%%%%%%%%%%%%%%%%%%%%%%%%%%%%%%%%%%%%%%%%%%%%%

In this section we show that endpoints of bicoloured geodesics are distinguished elements of $\partial W$.  We also establish two results which relate separation properties of subsets of vertices in $\G$, of subcomplexes of $\Sigma$, and of subsets of $\partial W$.  The proofs in this section use some standard facts about reduced words in (right-angled) Coxeter groups, from~\cite{davis-book}.

For the proofs of Lemma~\ref{lem:parallel} and Corollary~\ref{cor:parallel} below, we do not need all of the Standing Assumptions~\ref{assumptions}.  Indeed, in 
Lemma~\ref{lem:parallel} the conclusion that $\eta$ is eventually $(a,b)$-bicoloured holds for all right-angled Coxeter groups, and the entire conclusion of Lemma~\ref{lem:parallel} and of Corollary~\ref{cor:parallel} holds for all hyperbolic right-angled Coxeter groups with square-free defining graphs.

\begin{lemma}\label{lem:parallel}  Let $\gamma$ be an $(a,b)$-bicoloured geodesic in $\cC$, and let $\eta$ be  any other geodesic in $\cC$.  If $\gamma^+ = \eta^+$, then in the direction of $\eta^+$, the geodesic $\eta$ is eventually $(a,b)$-bicoloured, and the geodesic $\eta$ eventually coincides with either $\gamma$ or $\gamma c$, with the latter case occurring only if there is a (unique) vertex $c$ of $\G \setminus \{a,b\}$ which is adjacent to both $a$ and $b$.
\end{lemma}

\begin{proof} 
Let $\mu$ be a shortest path in $\cC$ from $\gamma$ to $\eta$, and let $w$ be the (possibly empty) word labelling $\mu$.  Then $\mu$ meets $\gamma$ only at its starting point.  We may assume, without loss of generality, that this point is $e$ (so in particular, $\gamma$ passes through $e$). 
Write $|w|$ for the length of $w$.

Since $\gamma^+ = \eta^+$, there is a $D > 0$ so that for all $n \geq 1$, there is a shortest path $\mu_n$ in $\cC$ from the vertex $(ab)^n$ of $\gamma$ (without loss of generality) to some vertex of $\eta$, such that $\mu_n$ has length at most $D$.  Let $w_n$ be the word labelling the path $\mu_n$.  Now only the endpoints of the paths $\mu$ and $\mu_n$ lie on $\eta$, so if $w'_n$ is the word labelling the subpath of $\eta$ from the endpoint of $\mu$ to the endpoint of $\mu_n$, the concatenation $w w'_n$ is reduced.

By construction, the words $ww_n'$ and $(ab)^n w_n$ are the labels of paths in $\cC$ from $e$ to the same vertex (the endpoint of $\mu_n$).  Hence $ww_n' = (ab)^n w_n$ in $W$.  The right-hand concatenation $(ab)^n w_n$ might not be reduced, but we have $(ab)^n w_n = z_n \overline{w_n}$, where $z_n$ is an initial subword of $(ab)^n$, $\overline{w_n}$ is a subword of $w_n$, and $z_n\overline{w_n}$ is reduced. 
Thus $ww_n'$ and $z_n\overline{w_n}$ are reduced words representing the same element of $W$.

We first claim that  in the direction of $\eta^+$, infinitely many edge-labels of $\eta$ are drawn from the set $\{a,b\}$. 
If not, then each word $w'_n$ contains at most $N$ instances of $a$ or $b$, say.  Hence for all $n$, the number of instances of $a$ or $b$ in the reduced word $ww'_n$ is bounded above by $|w| + N$.  However for $n$ large enough,  there are more than $|w|+N$ instances of $a$ and $b$ in $z_n$.  Any two reduced words for the same element of a Coxeter group have the same letters, so this is a contradiction.  This proves the claim. 

Now suppose that $\eta$ in the direction of $\eta^+$ is not eventually $(a,b)$-bicoloured.  Then in the direction of $\eta^+$, starting from the endpoint of $\mu$, the labels on $\eta$ are $g_1 h_1 g_2 h_2 \dots$, where each $g_i$ is a reduced word in $a$ and $b$, with $g_i$ nonempty for $i \geq 2$, and each $h_i$ is a nonempty reduced word in $\G \setminus \{a,b\}$.  For $n$ large enough, the word $w'_n$ contains $g_1h_1\dots g_{D+1}h_{D+1}$, and so the reduced word $ww'_n$ contains more than $D$ letters in $\Gamma \setminus \{a,b\}$.  However the reduced word $z_n \overline{w_n}$ contains at most $D$ letters in $\Gamma \setminus \{a,b\}$, and again we obtain a contradiction.

We have shown that in the direction of $\eta^+$, the geodesic $\eta$ is eventually $(a,b)$-bicoloured.  Assume that in this direction, the geodesics $\eta$ and $\gamma$ do not eventually coincide.  Then for all  large enough $n$, the path $\mu_n$ is not contained in the geodesic $\gamma$, and this path ends in the $(a,b)$-bicoloured subray of $\eta$ in the direction of $\eta^+$.  Thus the word $\overline w_n$ is nonempty and its first and last letters are not $a$ or $b$.  Fix such a large enough $n$, and choose $m > n$ so that $w'_m$ has $w'_n$ as an initial subword.  Then the word $w'_{n,m}$ obtained by cancelling $w'_n$ from the start of $w'_m$ is a word in $a$ and $b$, and for large enough $m$, we may assume that $w'_{n,m}$ contains both $a$ and $b$.  Similarly, if $z_{n,m}$ denotes the word obtained by cancelling $z_n$ from the start of $z_m$, then $z_{n,m}$ is a word in $a$ and $b$ which we may assume contains both $a$ and $b$.  Denote by $g_n$ the vertex of $\gamma$ at which the subpath of $\mu_n$ labelled by $\overline w_n$ begins.   By considering paths in $\cC$ from $g_n$ to the endpoint of $\mu_m$, we obtain that $\overline w_n w'_{n,m} = z_{n,m} \overline w_m$ in $W$, and both of these products are reduced.

Now we use the fact that any two reduced words representing the same element in a right-angled Coxeter group are related by a sequence of moves replacing a subword $st$ by a subword $ts$, where $s$ and $t$ are commuting generators.  Since both of the words $w'_{n,m}$ and $z_{n,m}$ contain both $a$ and $b$, it follows that every letter in $\overline w_n$ commutes with both $a$ and $b$, or is equal to $a$ or $b$.   The graph $\G$ is square-free so there is at most one vertex $c$ of $\G \setminus \{a,b\}$ which commutes with both $a$ and $b$.  As $\overline w_n$ is nonempty and reduced, and cannot start or end with $a$ or $b$, we deduce that $\overline w_n = c$.  Note that this holds for all large enough $n$.

To complete the proof, observe that for all large enough $n$, the geodesic $\gamma c$ is the unique $(a,b)$-bicoloured geodesic passing through the vertex $g_n c$.  Since the vertices $g_n$ approach $\eta^+ = \gamma^+$, and $\overline w_n = c$ for all large enough $n$, it follows that the geodesic $\eta$ in the direction of $\eta^+$ eventually coincides with $\gamma c$, as required.
\end{proof}

\begin{cor}\label{cor:parallel}  Let $\gamma$ be an $(a,b)$-bicoloured geodesic in $\cC$, and let $\gamma'$ be an $(a',b')$-bicoloured geodesic in $\cC$.  If $\partial \gamma = \partial \gamma'$ then $\{a,b\} = \{a',b'\}$, and either $\gamma = \gamma'$ or there is a (unique) vertex $c$ of $\G \setminus \{a,b\}$ which is adjacent to both $a$ and $b$, and $\gamma' = \gamma c$.
\end{cor}

The next lemma produces separating pairs of points in $\partial W$ from separating pairs of vertices in $\G$.  Note that since $\G$ is square-free, given a cut pair $\{a,b\}$ in $\G$ there is at most one component of $\G \setminus \{a,b\}$ consisting of a single vertex $c$ which commutes with both $a$ and~$b$.   For the remainder of this section, if $\Upsilon$ is a subcomplex of $\Sigma$, we denote by $N(\Upsilon)$ the union of the chambers which have nontrivial intersection with $\Upsilon$.

\begin{lemma}\label{lem:separation}
Let $\{a, b\}$ be a cut pair in $\G$ and let $\gamma$ be the (unique) $(a,b)$-bicoloured geodesic passing through the identity.  Suppose that $\G \setminus \{a, b\}$ has $k$ components (with $k \geq 2$). 
\begin{enumerate}

\item If no component of $\G \setminus \{ a, b \}$ consists of a single vertex, then $\Sigma \setminus \gamma$, and consequently $\partial W \setminus \partial \gamma $, has exactly $k$ components.  

More precisely, if 
$\G \setminus \{a, b\}$ has $k$ components $\Lambda_1, \dots, \Lambda_k$,
 none of which is a single vertex, 
then $\Sigma \setminus \gamma$ has $k$ components $Y_1,\dots,Y_k$.  
For $1 \leq i \leq k$ the component $Y_i$ contains exactly the Cayley graph vertices which can be written 
in the form $w_{ab} \lambda_i v$, where $w_{ab}$ is a reduced word in $a$ and $b$, 
$\lambda_i$ is a generator corresponding to a vertex of $\Lambda_i$, and $v$ is any 
element of $W$ so that the word $w_{ab}\lambda_i v$ is reduced.

\item If $\G \setminus \{a, b\}$ has a component consisting of a single vertex $c$, let $\gamma'$ be the geodesic $\gamma' = \gamma c = c\gamma$ and let $\Upsilon$ be the subcomplex of $\Sigma$ bounded by $\gamma$ and~$\gamma'$.  Then $\Sigma \setminus \Upsilon$, and consequently $\partial W \setminus \partial \gamma$, has exactly $2(k-1)$ components.

More precisely, if 
$\G \setminus \{a, b\}$ has $k$ components $\Lambda_1, \dots, \Lambda_k$, with $\Lambda_1 = \{c\}$, then $\Sigma \setminus \Upsilon$ has $2(k-1)$ components $Y_2,\dots,Y_k,Y_2',\dots,Y_k'$.    
For $2 \leq i \leq k$ the component $Y_i$ (respectively, $Y_i'$)  contains exactly the Cayley graph vertices which can be written 
in the form
$w_{ab} \lambda_i v$ (respectively, $w_{ab} c \lambda_i v$)
where $w_{ab}$ is a reduced word in $a$ and $b$, $\lambda_i$ is a generator corresponding to a vertex of $\Lambda_i$, and $v$ is any  
element of $W$ so that the word $w_{ab}\lambda_i v$ (respectively, $w_{ab} c \lambda_i v$) is reduced.
\end{enumerate}

\end{lemma}

\begin{proof}
Observe first that in both cases, since $a$ and $b$ are nonadjacent vertices in $\Gamma$, the Cayley graph geodesic $\gamma$ is also a geodesic in $\Sigma$.  

To prove (1), assume that no component of $\G \setminus \{a, b\}$ consists of a single vertex.  Recall from Section~\ref{sec:davis} that we may identify the boundary of each chamber of $\Sigma$ 
with the barycentric subdivision $\G'$ of $\G$, and that $\G'$ may be viewed as the union of all mirrors.  Now the complement of the mirrors $K_a$ and $K_b$ in $\G'$  has exactly $k$ components, with the $i$th component being the union of the mirrors of type a vertex of $\Lambda_i$.  The (infinite) union of chambers $N(\gamma)$ is obtained by gluing together chambers along panels of types $a$ and $b$, and this gluing is type-preserving.  Hence the  complement of $\gamma$ in  $N(\gamma)$ has exactly $k$ components, $N_1, \dots, N_k$, so that the boundary of $N_i$ is the union of $\gamma$ together with all panels in the boundary of $N(\gamma)$ with type a vertex  of $\Lambda_i$.  

We now define subsets $Y_1, \dots, Y_k$ of $\Sigma\setminus \gamma$  as follows.  Given a point $x\in \Sigma \setminus \gamma$ let $\pi(x)$ denote its closest-point projection to the geodesic $\gamma$.  There is then a unique geodesic from~$x$ to $\pi(x)$, and this geodesic passes through exactly one of $N_1, \dots, N_k$.  Define $x\in Y_i$ if this geodesic passes through 
$N_i$.  The uniqueness of the projection geodesic implies that the $Y_i$ are well-defined and disjoint.  It is clear that each $Y_i$ is connected and that $\cup_{1\le i \le k} Y_i$ is $\Sigma \setminus \gamma$.  Thus $\Sigma \setminus \gamma$ has $k$ components.  By similar arguments to those in the proof of Lemmas~2.1 and 2.3 of Lafont~\cite{lafont-3d}, it follows that $\partial \gamma$ separates $\partial W$ into exactly $k$ components as well.  

Now, given a Cayley graph vertex $w$ in $Y_i$, we construct a Cayley graph path from the identity to $w$ as follows.  Consider the geodesic $\eta$ in $\Sigma$ from $w$ to the projection $\pi(w) \in \gamma$.  We may choose a sequence of  chambers $K_1,\dots,K_n$ intersecting $\eta$ so that consecutive chambers in this sequence meet along a panel and the union $\cup_{i=1}^n K_i$ 
contains $\eta$.  Since chambers are convex in $\Sigma$, we may assume that the chambers in this 
sequence are pairwise distinct.  Also, if two chambers $K_j$ and $K_{j'}$ with $j < j'$ meet along a 
panel, then $K_j \cup K_{j'}$ is convex, so we may assume that this occurs only if $j' = j+1$.  Let the  
centres of these chambers be $w_1 = w,\ldots,w_n$ respectively.  Then $w_i$ and $w_{i+1}$ are 
adjacent vertices of $\cC$, so we may approximate the geodesic in $\Sigma$ from $w$ to $\pi(w)$ by 
the path in the Cayley graph with vertex set $w_1,\dots,w_n$.  The $w_1,\dots,w_n$ are pairwise 
distinct, since the corresponding chambers are pairwise distinct, and the label on this path does not 
contain any subword $sts$ where $s$ and $t$ are commuting generators, since two chambers in this 
sequence meet along a panel only if they are consecutive.  Hence this path in the Cayley graph is a 
geodesic in the Cayley graph.

Now concatenate the reverse of this geodesic with a Cayley graph geodesic from the identity to $w_n$.  The
 geodesic from $e$ to $w_n$ in $\cC$ is labelled by a reduced word of the form $w_{ab}$ as in the statement of the lemma.  The geodesic from $w_n$ to $w$ leaves $N(\gamma)$ through a panel of type a vertex $\lambda_i \in \Lambda_i$.  Thus the concatenated path is labelled by a word of the form $w_{ab} \lambda_i v$ as in the statement.  It is clear that if a Cayley graph vertex $w$ can be written as a reduced word $w_{ab} \lambda_i v$, then $w$ is in $Y_i$.  Thus $Y_i$ consists of exactly the vertices described.

In case (2), since $\gamma'$ is, like $\gamma$, bicoloured by $a$ and $b$,
both $\gamma$ and $\gamma'$ are geodesics in $\Sigma$.  Moreover the subcomplex $\Upsilon$ is 
convex in $\Sigma$. (The subcomplex $\Upsilon$ consists of a band of big squares bounded by $\gamma$ and $
\gamma'$, and is isometric to $\R \times [0,1]$.)  
By a similar argument to that in case (1), we see that $N(\Upsilon) \setminus \Upsilon$ 
has  exactly $2(k-1)$ components, $N_2, \dots, N_k,N_2',\dots,N'_k$, 
so that the boundary of $N_i$ (respectively $N_i'$) is 
the union of $\gamma$ (respectively $\gamma'$) together with all panels in the boundary of $N(\Upsilon)$ with type a vertex  of 
$\Lambda_i$.

We define subsets  $Y_2, \dots, Y_k, Y_2', \dots, Y_k'$ of $\Sigma\setminus \Upsilon$  as follows.  
Given a point $x\in \Sigma \setminus \Upsilon$ let $\pi(x)$ denote its closest-point projection to $\Upsilon$.  There is a unique geodesic from $x$ to $\pi(x)$, which passes through exactly one of $N_2, \dots, N_{k},N_2',\dots,N_{k}'$.  Define $x\in Y_i$ if this geodesic 
passes through 
$N_i$ (and so $\pi(x) \in \gamma$) and $x \in Y_i'$ if it passes through $N_i'$ (and so $
\pi(x) \in \gamma'$).  An argument similar to the proof of (1) implies that 
$\partial\gamma$ separates $\partial W$ into exactly $2(k-1)$ components, and the description of these  components is also similar to case (1).  
\end{proof}

The next result considers the separation properties of geodesics in the Cayley graph $\cC$.  Since an arbitrary geodesic $\gamma$ in $\cC$ need not be a geodesic in $\Sigma$, we cannot use projections to $\gamma$ in the following proof as we did in Lemma~\ref{lem:separation} above.  Similarly, $N(\gamma)$ need not be convex, so we cannot use projections to $N(\gamma)$.  Instead, we use paths in $\cC$ to approximate projections.

\begin{lemma}\label{lem:gamma separates}  Let $\gamma$ be a geodesic in  $\cC$.  If $\gamma$ separates $N(\gamma)$  then $\gamma $ separates~$\Sigma$.  \end{lemma}

\begin{proof}  Suppose that $\Sigma \setminus \gamma$ is connected, and let $p$ and $q$ be distinct points in $N(\gamma) \setminus \gamma$.  We will show that $N(\gamma) \setminus \gamma$ is connected by constructing a path from $p$ to $q$  in $N(\gamma)\setminus \gamma$.  

Since $\Sigma \setminus \gamma$ is connected, there is a path $\eta$ in $\Sigma \setminus \gamma$ from $p$ to $q$.  Denote by $p'$ the first point of intersection of $\eta$ with the boundary of $N(\gamma)$ and by $q'$ the last point of intersection of $\eta$ with the boundary of $N(\gamma)$.  We may assume without loss of generality that the interiors of $\eta_{[p,p']}$ and $\eta_{[q',q]}$ lie in $N(\gamma) \setminus \gamma$ and that the interior of $\eta_{[p',q']}$ lies in $\Sigma \setminus N(\gamma)$, and that these subpaths contain no loops or backtracking.  

Let $K_1, \dots, K_n$ be a sequence of pairwise adjacent chambers in the closure of $\Sigma \setminus 
N(\gamma)$ through which the interior of $\eta_{[p',q']}$ passes, so that $p'$ is in a panel shared by 
$K_1$ and some chamber in $N(\gamma)$ and $q'$ is in a panel shared by $K_n$ and some 
chamber in $N(\gamma)$.    Then similarly to the proof of Lemma~\ref{lem:separation}, we may approximate $\eta_{[p',q']}$ by a reduced path $\tau$ in the Cayley graph $\cC$ 
with vertex set $w_1,\dots,w_n$, where $w_1,\dots,w_n$ are the centres of the chambers $K_1,\dots,K_n$ 
respectively.  We construct another reduced path $\sigma$ in $\cC$ from $w_1$ to $w_n$ 
as follows.  Let $u$ be the vertex of $\gamma$ which is adjacent in $\cC$ to $w_1$, so that the 
midpoint of $u$ and $w_1$ is the centre of a panel of $K_1$ which contains $p'$, and let $v$ be the 
vertex of $\gamma$ which is adjacent in $\cC$ to $w_n$, so that the midpoint of $v$ and $w_n$ is 
the centre of a panel of $K_n$ containing $q'$.  Then define $\sigma$ to be the concatenation of the 
edge in $\cC$ between $w_1$ and $u$, the segment of $\gamma$ between $u$ and $v$, and the 
edge in $\cC$ between $v$ and $w_n$.  
Note that since $\gamma$ is a Cayley graph geodesic, the path $\sigma$ 
is reduced.

Now we have that $\tau$ and $\sigma$ are two distinct reduced paths in $\cC$ from $w_1$ to $w_n$, which intersect only at $w_1$ and $w_n$.  So there is a filling by big squares of the loop in $\cC$ obtained by concatenating $\tau$ with $\sigma$.  Denote by $S_1,\dots,S_N$ the pairwise adjacent big squares in this filling which have at least one edge or vertex on $\gamma_{[u,v]}$, so that $w_1$ and $u$ are adjacent vertices of $S_1$ and $w_n$ and $v$ are adjacent vertices of $S_N$.  Then the boundary of $\cup_{1 \leq i \leq N} S_i$ is the union of $\sigma$ with a path $\sigma'$ in $\cC$ from $w_1$ to $w_n$ so that $\sigma'$ is entirely contained in chambers in the closure of $\Sigma \setminus N(\gamma)$ which share a panel or a vertex with $\partial N(\gamma)$.  Now by concatenating the intersections of the squares $S_1,\dots,S_N$ with these panels we obtain a path in the intersection of $\Sigma$ with the boundary of $N(\gamma)$ which goes from the midpoint of $u$ and $w_1$ to the midpoint of $v$ and $w_n$.  This path may be extended and/or restricted within panels to obtain a path from $p'$ to $q'$ which lies in $N(\gamma) \setminus \gamma$.  Then by concatenating $\eta_{[p,p']}$, this path and $\eta_{[q',q]}$, we obtain a path from $p$ to $q$ in $N(\gamma) \setminus \gamma$.  Hence $N(\gamma) \setminus \gamma$ is connected.
\end{proof}

%%%%%%%%%%%%%%%%%%%%%%%%%%%%%%%
\subsection{Identification of the $\approx$-pairs and their stabilisers}
\label{sec:approx}
%%%%%%%%%%%%%%%%%%%%%%%%%%%%%%%

We now begin the construction of the JSJ tree $\cT$, and the determination of stabilisers of vertices in $\cT$.  In Remark~\ref{rem:approx}, we discuss how the results of this section may be generalised to higher dimensions.

Recall that 
one class of finite-valence vertices in $\cT$ is given by $\approx$-pairs, which are
pairs of points in $\partial W$ that separate $\partial W$ into at least three components. 
The following immediate corollary of Lemma~\ref{lem:separation} yields a collection of $\approx$-pairs.  See Figure~\ref{fig:approx-pairs} for an example.

\begin{cor}[$\approx$-pairs]\label{cor:approx}
Let $\{a, b\}$ be a pair of essential vertices in $\G$ such that $\G \setminus \{a,b\}$ has $k \ge 3$ components.   Then $\{a, b\}$ corresponds to a $W$-orbit of $\approx$-pairs in $\cT$.  Each $\approx$-pair in this orbit consists of the endpoints of a geodesic bicoloured by $a$ and $b$, and
yields a vertex of Type 1 in $\cT$ of valence:
\begin{enumerate}
\item $k \geq 3$ if no component of $\Gamma \setminus \{a,b\}$ consists of a single vertex; and
\item $2(k-1) \geq 4$ if $\Gamma \setminus \{a, b\}$ has a component consisting of a single vertex.
\end{enumerate}
\end{cor}

\smallskip

\begin{figure}[ht]
\begin{center}
\begin{overpic}{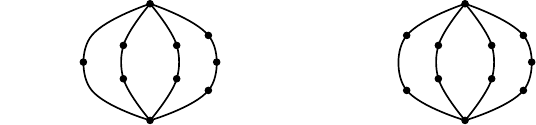}
\put(27,25){\footnotesize $a$}
\put(86,25){\footnotesize $a$}
\put(27,-4){\footnotesize $b$}
\put(86,-4){\footnotesize $b$}
\end{overpic}
\caption{{\footnotesize In both graphs, $\{a,b\}$ gives an orbit of $\approx$-pairs.  On the left the corresponding Type 1 vertices of $\cT$ have valence $6=2(4-1)$ and on the right, they have valence~$4$.}} 
\label{fig:approx-pairs}
\end{center}
\end{figure}
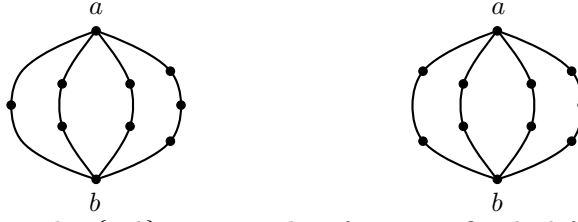

We now show that every $\approx$-pair vertex of $\cT$ arises as in 
Corollary~\ref{cor:approx}.

\begin{lemma} \label{lem:all approx}
If $\xi$ and $\xi'$ are points of $\partial W$ such that $\xi \approx \xi'$, then there is an essential cut pair $\{a, b\}$ of $\G$ such that $\G \setminus \{a, b\}$ has at least three components, 
and a geodesic $\gamma$  bicoloured by $a$ and $b$ such that 
$\partial \gamma = \{\xi, \xi'\}$. 
\end{lemma}

\begin{proof}
By Bowditch's construction, the pair $\{\xi, \xi'\}$ corresponds to a vertex $v$ of Type 1 of the JSJ tree for $W$, 
the stabiliser $H$ of $v$ is $2$-ended, and $W$ splits over $H$. Now by Theorem~\ref{thm:mihalik-tschantz}, 
there exists a unique cut pair $\{a, b\}$ such that some conjugate of $\langle a, b \rangle$ is a finite-index subgroup of $H$.   It follows that there is a geodesic 
$\gamma$ bicoloured by $a$ and $b$ such that $\partial \gamma = \{\xi, \xi'\}$.  Suppose 
$\G \setminus \{a, b\}$ had exactly two components.  Then by applying the appropriate case of 
Lemma~\ref{lem:separation}, we see that $\partial W \setminus \{\xi, \xi'\}$  also has two components, which contradicts the assumption that 
$\xi \approx \xi'$.  
Thus $\G \setminus \{a, b \}$ has at least three components.  

We claim that $a$ and $b$ are both essential.   By Standing Assumptions~\ref{assumptions}(2), $\G$ has no valence one vertices.  Suppose one of $a$ and $b$, say $a$, has valence two, and let $\Lambda$ be a component of $\G \setminus \{a, b\}$ which does not contain either of the two edges attached to $a$.  Then $a$ and $\Lambda$ are necessarily in different components of $\G \setminus\{ b\}$, which contradicts Standing Assumptions~\ref{assumptions}(2).  
This proves the claim.
 Thus $\{a, b \}$ satisfies  the hypothesis of Corollary~\ref{cor:approx}.  Therefore the $\approx$-pair $\{\xi, \xi'\}$ comes from the 
construction in Corollary~\ref{cor:approx}.
\end{proof}

We next determine the stabilisers of $\approx$-pairs.

\begin{lemma}\label{lem:approx stabiliser} Let $\xi$ and $\xi'$ be points of $\partial W$ such that $\xi \approx \xi'$, and let $\gamma$ be an $(a,b)$-bicoloured geodesic so that $\partial \gamma = \{\xi,\xi'\}$ (as guaranteed by Lemma~\ref{lem:all approx}).  Then the stabiliser of the $\approx$-pair $\{\xi,\xi'\}$ is a conjugate of either $\langle a,b\rangle$, if there is no vertex $c$ of $\G$ adjacent to both $a$ and $b$, or of $\langle a, b, c\rangle$, if there is such a vertex $c$.
\end{lemma}

\begin{proof} The action of $W$ on $\Sigma$ preserves the labels of edges in $\cC$, so that if $g \in W$ then $g\gamma$ is also an $(a,b)$-bicoloured geodesic.  Conjugates of $\langle a,b\rangle$ obviously stabilise endpoints of $(a, b)$-bicoloured geodesics, and by Corollary~\ref{cor:parallel} the only other elements of $W$ which can stabilise these endpoints are those in conjugates of the subgroup $\langle a, b, c \rangle$, in the cases where there exists a (unique) vertex $c$ which commutes with both $a$ and $b$.  
\end{proof}

The following corollary of the results in this section will be used in the next section.

\begin{cor}\label{cor:recognise approx}  Let $A$ be a set of vertices of $\G$ and let $\cC_A$ be a copy of the Cayley graph of $\langle A \rangle$ in $\cC$. Then $\partial \cC_A$ is an $\approx$-pair if and only if the following holds: there is an essential cut pair $\{a,b\}$ in $\G$ so that $\G \setminus \{a,b\}$ has at least three  components,  and $A = \{a,b\}$ or $A = \{a,b,c\}$, with the latter occurring only if there is a vertex $c$ of $\G$ adjacent to both $a$ and~$b$.
\end{cor}

We finish this section by discussing the generalisation of our results on $\approx$-pairs to higher dimensions, that is, where the defining  graph $\G$ is not assumed to be triangle-free.

\begin{remark}\label{rem:approx}  Let $\G$ be a finite, simplicial graph which is not a complete graph, and write~$K_n$ for the complete graph on $n$ vertices, $n \geq 1$.  The (infinite) right-angled Coxeter group $W = W_\G$ is $1$-ended if and only if $\G$ is connected and has no separating induced $K_n$ subgraph, by \cite[Theorem~8.7.2]{davis-book}.  Moreover, $W_\G$ is hyperbolic if and only if $\G$ has no ``empty squares", by~\cite[Corollary~12.6.3]{davis-book}.  

Assume that $\{a,b\}$ is a pair of essential vertices in $\G$ so that $\G \setminus \{a,b\}$ has components $\Lambda_1$, \dots, $\Lambda_k$, with $k \geq 2$.  As we now explain, determining whether the endpoints of $(a,b)$-bicoloured geodesics form $\approx$-pair vertices in $\cT$, and the valence of these vertices, is quite delicate.  
The easiest case is when none of the $\Lambda_i$ contain a vertex which is adjacent to both $a$ and~$b$.  Then Lemma~\ref{lem:separation}(1) generalises to show that the endpoints of $(a,b)$-bicoloured geodesics are $\approx$-pair vertices in $\cT$ if and only if $k \geq 3$, and such vertices have valence $k$.  Suppose now that $\Lambda = \Lambda_1$ does contain a vertex $c$ which is adjacent to both $a$ and $b$.  (Note that in this case, none of $\Lambda_2$, \dots, $\Lambda_k$ contain a vertex adjacent to both $a$ and $b$, as this would yield an empty square.)  If $\Lambda$ consists of a complete graph on $n \geq 1$ vertices, all $n$ of which are adjacent to both $a$ and $b$, then Lemma~\ref{lem:separation}(2) (which considers $n = 1$) generalises to show that the endpoints of $(a,b)$-bicoloured geodesics are $\approx$-pair vertices in $\cT$ if and only if either $k \geq 3$ or $n \geq 2$, and such vertices have valence $2^n(k-1)$.

Other cases are also possible.  For instance, suppose that $\Lambda$ has $3$ vertices $c$, $x$ and $y$ forming a triangle, with $c$ adjacent to both $a$ and $b$, $x$ adjacent to $a$ but not $b$ and $y$ adjacent to $b$ but not $a$.  Then it can be seen that the endpoints of $(a,b)$-bicoloured geodesics separate $\partial W$ into  $2(k-1) + 1$ components, with the $2(k-1)$ coming from $\Lambda_2,\dots,\Lambda_k$, and the remaining component coming from $\Lambda$.  Thus in this case, $\{a,b\}$ yields $\approx$-pairs of valence $2(k-1) + 1 = 2k - 1$ for all $k \geq 2$.  In fact, a careful analysis, with repeated use of the ``no separating $K_n$" and ``no empty square" conditions, shows that when $\Lambda$ has exactly $3$ vertices, the only possible valences of $\approx$-pairs coming from $\{a,b\}$  are $k \geq 3$, $2^3(k-1) \geq 8$ and $2k - 1 \geq 3$; for larger graphs $\Lambda$, the analysis will be even more involved.  Thus in general, the existence of $\approx$-pairs coming from $\{a,b\}$, and their valence, will depend on both  $k$ and the finer structure of $\G$.

Having determined the essential cut pairs $\{a,b\}$ which yield $\approx$-pair vertices in $\cT$, and their valence, other results in this section will generalise without too much trouble.  The fact that all $\approx$-pair vertices come from such essential cut pairs can be obtained by using similar arguments to those in Lemma~\ref{lem:all approx}, together with a more general version of Theorem~\ref{thm:mihalik-tschantz}, from Mihalik and 
Tschantz~\cite{mihalik-Tschantz}.  Moreover, using a suitable generalisation of Corollary~\ref{cor:parallel}, it is not hard to see that the stabiliser of any $\approx$-pair vertex corresponding to  the endpoints of an $(a,b)$-bicoloured geodesic will be a conjugate of the subgroup of $W$ generated by $a$, $b$ and any vertices which are adjacent to both $a$ and $b$.

\end{remark}

%%%%%%%%%%%%%%%%%%%%%%%%%%%%%%%
\subsection{Construction of certain $\sim$-classes}\label{sec:sim}
%%%%%%%%%%%%%%%%%%%%%%%%%%%%%%%
Recall from Section~\ref{sec:bowditch} that the $\sim$-classes 
are equivalence classes of points in $\partial W (2)$ such that any pair of points in the set separates $\partial W$ into exactly two components, and the set is maximal with respect to this property.  In 
Proposition~\ref{prop:sim} below, we describe a construction which yields $\sim$-classes.  We show in Section~\ref{sec:type 1 and 2} that every $\sim$-class arises in this way. 

Let $|\G|$ denote the geometric realisation of $\G$.  
Note the distinction between $\G\setminus \{a, b\}$ and $|\G|\setminus \{a, b\}$, for an arbitrary pair of vertices $a$ and $b$ of $\G$: a component of the former must have a vertex, while a component of the latter is allowed to be an arc between adjacent vertices $a$ and $b$.  In particular, if $a$ and $b$ are adjacent vertices of $\G$, they do not separate $\G$ 
(by Standing Assumptions~\ref{assumptions}(2))
but they do separate $|\G|$.  A \emph{branch} of $\G$ is a subgraph of $\G$ consisting of a (closed) reduced path between a pair of essential vertices, which does not contain any essential vertices in its interior.  For example, both of the graphs in Figure~\ref{fig:approx-pairs} above contain $4$ branches between $a$ and $b$.

\begin{prop}[$\sim$-classes]\label{prop:sim}
Let $A$ be a set of (not necessarily essential) vertices of $\G$ such that $\langle A \rangle$ is infinite and:
\begin{enumerate}
\item[\emph{(A1)}] elements of $A$ pairwise separate $|\G|$; 
\item[\emph{(A2)}] given any subgraph $\Lambda$ of $\G$ which is a 
subdivided copy of $K_4$, if $\Lambda$ contains at least three vertices of $A$ then all the vertices of $A$ lie on a single branch of $\Lambda$; and 
\item[\emph{(A3)}] the set $A$ is maximal among all sets satisfying both \emph{(A1)} and \emph{(A2)}.
\end{enumerate}
Let $\cC_A$ be any copy of the Cayley graph of $\langle A\rangle $ in $\cC$, and let  $\cA$ be the set of geodesics $\gamma \subset \cC_A$ which are bicoloured by pairs $\{a, b\}$ as in 
Corollary~\ref{cor:approx} (so that $\partial \gamma$ is an $\approx$-pair).  Assume that $\partial \cC_A$ is not an $\approx$-pair.   Then $\partial \cC_A \setminus \partial \cA$ contains at least two points, and:

\begin{enumerate}
\item If $\langle A \rangle $ is $2$-ended, $\partial \cC_A \setminus \partial \cA = \partial \cC_A$ is a $\sim$-pair.

\item Otherwise, $\partial \cC_A \setminus \partial \cA$ is an infinite $\sim$-class.
\end{enumerate}
\end{prop}

\begin{remark}\label{rmk: A infinite}  Let $A$ be a subset of vertices of $\G$.  Then $\langle A \rangle$ is infinite if and only if $A$ contains a pair of non-adjacent vertices; in particular, if $\langle A \rangle$ is infinite then $\card(A) \geq 2$.  Also, we can recognise whether $\partial \cC_A$ is an $\approx$-pair using the graph-theoretic criteria on $A$ given by Corollary~\ref{cor:recognise approx}.  Finally, $\langle A \rangle$ is $2$-ended if and only if either $A = \{a,b\}$, where $a$ and $b$ are non-adjacent vertices of $\G$, or $A = \{a,b,c\}$, where $a$ and $b$ are non-adjacent and $c$ is adjacent to both $a$ and $b$.  Thus all hypotheses of Proposition~\ref{prop:sim} can be verified in the graph~$\G$.
\end{remark}

\begin{example}\label{eg:sim}  We now discuss some examples which illustrate the statement of Proposition~\ref{prop:sim}.  In the graphs in Figure~\ref{fig:approx-pairs} there are no subgraphs which are subdivided copies of $K_4$, so a set of vertices $A$ satisfying (A1)--(A3) is just required to be maximal with respect to (A1).  In the left-hand graph in this figure, if $A$ is the union of $\{a,b\}$ with the vertex adjacent to both $a$ and $b$, then $A$ is maximal with respect to (A1).  However by Corollary~\ref{cor:recognise approx}, $\partial \cC_A$ is an $\approx$-pair, so the set $A$ does not provide a $\sim$-class in this case.  Now suppose that in either graph, $A$ is
a branch from $a$ to $b$ of length at least three.  Then $A$ is maximal with respect to property (A1), hence properties (A1)--(A3) hold, and since $\card(A) \geq 4$ we have by 
Corollary~\ref{cor:recognise approx} that $\partial \cC_A$ is not an $\approx$-pair. 
Thus $A$ provides an infinite $\sim$-class.

Now consider the examples in 
Figure~\ref{fig:sim-eg}.  In this figure, the sets $\{a, b, c, d\}$ and $\{a, d, f, g, e\}$ on the left, the set $\{a_1,a_2,a_3,a_4,a_5\}$ in the centre and the set $\{p, q\}$ on the right all satisfy properties (A1)--(A3) in Proposition~\ref{prop:sim} and do not give $\approx$-pairs. The set $\{p, q\}$ thus yields a $\sim$-pair, while the other three correspond to infinite $\sim$-classes.  On the left and in the centre of Figure~\ref{fig:sim-eg} there is no subdivided $K_4$ subgraph, so (A3) just requires maximality with respect to (A1).  In the right-hand graph, the situation is more subtle.  The set $\{p,q\}$ satisfies (A1) and (A2), but is not maximal with respect to (A1).  Adding any non-essential vertex to $\{p,q\}$ means that (A1) fails, and adding any essential vertex to $\{p,q\}$ means that (A2) fails, so (A3) does hold for the set $\{p,q\}$.  Now consider the set $\{u,q\}$.  This satisfies (A1) and (A2) but is maximal with respect to neither, since adding any vertex on the branch between $u$ and $q$ results in a larger set also satisfying both (A1) and (A2).  So $\{u,q\}$ does not satisfy (A3) (but the branch between $u$ and $q$ does).  Finally, the set $\{p,q, r,s\}$ is maximal with respect to  (A1) but fails (A2), and so does not give a $\sim$-class.
\end{example}

\begin{figure}[ht]
\begin{center}
\begin{overpic}%[grid, tics=5]
{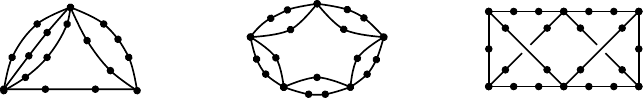}
%left
\put(10,15.5){\footnotesize $a$}
\put(4,13){\footnotesize $b$}
\put(-1,7){\footnotesize $c$}
\put(-2.5,-1){\footnotesize $d$}
\put(22,-1){\footnotesize $e$}
\put(6,-2){\footnotesize $f$}
\put(14,-1){\footnotesize $g$}
%middle
\put(48,16){\footnotesize $a_1$}
\put(61, 9){\footnotesize $a_5$}
\put(54,-1){\footnotesize $a_4$}
\put(42,-1){\footnotesize $a_3$}
\put(34.5,9){\footnotesize $a_2$}
%right
\put(87,-1.5){\footnotesize $q$}
\put(87,15.5){\footnotesize $p$}
\put(100,-1){\footnotesize $s$}
\put(100,14.5){\footnotesize $r$}
\put(73,-1){\footnotesize $v$}
\put(73,14.5){\footnotesize $u$}
\end{overpic}
\vspace{0.5cm}
\caption{\footnotesize{Examples to illustrate properties (A1)--(A3) in Proposition~\ref{prop:sim}.}}
\label{fig:sim-eg}
\end{center}
\end{figure}
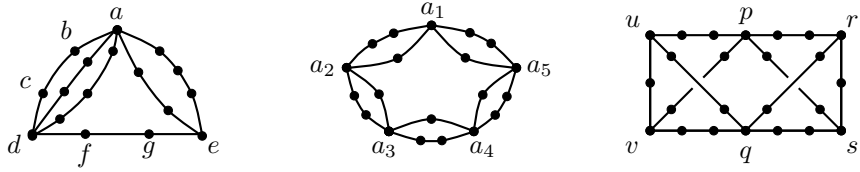

For the proof of Proposition~\ref{prop:sim}, we need several preliminary results concerning the structure of a set $A$ satisfying properties (A1), (A2) and/or (A3).  Lemma~\ref{lem:A1 cyclic} shows that subsets $A$ which satisfy property (A1) have a ``cyclic'' configuration in $\G$.  
In Lemma~\ref{lem:A1A2xy} we prove that if in addition $A$ satisfies (A2), then the cyclic ordering on $A$ is well-defined, and paths between pairs of points on induced cycles which contain $A$ are strongly restricted.  We use this in Corollary~\ref{cor:all A not b} to show that if $A$ also satisfies (A3), then for any vertex $b \not \in A$, there is an induced cycle in $\G$ which contains  $A$ but does not contain $b$.  

\begin{lemma} \label{lem:A1 cyclic}
Let $A$ be a set of vertices of $\G$ with $\langle A \rangle$ infinite and satisfying property \emph{(A1)}.  Then the elements of $A$ lie on an induced cycle $\alpha$ in $\G$.
\end{lemma}

\begin{proof}  The proof is by induction on $\card(A) = n \geq 2$.  Since $\G$ is connected and has no separating vertices or edges, the statement is satisfied when $A$ has $2$ elements.   

Now let $A$ be a set of $n+1$ elements satisfying property (A1).  By the induction hypothesis, after removing one element, say $b$, the remaining $n$ elements lie on an induced cycle $\alpha$.  If $b$ also lies on $\alpha$ we are done, otherwise label the elements of $A\setminus \{b\}$ going around $\alpha$ cyclically as $a_1,\dots,a_n$ and let $\beta$ be the subpath of $\alpha$ from $a_1$ to $a_n$ which contains all of $a_2,\dots,a_{n-1}$.  

Since the pair $\{a_n,b\}$ separates $|\G|$ there is a path $\eta$ in $\G$ from $a_n$ to $b$ which meets $\alpha$ only at $a_n$.  Similarly, there is a path $\zeta$ in $\G$ from $b$ to $a_1$ which meets $\alpha \cup \beta$ only at $a_1$ and $b$.  Then $\alpha \cup \beta \cup \zeta$ is an embedded cycle containing all elements of $A$, as desired.  It is clear from the construction that this cycle can be chosen to be an induced cycle.
\end{proof}

\begin{remark}\label{rmk:hyp plane}
Let $A$ be as in Lemma~\ref{lem:A1 cyclic} and let $\alpha$ be an induced cycle in $\G$ which contains all elements of $A$.  Consider any copy of the subcomplex $\Sigma_\alpha$ of $\Sigma$ corresponding to the special subgroup generated by all vertices in $\alpha$.  Since $\alpha$ contains at least $5$ vertices and a copy of the Cayley graph $\cC_A$ is contained in $\Sigma_\alpha$, we see that any copy of $\cC_A$ embeds in a subcomplex of $\Sigma$ which is quasi-isometric to the hyperbolic plane. 
\end{remark}

\begin{lemma}\label{lem:A1A2xy}
Let $A$ be a set of vertices of $\G$ with $\langle A \rangle$ infinite and satisfying properties \emph{(A1)} and \emph{(A2)}. Let $\alpha$ be an induced cycle in $\G$ containing all elements of $A$ (as guaranteed by Lemma~\ref{lem:A1 cyclic}) and label the elements of $A$ going around $\alpha$ cyclically as $a_1, \dots, a_n$. Let $\sigma$ be any other induced cycle containing all elements of $A$.  Then:
\begin{enumerate}
\item The cycle $\sigma$ can be oriented so as to induce this same order on $A$.   Thus up to orientation, there is a well-defined cyclic ordering on the elements of $A$.
\item  Assuming $\card(A) \geq 3$, if $x$ and $y$ are points on $\sigma$ (not necessarily vertices of $\G$) such that $\sigma \setminus 
\{x, y \}$ partitions $A$ into two nonempty sets  $\{a_i,a_{i+1}, \dots, a_{j-1}\}$ and $\{a_{i-1}, \dots, a_{1},a_n, \dots, a_{j+1},a_j\}$, where $i < j$ and $x$ (respectively $y$) lies between $a_i$ and $a_{i-1}$ (respectively $a_j$ and $a_{j-1}$), then any reduced path in $\G$ from $x$ to $y$ passes through either all of $a_i,a_{i+1}, \dots, a_{j-1}$, or 
all of $a_{i-1}, \dots, a_{1},a_n, \dots, a_{j+1},a_j$. 
\end{enumerate}
\end{lemma}

\begin{proof}
Part (1) is immediate when $A$ has $2$ or $3$ elements.  If $\card(A) = n \geq 4$ suppose there is an induced cycle $\sigma$ containing all elements of $A$ which cannot be oriented to induce the same ordering.  Then without loss of generality there is a vertex $a_i$ with $i \neq 2,n$ so that $a_1$ is adjacent to $a_i$ on $\sigma$.  Since $\sigma$ is an induced cycle and $n \geq 4$ there is also some $a_k$ with $1 < k < i$ and some $a_l$ with $i < l \leq n$ so that $a_k$ and $a_l$ are adjacent on $\sigma$.  Let $\sigma_{1i}$ be the subpath of $\sigma$ connecting $a_1$ to $a_i$ and containing no other vertices of $A$, and similarly define $\sigma_{kl}$.  Notice that $\sigma_{1i}$ connects the two components of $\alpha\setminus \{a_k,a_l\}$ and $\sigma_{kl}$ connects the two components of $\alpha \setminus \{a_1,a_i\}$.  

Now there is a subpath $\sigma_{1i}'$ of $\sigma_{1i}$ which connects the two components of $
\alpha \setminus \{a_k,a_l\}$ and intersects $\alpha$ only at its endpoints, and 
there is a similar subpath $\sigma_{kl}'$ of $\sigma_{kl}$.
 Consider the graph $\Lambda = \alpha \cup \sigma_{1i}' \cup \sigma_{kl}'$.  This graph 
is a subdivided $K_4$, and since it contains $\alpha$ it contains all vertices of $A$.  But by 
construction, at most two of the four vertices $a_1, a_i, a_k, a_l$ can lie on any branch of $\Lambda$.  
This contradicts (A2), and so completes the proof of part (1).

For (2), assume $\card(A) = n \geq 3$ and let $x$ and $y$ be points as in the statement.   Assume by way of contradiction that there is a reduced path $\xi$ in $\G$ from $x$ to $y$ which fails the condition. 
Suppose that $\xi$ misses some vertex, say $a_k$, in the set $\{a_i,a_{i+1}, \dots, a_{j-1}\}$, and some vertex, say $a_l$, in the set $\{a_{i-1}, \dots, a_{1},a_n, \dots, a_{j+1},a_j\}$.  
Then since $a_k$ and $a_l$ separate $|\G|$, there is a reduced path $\beta$ connecting $a_k$ and $a_l$ which meets $\sigma \cup \xi$ only at its endpoints.  Now $\xi$ connects the two components of $\sigma \setminus \{a_k,a_l\}$, so $\xi$ contains a subpath $\xi'$ which intersects $\sigma$ only at its endpoints and connects the two components of $\sigma \setminus \{a_k,a_l\}$, as shown in 
Figure~\ref{fig:lem-cyclic}.  Then $\sigma \cup \xi' \cup \beta$ forms a subdivided $K_4$ in $\G$, which contains all $n \geq 3$ vertices of $A$.  But the branch $\beta$ of this $K_4$ contains only the two vertices $a_k$ and $a_l$, and this contradicts (A2).  So $\xi$ must contain all of the vertices in at least one of the sets in the induced partition of $A$.
\begin{figure}[ht]
\begin{center}
\begin{overpic}[scale=0.5]%, grid, tics=15]
{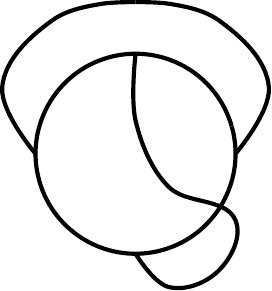}
\put(45, 85){\scriptsize $x$}
\put(37,4){\scriptsize $y$}
\put(-1,40){\scriptsize $a_k$}
\put(83,40){\scriptsize $a_l$}
\put(83,10){\scriptsize $\xi$}
\put(57,44){\scriptsize $\xi'$}
\put(15,15){\scriptsize $\sigma$}
\put(-5,80){\scriptsize $\beta$}
\end{overpic}
\caption{}
\label{fig:lem-cyclic}
\end{center}
\end{figure}
\end{proof}

\begin{cor}\label{cor:all A not b}   Let $A$ be a set of vertices of $\G$ satisfying properties \emph{(A1)}, \emph{(A2)} and \emph{(A3)}, such that $\langle A \rangle$ is infinite and $\partial \cC_A$ is not an $\approx$-pair.   Let $b$ be a vertex of $\G$ which is not in $A$.  
\begin{enumerate}
\item If $\langle A \rangle$ is $2$-ended, let $\{a_1,a_2\}$ be the unique cut pair in $A$.  Then there is an induced cycle $\alpha$ in $\G$ which contains $a_1$ and $a_2$ and contains a vertex from both components of $\G \setminus \{a_1,a_2\}$, but does not contain $b$.
\item If $\langle A \rangle$ is not $2$-ended, there is an induced  cycle $\alpha$ in $\G$ which contains all elements of $A$ but does not contain $b$.
\end{enumerate}
\end{cor}

\begin{proof}  Assume first that $\langle A \rangle$ is $2$-ended.  Note that $\G \setminus \{a_1,a_2\}$ must have exactly two components, say $\Lambda_1$ and $\Lambda_2$, otherwise $\partial \cC_A$ would be an $\approx$-pair.  Suppose $b \in \Lambda_1$ and let $\mu$ be a reduced path from $a_2$ to $a_1$ in $\Lambda_2$.  To obtain the desired cycle $\alpha$, it suffices to concatenate $\mu$ with a reduced path from $a_1$ to $a_2$ in $\Lambda_1$ which does not contain $b$.  We assume by contradiction that every reduced path from $a_1$ to $a_2$ in $\Lambda_1$ passes through $b$.  It follows that property (A1) holds for the set $\{a_1,a_2,b\}$.  Now suppose that there is a subgraph of $\G$ which is a subdivided $K_4$ and contains all three vertices $a_1$, $a_2$ and $b$.  Then as every reduced path from $a_1$ to $a_2$ in $\Lambda_1$ contains $b$, the vertices $a_1$, $a_2$ and $b$ lie on the same branch of this $K_4$.  Thus property (A2) also holds for the set $\{a_1,a_2,b\}$.  Hence $\{a_1,a_2,b\}$ is contained in a maximal set which satisfies both (A1) and (A2).  If $A = \{a_1,a_2\}$ this contradicts (A3), since $b \not \in A$.  

Since $\langle A \rangle$ is $2$-ended the only other possibility is that $A = \{a_1,a_2,c\}$ with $c \neq b$ and $c$ adjacent to both $a_1$ and $a_2$.  Notice that as every reduced path from $a_1$ to $a_2$ in $\Lambda_1$ passes through $b$, the vertex $c$ lies in $\Lambda_2$.  In particular, $b$ and $c$ are non-adjacent.  Now the set $\{a_1,a_2,b,c\}$ must fail at least one of (A1) and (A2).  Suppose first that (A1) fails for this set.  Since (A1) does hold for $\{a_1,a_2,b\}$ and $\{a_1,a_2,c\}$, the pair $\{b,c\}$ is not a cut pair.  Thus there is a reduced path $\eta$ from $a_1$ to $a_2$ in $\G \setminus \{b,c\}$.  If $\eta$ is in $\Lambda_1$ this contradicts our assumption that every reduced path from $a_1$ to $a_2$ in $\Lambda_1$ contains $b$, so $\eta$ is in $\Lambda_2$.  Now $c \in \Lambda_2$ is adjacent to both $a_1$ and $a_2$, so there must be a reduced path in $\Lambda_2$, say $\beta$, from $c$ to some point in $\eta$.  Notice that $\beta$ does not contain either $a_1$ or $a_2$, and $\eta$ does not contain $c$.  Let $\Lambda$ be the union of $\eta$, $\beta$, the edges between $a_i$ and $c$ for $i = 1,2$, and a path in $\Lambda_1$ from $a_1$ to $a_2$ (via $b$).  Then $\Lambda$ is a subdivided $K_4$ subgraph which has $a_1$, $a_2$ and $c$ as three of its essential vertices.  Hence these three vertices are not contained in a single branch of $\Lambda$.  This contradicts property (A2) for the set $\{a_1,a_2,c\}$.  Therefore property (A1) must hold for the set $\{a_1,a_2,b,c\}$.

As (A1) holds for the set $\{a_1,a_2,b,c\}$, property (A2) must fail.  Thus there is a subgraph $\Lambda$ of $\G$ which is a subdivided $K_4$ and contains at least three vertices of $\{a_1,a_2,b,c\}$, so that these vertices are not all on the same branch of $\Lambda$.  Since (A2) holds for the sets $\{a_1,a_2,b\}$ and $\{a_1,a_2,c\}$, we may assume that the vertices $a_1$, $b$ and $c$ all lie on $\Lambda$, but they are not all on the same branch.  Note that then $a_2$ cannot lie on $\Lambda$.   However $\Lambda \setminus \{a_1\}$ contains both $b$ and $c$, and is contained in a single component of $\G\setminus \{a_1,a_2\}$.  This contradicts $b$ and $c$ being in different components of $\G \setminus \{a_1,a_2\}$.  We conclude that there is a reduced path from $a_1$ to $a_2$ in $\Lambda_1$ which does not pass through $b$, and so the required cycle $\alpha$ can be obtained.

Now suppose $\langle A \rangle$ is not $2$-ended.  Then $\card(A) \geq 3$.  Consider an induced cycle $\alpha$ containing all elements of $A$, and inducing the cyclic order $a_1,\dots,a_n$ on the elements of $A$.  Suppose $b$ is on the cycle $\alpha$, say between $a_i$ and $a_{i+1}$.   Since $b$ is not in $A$, by property (A3) the set $A \cup \{b\}$ fails either (A1) or (A2).  

If (A1) fails, there is some $a_k \in A$ (possibly equal to $a_i$ or $a_{i+1}$) so that $\{b, a_k\}$ does not separate $|\G|$.  Then there is a path $\mu$ connecting the two components of $\alpha\setminus \{b, a_k\}$ and meeting $\alpha$ only at its endpoints $x$ and $y$.  We claim that both $x$ and $y$ lie on the (closed) subpath of $\alpha$ between $a_i$ and $a_{i+1}$ containing $b$.   If not, then one possibility is that both $x$ and $y$ both lie on the (open) component of $\alpha \setminus \{a_i,a_{i+1}\}$ which does not contain $b$ (in this case, $a_k$ must be distinct from both $a_i$ and $a_{i+1}$).  Then $a_k$ and $a_i$ (say) lie in different components of $\alpha \setminus \{x,y\}$.  By slightly extending $\mu$ if either $x$ or $y$ is in $A$, we obtain a reduced path $\mu'$ connecting two points $x'$ and $y'$ in $\alpha \setminus A$, so that $a_k$ and $a_i$ lie in different components of $\alpha \setminus \{x',y'\}$, but neither $a_i$ nor $a_k$ lies on $\mu'$.  This contradicts 
Lemma~\ref{lem:A1A2xy}(2).  By a similar argument we can rule out exactly one of $x$ and $y$ lying on the (closed) subpath of $\alpha$ between $a_i$ and $a_{i+1}$ containing $b$. This proves the claim.
Since $\mu$ does not contain $b$, we can then use $\mu$ to replace the subpath of $\alpha$ from $x$ to $y$ via $b$, and so obtain the required cycle.

Now suppose (A1) holds for the set $A \cup \{b\}$, but (A2) fails.  Then there is a subgraph $\Lambda$ of $\G$ which is a subdivided copy of $K_4$, and distinct vertices $a_k, a_l \in A$ such that $a_k$, $a_l$ and $b$ are not contained in the same branch of $\Lambda$.  If $a_k$ and $a_l$ are on different branches of $\Lambda$, then 
by (A1) there is a path $\eta$ from $a_k$ to $a_l$ which meets $\Lambda$ only at its endpoints.  
Then $\Lambda \cup \eta$ contains a subgraph, say $\Lambda'$, which is a subdivided $K_4$, and 
such that $a_k$ and $a_l$ are on the same branch of $\Lambda'$.   Moreover, $\Lambda'$ can be 
chosen so that $b \in \Lambda'$, on a branch other than the one containing $a_k$ and $a_l$.  Thus we may assume 
that $a_k$ and $a_l
$ are on the same branch of $\Lambda$.

We now have that $a_k$ and $a_l$ lie on the same branch of $\Lambda$, say $\beta$, and that $b$ is not on $\beta$. Then there is an induced cycle $\sigma$ in $\Lambda$ so that $\sigma$ contains $\beta$ but not $b$.  If all elements of $A$ lie on $\beta$ then $\sigma$ is an induced cycle containing all elements of $A$ but not $b$, so suppose that there is some $a_m \in A$ which does not lie on $\beta$.
 By property (A2) for $A$, the vertex $a_m$ then cannot lie on $\Lambda$.  Recall that by assumption, the cycle $\alpha$ which contains all of $A$ also contains $b$.  Then it is not hard to see that $\alpha \cup \Lambda$ contains a subdivided $K_4$ subgraph in which $a_k$, $a_l$ and $a_m$ do not lie on the same branch.  This contradicts property (A2) for $A$.  Thus all elements of $A$ lie on the branch $\beta$, which completes the proof.
\end{proof}

Finally, we consider 
separation properties of geodesics in $\cC$ labelled by elements of a set $A$ satisfying (A1) and (A2).

\begin{lemma}\label{lem:Aseparates}
Let $A$ be a set of vertices of $\G$ with $\langle A \rangle$ infinite satisfying properties \emph{(A1)} and \emph{(A2)}.  Let $\gamma$ be a bi-infinite geodesic in $\cC_A$ labelled by elements of $A$.  Then $\gamma$ separates $\Sigma$ and $\partial \gamma$ separates $\partial W$.
\end{lemma}

\begin{proof}
If $\card(A) =2$, with $A = \{a, b\}$, then $a$ and $b$ are not adjacent in $\G$ (since $\langle A \rangle$ is infinite), and so 
$\{a, b\}$ is a cut pair by (A1).  Hence by 
Lemma~\ref{lem:separation}, $\gamma$ separates $\Sigma$ and $\partial \gamma$ separates $\partial W$.

Now assume $\card(A) \geq 3$.  We first show that $\gamma$ separates $\Sigma$.  By Lemma~\ref{lem:gamma separates}, it is enough to show that $\gamma$ separates $N(\gamma)$.  
If not, then given any pair of points $x$ and $y$ in $N(\gamma) \setminus \gamma$, there is a path $\eta$ in $N(\gamma) \setminus \gamma$ connecting them.  Since $\eta$ misses $\gamma$, and in particular, the cone points of the chambers that make up $N(\gamma)$, we may assume that $\eta$ (as well as $x$ and $y$) lie in $\partial N(\gamma)$.  We show below that this leads to a contradiction.  

Write $N(\gamma)$ as the union of chambers $K_i$, with $-\infty \le i \le \infty$, such that for each $i$, the chambers $K_{i-1}$ and $K_{i}$ intersect in the panel $P_i$, of type $g_i$.  
Recall that each panel $P$ is a star, and denote by $\partial P$ the set of endpoints of the spokes of this star.  
The boundary of any chamber $K$ can be identified with the graph $\G$, and under this identification, the cone point of a panel $P$ of type $g$ corresponds to the vertex $g$ in $\G$, while the points of $\partial P$ correspond to midpoints of the edges emanating from $g$ in $\G$. 

Identify the boundary of  $K_0$ with $\G$ as above.  Fix an embedded cycle $\alpha$ containing $A$ so that $\alpha$ induces a cyclic ordering $a_1, \dots, a_n$ on the elements of $A$ with $a_1= g_0$.  
This order on $A$ induces a partition of $\partial P_0$ into two sets $P_0^+$ and $P_0^-$, as follows.  Let $x \in \partial P_0$, so that $x$ may be identified with the midpoint of an edge emanating from $a_1$ in $\G$.  Let  $\alpha_{1n}$ be the subpath of $\alpha$ connecting $a_1$ to $a_n$ and containing no other vertices of $\alpha$, and similarly define $\alpha_{12}$.  If $x$ lies on $\alpha$, then we put $x$ in $P_0^-$ if $x$ lies on $\alpha_{1n}$, and $x$ in $P_0^+$ if $x$ lies on $\alpha_{12}$.  

If $x$ does not lie on $\alpha$, let $b$ be the other vertex of the edge of $\G$ which has $x$ as its midpoint.  As $\G$ has no separating vertices or edges, there is a reduced path $\eta$ connecting $b$ to some vertex $b'$ of $\alpha$ other than $a_1$, so that $\eta$ intersects $\alpha$ only at $b'$.  
An argument similar to the proof of Lemma~\ref{lem:A1A2xy}(2), applied to the union of $\eta$ with the edge containing $x$, implies that $b'$ must lie on either $\alpha_{1n}$ or $\alpha_{12}$.
Put $x \in P_0^-$ if $b'$ lies on $\alpha_{1n}$, and $x \in P_0^+$ if $b'$ lies on $\alpha_{12}$.

To see that this is well-defined, 
suppose that there are reduced paths $\xi$ and $\xi'$ connecting $a_1$ via $b$ to vertices $b' \neq a_1$ 
on $\alpha_{1n}$ and $b'' \neq a_1$ on $\alpha_{12}$ respectively, so that $\xi$ and $\xi'$ meet $\alpha$ only at their endpoints.  
Then the graph $\Lambda=\alpha \cup \xi \cup \xi'$
 is a subdivided $K_4$ containing all of $A$, and its branch $\xi$ contains at most two vertices of $A$. This contradicts (A2).   Thus $P_0^-$ and 
$P_0^+$ are well-defined, and  we have completed the partition of $\partial P_0$.

Now we inductively define a partition of $\partial P_i$ into $P_i^+$ and $P_i^-$ for each $i \neq 0$ such that:
\begin{enumerate}
\item If $\nu$ is a path from $\partial P_{i}$ to $\partial P_{i+1}$ in $\partial K_i$ whose interior contains no points of $\partial P_{i} \cup \partial P_{i+1}$, then $\nu$ either connects 
$P_{i}^+$ to  $P_{i+1}^+$ or $P_{i}^-$ to  $P_{i+1}^-$.  
\item If $\nu$ is a path from $ P_{i}^+$ to $P_{i}^-$ in $\partial K_i \setminus P_i$, then $\nu$ passes through the cone point of~$P_{i+1}$.
\end{enumerate}

We first prove property (2) for $i=0$.   Suppose $\nu$ is a path from $ P_{0}^+$ to $P_{0}^-$ in $\partial K_0 \setminus P_0$.  Consider $\nu$ under the identification of $\partial K_0$ with $\G$ and call its endpoints $x$ and $y$.  Assume $\nu $ is reduced. By our assumption, 
$\nu$ does not pass through the vertex of $\G$ labelled $g_0 = a_1$.  
Then by Lemma~\ref{lem:A1A2xy}(2), $\nu$ must pass through all the other vertices of $A$.  Translating back to $K_0$, this means, in particular, that $\nu$ passes through the cone point of $P_1$.  

Now suppose that $P_1^\pm, \dots, P_{i-1}^\pm$ 
have been defined satisfying (2), with $P_1^\pm, \dots, P_{i-2}^\pm$ satisfying (1).
Identify $\partial K_i$ with $\G$, and let $x$ correspond to a point of $\partial P_i$.  Then there is a path $\nu$ in $\G$ between the vertex labelled $g_i$ and the vertex labelled $g_{i-1}$ passing through $x$.  Put $x$ in $P_i^+$ 
if $\nu$ passes through $P_{i-1}^+$ and in $P_i^-$ otherwise.  If $\mu$ and $\nu$ are two such paths, entering $P_{i-1}$ through $P_{i-1}^-$ and $P_{i-1}^+$ respectively, then $\mu\cup\nu$ reduces to a path connecting $P_{i-1}^+$ and $P_{i-1}^-$,
which does not pass through the cone point of $P_i$.  However, this contradicts property (2) for $P_{i-1}$.
This shows that the partition on $\partial P_i$ is well defined, and 
the partition on $\partial P_{i-1}$ satisfies property (1).  The proof that the partition on $\partial P_i$ satisfies property (2) is similar to the proof in the base case.  The definition of $P_i^\pm$ for $i<0$ is similar.  

Observe that an argument similar to the base case of (2) shows for all $i$

\begin{enumerate}
\setcounter{enumi}{2}
\item If $\nu$ is a path from $ P_{i}^+$ to $P_{i}^-$ in $\partial K_{i-1} \setminus P_i$, then $\nu$ passes through the cone point of~$P_{i-1}$.
\end{enumerate}

Now assume that the path $\eta$ in $\partial N(\gamma)$ constructed above connects 
some point $x\in P_0^+$ to some point $y \in P_0^-$.  Then $\eta$ is a union of segments, each of which connects some point of $\partial P_{i }$ to some point of 
$\partial P_i$ or $\partial P_{i+1}$ for some $i$, and does not pass through the cone points of the panels $P_i$ and $P_{i\pm1}$.  Now since $\eta$ connects $x$ to $y$, it must (without loss of generality) have some segment which connects a point of $P_j^+$ to either $P_j^-$ or $P_{j+1}^-$ for some $j$.  However this contradicts one of properties (1), (2), and (3) above.  

It follows that $\gamma$ does separate $N(\gamma)$, and therefore $\gamma$ separates $\Sigma$ as required. By similar arguments to those in Lemma 2.3 of Lafont~\cite{lafont-3d}, we obtain that $\partial \gamma$ separates $\partial W$ as well.
\end{proof}

We are now ready to prove Proposition~\ref{prop:sim}.

\begin{proof}[Proof of Proposition~\ref{prop:sim}]
Let $\cC_A$ be some copy of the Cayley graph of $\langle A \rangle$ in $\cC$.  Assume that $\langle A \rangle$ is infinite and that $\partial \cC_A$ is not an $\approx$-pair.  

We first show that $\partial \cC_A \setminus \partial \cA$ contains at least two points.  Since $\langle A \rangle$ is infinite, the set $\partial \cC_A$ has at least two points.  If $\langle A \rangle$ is $2$-ended, then $\partial \cC_A$ has exactly two points, and by the assumption that $\partial \cC_A$ is not an $\approx$-pair, it follows that no bicoloured geodesic in $\cC_A$ has endpoints an $\approx$-pair.  So the set $\cA$ is empty and $\partial \cC_A \setminus \partial \cA = \partial \cC_A$ has exactly two points.  If $\langle A \rangle$ is not $2$-ended, we may construct a geodesic $\gamma$ in $\cC_A$ so that neither direction of $\gamma$ is eventually bicoloured by a pair of vertices in $A$.  Thus by Lemma~\ref{lem:parallel}, neither $\gamma^+$ nor $\gamma^-$ is the endpoint of a bicoloured geodesic.  So $\gamma^+$ and $\gamma^-$ are two distinct points in $\partial \cC_A \setminus \partial \cA$.

We next show that any two points in $\partial \cC_A \setminus \partial \cA$ are in the same $\sim$-class.  For this, let $\xi, \eta \in \partial \cC_A \setminus \partial \cA$, with $\xi \neq \eta$.  To show that $\xi \sim \eta$, we must show that $\partial W \setminus \{\xi,\eta\}$ has exactly two components.  
Since $\langle A \rangle$ is convex in $W$ and $\xi,\eta \not \in \partial \cA$, the boundary points $\xi$ and $\eta$ are the endpoints of some geodesic $\gamma$ in $\cC_A \setminus \cA$.  By Lemma~\ref{lem:Aseparates}, $\partial \gamma = \{ \xi, \eta\}$ separates $\partial W$.  Thus $\xi$  and $\eta$ are in $\partial W(2+)$.  On the other hand, since $\xi$ and $\eta$ are not in $\partial \cA$, by our characterisation of $\approx$-pairs in Lemma~\ref{cor:approx} we know that neither of them is part of an 
$\approx$-pair. 
It follows that neither $\xi $ nor $\eta$ is in $\partial W (3+)$, 
by property (2) following Definition~\ref{def:approx}.
So $\xi$ and $\eta$ are both in $\partial  W(2)$, and they separate $\partial W$ into exactly two components.  
Thus $\xi \sim \eta$.  

To complete the proof of Proposition~\ref{prop:sim}, we show that $\partial \cC_A \setminus \partial \cA$ is a full $\sim$-class, i.e, that no point of $\partial W(2)$ outside the set $\partial \cC_A \setminus \partial \cA$ is equivalent to a point inside it.  Assume by contradiction that there are points $\xi \in \partial \cC_A \setminus \partial \cA$ and 
$\eta \not \in \partial \cC_A$ 
so that $\xi \sim \eta$.  Then since $\xi \sim \eta$, we have that $\{\xi,\eta\}$ separates $\partial W$ 
into exactly two components, say $U_1$ and $U_2$.  Now $\xi \in \partial W(2)$, so there is a 
neighbourhood $V$ of $\xi$ in $\partial W$ so that $V \setminus \{\xi\}$ has exactly two components, 
say $V_1$ and $V_2$.  Suppose $V$ intersects only one of $U_1$ and $U_2$, say $U_1$.  Then $\partial W \setminus \eta $ is the union of the open sets $V\cup U_1$ and $U_2$, and hence $\eta$ is a cut point of $\partial W$.  This contradicts the result from~\cite{swarup} that the boundary of a $1$-ended hyperbolic group has no global cut points, as discussed in the introduction to~\cite{bowditch}.  Thus 
$V$ intersects both $U_1$ and $U_2$ nontrivially.
On the other hand, if $V_1$ has nonempty intersection with both $U_1$ and $U_2$, this 
contradicts the connectedness of $V_1$.  So without loss of generality $V_1 = V \cap U_1$, and 
similarly $V_2 = V \cap U_2$.  
We claim that there is a path in $\partial W \setminus \{\xi,\eta\}$ which connects a point in $V_1$ to a 
point in $V_2$.  This yields a contradiction, since this path connects $U_1$ to $U_2$ in $\partial W \setminus \{\xi,\eta\}$.

Let $\xi'$ be a point of $\partial \cC_A \setminus \partial \cA$ with $\xi' \neq \xi$.  We showed  above that $\xi \sim \xi'$.  Let $\gamma$ be a geodesic in $\cC$ connecting $\xi$ to $\eta$.  Since $\xi \in \partial \cC_A$, in the direction of $\xi$ the geodesic $\gamma$ is eventually labelled by elements of $A$.   
Now since $\eta \not \in \partial \cC_A$, not all of $\gamma$ is contained in $\cC_A$.  Starting at $\xi$ and travelling back along $\gamma$, let $b$ be the first label of $\gamma$ which is not in $A$.   

We first consider the case that $\langle A \rangle$ is $2$-ended.  Then $\partial \cC_A \setminus \partial \cA = \partial \cC_A$ consists of the two endpoints $\xi$ and $\xi'$ of an $(a_1,a_2)$-bicoloured geodesic, where $a_1, a_2 \in A$ and $\{a_1,a_2\}$ is a cut pair in $\G$.  Let $\Lambda_1$ and $\Lambda_2$ be the two components of $\G \setminus \{a_1,a_2\}$.  Then by Corollary~\ref{cor:all A not b}(1), there is an induced cycle $\alpha$ in $\G$ which contains both $a_1$ and $a_2$, does not contain $b$, and has nonempty intersection with both $\Lambda_1$ and $\Lambda_2$.  Let $\Sigma_{\alpha}$ be the corresponding subcomplex of $\Sigma$ containing our chosen copy of $\cC_A$, as in Remark~\ref{rmk:hyp plane}.  
Then $\Sigma_{\alpha}$ is quasi-isometric to the hyperbolic plane, and $\xi, \xi' \in \partial \Sigma_{\alpha}$ but $\eta \not \in \partial \Sigma_{\alpha}$ since $b$ is not on $\alpha$.  Let $C_1$ and $C_2$ be the two components of $\partial W \setminus \{\xi,\xi'\}$.  Since the cycle $\alpha$ contains vertices in both components of $\G \setminus \{a_1,a_2\}$, Lemma~\ref{lem:separation} implies that the circle $\partial \Sigma_{\alpha}$ has nonempty intersection with both $C_1$ and $C_2$.  Now if both $C_1$ and $C_2$ have nonempty intersection with $V_1$ this contradicts the connectedness of $V_1$, and so without loss of generality $V \cap C_i = V_i$ for $i = 1,2$.  Hence $\partial \Sigma_{\alpha}$ has nonempty intersection with both $V_1$ and $V_2$.  Let $p_i \in \partial \Sigma_{\alpha} \cap V_i$ for $i = 1,2$.  Then the arc of $\partial \Sigma_{\alpha}$ from $p_1$ to $p_2$ which misses $\xi$ gives a path in $\partial W \setminus \{\xi,\eta\}$ from $V_1$ to $V_2$, proving the claim in this case.

Now suppose that $\langle A \rangle$ is not $2$-ended.  By Corollary~\ref{cor:all A not b}(2), there is an induced cycle $\alpha$ in $\G$ which contains all elements of $A$ but does not contain $b$.  Let $\Sigma_\alpha$ be the corresponding subcomplex of $\Sigma$ which contains $\cC_A$ and is quasi-isometric to the hyperbolic plane.  Then $\xi \in \partial \Sigma_\alpha$ but $\eta \not \in \partial \Sigma_\alpha$.  Similarly to the previous case, to prove the claim it suffices to show that the sets $\partial \Sigma_\alpha \cap V_1$ and $\partial \Sigma_\alpha \cap V_2$ are both nonempty.

Since $\langle A \rangle$ is not $2$-ended, there is a point $\xi'' \in \partial \cC_A \setminus \partial \cA$ with $\xi''$ distinct from both $\xi$ and $\xi'$.  Let $C_1$ be the component of $\partial W \setminus \{\xi,\xi'\}$ which does not contain $\xi''$ and let $C_2$ be the component of $\partial W \setminus \{\xi,\xi''\}$ which does not contain $\xi'$.  Then similarly to the previous case, we may assume that $V \cap C_i = V_i$ for $i = 1,2$.  Now the component $C_1$ contains the interior of the arc of $\partial \Sigma_\alpha$ from $\xi$ to $\xi'$ which misses $\xi''$, hence there is a point $p_1 \in \partial \Sigma_\alpha \cap V_1$.  Similarly, we obtain a point $p_2 \in \partial \Sigma_\alpha \cap V_2$, which completes the proof of the claim in this case.

We have shown that $\partial \cC_A \setminus \partial \cA$ is a $\sim$-class in $\partial W$.  Now this $\sim$-class is a $\sim$-pair if and only if $\partial \cC_A$ contains exactly two points, which occurs if and only if $\langle A \rangle$ is $2$-ended.  This completes the proof
of Proposition~\ref{prop:sim}.
\end{proof}

\begin{remark}\label{rem:approx-sim edges}  We can now construct some edges in the pre-tree $T$.  Let $A$ be a set of vertices of $\G$ so that $\langle A \rangle$ is infinite and not $2$-ended and properties (A1)--(A3) from Proposition~\ref{prop:sim} hold.  Let $\cC_A$ be a copy of the Cayley graph of $\langle A \rangle$.  Then since $\langle A \rangle$ is not $2$-ended, $\partial \cC_A$ is not a $\approx$-pair.  Let $v = \partial \cC_A \setminus \partial \cA$ be the corresponding infinite $\sim$-class.  Then for each $\gamma \in \cA$, the endpoints $\{\gamma^+,\gamma^-\}$ form an $\approx$-pair which is in the closure of the set $v$.  Hence there is an edge in the pre-tree $T$ connecting $v$ to this $\approx$-pair vertex.\end{remark}

%%%%%%%%%%%%%%%%%%%%%%%%%%%%%%%%%%  
\subsection{Identification of the $\sim$-classes and their stabilisers}\label{sec:type 1 and 2}
%%%%%%%%%%%%%%%%%%%%%%%%%%%%%%%%%%

We now show that we have already constructed all the $\sim$-classes in $\partial W$.  Together with our identification of the $\approx$-pairs in Lemma~\ref{lem:all approx} and the relation of betweenness, 
this means we have constructed the pretree $T$. This section also identifies the stabilisers of $\sim$-classes.

Bowditch shows that $\partial W(2+)$ is equal to the disjoint union of the $\approx$-pairs and $\sim$-classes in $\partial W$.  (See Remark~\ref{rem:allM(2+)}.)
Observe that the constructions in Corollary~\ref{cor:approx} and Proposition~\ref{prop:sim}
use exactly those points in $\partial W(2+)$ which correspond to geodesic rays that are eventually labelled by elements of some set $A$ satisfying conditions (A1) and (A2) in 
Proposition~\ref{prop:sim}.
 (Given  a $\approx$-pair associated to an $(a, b)$-bicoloured bi-infinite geodesic,  
 the set $\{a,b\}$ satisfies (A1) and (A2).)   The following proposition shows that such points in $\partial W (2+)$ exhaust all of $
\partial W(2+)$.  It then follows that there are no additional $\sim$-classes.

 \begin{prop}\label{prop:all 2+}
Let $\gamma^+ \in \partial W$ be represented by a geodesic ray $\gamma$ in $\cC$ based at $e$.  Then $\gamma^+ \in \partial W(2+)$ if and only if $\gamma$ is eventually labelled by the elements of a set $A$ satisfying properties \emph{(A1)} and  \emph{(A2)} in 
Proposition~\ref{prop:sim}.
\end{prop}
 
 \begin{proof}
The ``if'' direction follows from the discussion in the paragraph above the statement. 

 For the converse, let $C$ be the set of generators of $W$ that occur infinitely often as labels of $\gamma$ (so that $\gamma$ is eventually labelled by $C$), and suppose that $C$ fails either (A1) or (A2).  To establish that $\gamma^+$ is not in $ \partial W(2+)$, it 
  is enough to show that $U_n \setminus \{\gamma^+ \}$ is connected for all $n$, where 
$U_n$ is the subset of $\partial W$ represented by geodesic rays in $\cC$ based at $e$
which agree with $\gamma$ for the first $n$ edges. Given $\eta^+, \mu^+ \in U_n\setminus \{\gamma^+ \}$, represented by geodesic rays $\eta$ and $\mu$ in $\cC$ which agree with $\gamma$ for the first $n$ edges, 
 we will construct a path between $\eta^+$ and $\mu^+$ in $U_n \setminus \{\gamma^+ \}$. 

We begin with two special cases.  First, suppose that $\eta$ and $\mu$ are both labelled, starting at $\gamma(n)$, exclusively by vertices of an induced cycle $\alpha$ of $\Gamma$. 
If $\Sigma_\alpha$ denotes the subcomplex corresponding to $\alpha$ based at $\gamma(n)$, then 
the parts of $\eta$ and $\mu$ beyond $\gamma(n)$ lie in 
$\Sigma_\alpha$. 
Thus $\eta^+$ and $ \mu^+$ are in $ \partial \Sigma_\alpha $, which is homeomorphic to  $S^1$
(see Remark~\ref{rmk:hyp plane}).  There are two arcs connecting $\eta^+$ and $\mu^+$ in $\partial \Sigma_\alpha$, and at least one of these is contained in $U_n$; call this arc $\nu$. 

If $\nu$ does not pass through $\gamma^+$, then it is the desired path between 
$\eta^+$ and $ \mu^+$.  Now suppose $\gamma^+$ lies on $\nu$.  This means that if 
$V \subseteq U_n \cap \partial \Sigma_\alpha$ is such that $V\setminus \{ \gamma^+\}$ has two components $V_1$ and $V_2$, then $\nu$ intersects both $V_1$ and $V_2$.   We will obtain a detour to $\nu$ by constructing a path between $V_1$ and $V_2$ in $U_n \setminus \{\gamma^+ \}$. 
Observe that since $\gamma^+$ lies on $\nu \subset  \partial \Sigma_\alpha $, the set $C$ defined  above, of eventual labels of $\gamma$, is a subset of the vertices of $\alpha$.

If $C$ fails (A1), then there exists a pair $\{c_1, c_2 \}\subseteq C$ which does not  separate $|\G|$.  Then $c_1$ and $c_2$ are nonadjacent, so  the pair $\{c_1,c_2\}$ does not separate $\G$.  It follows 
that there is a path $\beta$ in $\G$ which connects the two components of $\alpha \setminus \{c_1, c_2\}$ 
and meets $\alpha$ only at its endpoints.  Let $\tau$ be the cycle obtained by taking the union of $\beta$ with the arc of $\alpha$ between the endpoints of $\beta$ which contains $c_1$.  The path $\beta$ can be chosen so that $\tau$ is induced.  
Now for $m \ge n$, 
let $\Sigma_{\tau,m}$ be the subcomplex corresponding to $\tau$ based at 
$\gamma(m)$.  Since $c_2$ is not on $\tau$, it follows that $\gamma^+$ does not lie on 
$\partial  \Sigma_{\tau,m}$.   Let $x$ and $y$ be two vertices of $\alpha$ which lie on $\tau \setminus \{c_1\}$ and are such that $c_1$ and $c_2$ lie in distinct components of $\alpha \setminus \{x,y\}$.  
Then  $\Sigma_{\tau,m} \cap \Sigma_\alpha$ contains (at least) 
 an $(x, y)$-bicoloured geodesic $\zeta_m$ passing through $\gamma(m)$.   We may choose $m$ large enough so that $\partial \zeta_m \subset V$, 
 and at least one arc of $\partial \Sigma_{\tau, m} $ connecting $\zeta_m^+$ and $\zeta_m^-$ 
 is in $U_n$.  This arc will yield the required path between $V_1$ and $V_2$, provided we can show that  $m$ can be chosen so that $\partial \zeta_m $ has non-trivial intersection with 
 $V_1$ as well as $V_2$.  

Assume $m$ is large enough that the labels of $\gamma$ beyond $\gamma(m-1)$ are in $C$.  
Now suppose the two edges of $\gamma$ incident to $\gamma(m)$ are labelled $c$ and $c'$ (so in particular, $c$ and $c'$ lie on $\alpha$), and $\xi_1, \xi_2$ are geodesic rays based at $\gamma(m)$ such that 
$\xi_1^+, \xi_2^+ \in V$.  Observe that by similar arguments to Lemma~\ref{lem:separation}, carried out in the subcomplex $\Sigma_\alpha$, if the first letters labelling $\xi_1$ and $\xi_2$ are in different components of $\alpha \setminus \{c, c'\}$ then one of $\xi_1^+$ and $\xi_2^+$ lies in $V_1$, and the other lies in $V_2$.  We use this fact to choose $m$ as follows.  

Let $C_1$ denote the elements of $C$ which lie on the 
arc of $\alpha \setminus \{x, y\}$ containing $c_1$, and let $C_2$ denote the remainder of $C$ (which may include one or both of $x$ and $y$).   It follows from the definition of $C$ that $\gamma$ is (eventually) labelled alternately by nonempty words in $C_1$ and $C_2$.  Now if $m$ is chosen as in the previous paragraphs with the additional property that at $\gamma(m)$, the label of $\gamma$ transitions from a word in $C_1$ to a word in $C_2$, then the above observation can be used to show that 
$\partial \zeta_m$
intersects both $V_1$ and $V_2$, and therefore, as described above, $\partial \Sigma_{\tau,m}$ contains the desired detour.  Together with $\nu$, this yields a path between $\eta^+$ and $\mu^+$ in $U_n\setminus \{\gamma^+\}$ when $C$ fails (A1).

If $C$ fails (A2), then there is a subdivided $K_4$ subgraph $\Lambda$ of $\Gamma$ (with $\alpha \subset \Lambda$) and elements $c_1, c_2 \in C$ which lie on distinct branches of $\Lambda$.  Let $\beta_1$ and $\beta_2$ be the branches of $\Lambda$ containing $c_1$ and $c_2$ respectively.  Without loss of generality, we may assume $c_1$ is an interior vertex of the branch $\beta_1$.  
We now apply a similar argument as in the case that $C$ fails (A1), with 
$\tau$ taken to be a cycle in $\Lambda$ which contains $\beta_1$ but not $\beta_2$, and $x$ and $y$ the endpoints of the branch $\beta_1$.  
This completes the proof of the first special case.

The second special case we consider is when $\eta$ (respectively $\mu$) is labelled, starting at $\gamma(n)$,
exclusively by the vertices of an induced cycle $\alpha$ (respectively $\beta$), with $\alpha \neq \beta$.   
By Standing Assumptions~\ref{assumptions} and an elementary graph-theoretic argument,
 there exists a sequence of induced cycles 
$\alpha = \sigma_1, \sigma_2, \dots, \sigma_k=\beta$ such that 
every consecutive pair of cycles intersects in at least a pair of non-adjacent vertices.  
We claim that the $\sigma_i$ can be chosen so that 
for all $1 \le i<k$, the intersection of $\sigma_i$ and $\sigma_{i+1}$ generates a group with infinitely many ends (that is, strictly speaking, this intersection is the defining graph of a special subgroup with infinitely many ends).  
 
 To see this claim, observe that if the intersection of consecutive cycles generates a 
 2-ended group, then the intersection is exactly a pair of non-adjacent vertices or a pair of adjacent edges. 
 For each pair 
$\sigma_i, \sigma_{i+1}$ which intersect in a pair of non-adjacent vertices, add an extra cycle between $\sigma_i$ and $\sigma_{i+1}$ consisting of one arc from $\sigma_i$ connecting these vertices and one from  $\sigma_{i+1}$.
Then the intersection of this new cycle with $\sigma_i$ or $\sigma_{i+1}$ either generates a group with infinitely many ends,  or consists of a pair of adjacent edges.  Thus after reindexing we have a sequence in which we only need to deal with the latter case.  Now if $\sigma_i$ and $\sigma_{i+1}$ intersect in a subpath $\epsilon$ consisting of a pair of edges, then adding the cycle obtained by deleting the interior of $\epsilon$ from $\sigma_i
\cup \sigma_{i+1}$ results in the desired intersections between successive cycles in the sequence. 
 
Assume that $\sigma_1, \dots, \sigma_k$ have been chosen to satisfy the above claim, and let $\Sigma_i$ be the subcomplexes corresponding to $\sigma_i$ based at $\gamma(n)$. 
Now for all 
$1<i<k$ choose a geodesic ray $\eta_i$ in $\cC$ based at $e$,
 which agrees with $\gamma$ for the first $n$ edges, lies in $\Sigma_i \cap \Sigma_{i+1}$ beyond $\gamma(n)$, and is not 
 equal to $\gamma$ (this last criterion is possible because by construction $\Sigma_i \cap \Sigma_{i+1}$ contains 
 a tree with infinitely many ends).  
 Finally, define $\eta_1 = \eta$ and $ \eta_k = \mu$. 
 Observe that $\eta_i^+ \in U_n \setminus \{\gamma^+\}$ for all $i$, 
and moreover, $\eta_i^+$ and $\eta_{i+1}^+$ satisfy the hypotheses of the first special case
above.  Now we can construct the desired path between $\eta^+$ and $\mu^+$ 
by concatenating the paths obtained above between $\eta_i^+$ and $\eta_{i+1}^+$ for all $i$. 

In the general case, we consider arbitrary $\mu_1^+, \mu_2^+
\in U_n \setminus \{\gamma^+\}$.  Write $\mu^+$ for either $\mu_1^+$ or $\mu_2^+$.  It is enough to construct a path in $U_n \setminus \{\gamma^+\}$ from $\mu^+$ to some $\eta^+ \in U_n \setminus \{\gamma^+\}$, where $\eta$ is labelled, starting at $\gamma(n)$, exclusively by vertices of an induced cycle $\alpha$.  
Suppose $\mu$ is labelled, starting at $\gamma(n)$, by $w_1w_2w_3\dots $ such that
$w_i$ is a word in the letters contained in some cycle $\beta_i$ for all $i$.
We define geodesic rays $\eta_i$ interpolating between $\eta$ and $\mu$ as follows: $\eta_0=\eta$ and for $i>0$,  $\eta_i$ agrees with $\gamma$ until $\gamma(n)$, and is 
labelled $w_1w_2\dots w_i u_i$ beyond $\gamma(n)$, where $u_i$ is an infinite word in $\beta_i$  chosen so that the result is a geodesic ray not equal to $\gamma$.   Then $\eta_i^+ \in U_n
\setminus \{\gamma^+\}$ for all $i$. Beyond $\gamma(n)$, the geodesics $\eta_i$ and $\eta_{i+1}$ agree on a segment labelled $w_1w_2\dots w_i$, after which $\eta_i$ is labelled by $\beta_i$ and $\eta_{i+1}$ is labelled by $\beta_{i+1}$.
An argument similar to the second special case (with an appropriate change of base point) shows that we can construct a path in $U_n \setminus \gamma^+$ between $\eta_i^+$ and  $\eta_{i+1}^+$ for all $i$. This completes the proof of the proposition, as concatenating the paths obtained between $\eta_i^+$ and $\eta_{i+1}^+$
for each $i$ results in a continuous path in $U_n \setminus \{\gamma^+\}$ between $\eta^+= \eta_0^+$ and $\mu^+$.
\end{proof}

\begin{example}\label{eg:sim and approx}  We can now describe the $\sim$-classes for Figures~\ref{fig:approx-pairs} and~\ref{fig:sim-eg}, and the $\approx$-pairs to which they are adjacent in the pre-tree $T$ (see also Example~\ref{eg:sim} and Remark~\ref{rem:approx-sim edges}).

First consider Figure~\ref{fig:approx-pairs}.  In the left-hand graph, there are three $W$-orbits of infinite $\sim$-classes, corresponding to the three branches between $a$ and $b$ which are of length at least three.  In the right-hand graph, there are four $W$-orbits of infinite $\sim$-classes, corresponding to the four branches between $a$ and $b$.   In both cases, there are no $\sim$-pairs, all infinite $\sim$-classes are adjacent in $T$ to $\approx$-pairs, and all $\approx$-pairs correspond to $(a,b)$-bicoloured geodesics.

On the left of Figure~\ref{fig:sim-eg}, the $W$-orbits of infinite $\sim$-classes correspond to the three branches between $a$ and $d$, the two branches between $a$ and $e$, and the set $\{a,d,f,g,e\}$, and there are no $\sim$-pairs.  The infinite $\sim$-classes for branches between $a$ and $d$ (respectively, $a$ and $e$) are adjacent in $T$ to $\approx$-pairs of valence $4$ corresponding to $(a,d)$-bicoloured geodesics (respectively, valence $3$ corresponding to $(a,e)$-bicoloured geodesics), while the infinite $\sim$-classes for the set $\{a,d,f,g,e\}$ are adjacent to $\approx$-pairs of both kinds.  

In the centre of Figure~\ref{fig:sim-eg}, there is a $W$-orbit of infinite $\sim$-classes for each of the five branches of length three, and a $W$-orbit of infinite $\sim$-classes corresponding to the set $\{a_1,a_2,a_3,a_4,a_5\}$.  There are no $\sim$-pairs.  All $\approx$-pairs correspond to $(a_i,a_{i+1})$-bicoloured geodesics.  For each $i$, the infinite $\sim$-classes corresponding to the branch between $a_i$ and $a_{i+1}$ are adjacent in $T$ to $\approx$-pairs of valence $4$ corresponding to $(a_i,a_{i+1})$-bicoloured geodesics.  The infinite $\sim$-classes for $\{a_1,a_2,a_3,a_4,a_5\}$ are adjacent in $T$ to all five kinds of $\approx$-pairs.  

On the right of Figure~\ref{fig:sim-eg}, there is a $W$-orbit of infinite $\sim$-classes corresponding to each branch of length three, as well as three $W$-orbits of $\sim$-pairs, corresponding to the set $\{p,q\}$, the branch between $u$ and $v$, and the branch between $r$ and $s$.  There are no $\approx$-pairs.
\end{example}

The following results and observations concerning sets $A$ which correspond to $\sim$-classes will be used to identify the stabilisers of $\sim$-classes, and in later sections.  In Lemma~\ref{lem:A branch} we prove that if $A$ satisfies (A1)--(A3) and $A$ contains an interior vertex of a branch of $\G$, then $A$ contains all vertices of this branch.  This is used to establish Corollary~\ref{cor:A ess cut pair}, which says that $A$ contains an essential cut pair $\{a,b\}$ such that $a$ and $b$ are consecutive in the cyclic order on $A$.  Remark~\ref{rmk:components} then records some implications for $\Sigma$.

\begin{lemma}\label{lem:A branch}  Let $A$ be a set of vertices of $\G$ satisfying properties \emph{(A1)}, \emph{(A2)} and \emph{(A3)}, and such that $\langle A \rangle$ is infinite.  Suppose an element $a \in A$ is an interior vertex of a branch of $\G$.  Then $A$ contains all vertices of this branch, including its endpoints.
\end{lemma}

\begin{proof}  Let $\alpha$ be an induced cycle in $\G$ containing all elements of $A$.
Let $\beta$ be the branch of $\G$ containing $a$ and let $b$ and $b'$ be the essential vertices of $\G$ which are its endpoints.

First consider a non-essential vertex $c \neq a$ which also lies between $b$ and $b'$ on $\alpha$.  Then $\{a,c\}$ separates $|\G|$, and for all $a' \in A \setminus \{a\}$, the pair $\{a,a'\}$ separates $|\G|$ if and only if the pair $\{c,a'\}$ separates $|\G|$.  Thus (A1) holds for $A \cup \{c\}$.  

Now suppose there is a subgraph $\Lambda$ of $\Gamma$ which is a subdivided $K_4$ and contains at least three vertices of $A \cup \{c\}$.  Then $\Lambda$ contains at least three vertices of $A$, or $\Lambda$ contains $c$ (these cases are not mutually exclusive).  
In the former case, 
by (A1) all vertices of $A$ lie on the same branch of $\Lambda$.  Since $a$ and $c$ are both interior vertices of $\beta$, the vertex $c$ must lie on this branch of $\Lambda$ as well, and so in this case (A2) holds for $A \cup \{c \}$.  

In the case that  $\Lambda$ contains $c$, we claim that all elements of $A$ lie on the same branch of $\Lambda$ as $c$.  Since $c$ is non-essential, $\Lambda$ must contain the entire branch $\beta$, and so $\Lambda$ contains $a$ as well.  Now suppose there exists $a' \in A$ which does not lie on the branch of $\Lambda$ containing $a$ and $c$.  Then $\Lambda \setminus \{a,a'\}$ is connected but $\{a,a'\}$ separates $|\G|$, so there is a path between $a$ and $a'$ which intersects $\Lambda$ only at $a$, and possibly $a'$, which contradicts $a$ being non-essential. This proves the claim.  Thus (A2) holds for $A \cup \{ c\}$ in this case as well. We have shown that both (A1) and (A2) hold for the set $A \cup \{c\}$, and so by (A3) we conclude that $c \in A$.

To show that $b$ and $b'$ are both in $A$, notice first that $\{b,b'\}$ separates $|\G|$, and for all vertices $c$ in the interior of the branch $\beta$, the pairs $\{b,c\}$ and $\{b',c\}$ separate $|\G|$.  Consider the case that every vertex of $A$ lies on the branch $\beta$.  Then (A1) holds for $A \cup \{b,b'\}$ by the previous observations.  For (A2), if any subgraph $\Lambda$ which is a subdivided copy of $K_4$ contains a triple of vertices of $A$ then $\Lambda$ must contain a non-essential vertex of $\beta$.  It follows that a branch of $\Lambda$ contains all of $\beta$, and so all vertices of $A \cup \{b,b'\}$ lie on the same branch of $\Lambda$.  Thus (A2) also holds for $A \cup \{b,b'\}$ in this case, and so by (A3) we have $b,b' \in A$.  

Now suppose there is some $a' \in A$ which does not lie on the branch $\beta$.  If $\{a',b\}$ does not separate $|\G|$ then $a'$ and $b$ are non-adjacent and there must be a path $\eta$ in $\G$ which connects the two components of $\alpha \setminus \{a',b\}$.  
Since $\{a,a'\}$ separates $|\G|$, there is a path between $a$ and $a'$ which intersects $\alpha \cup \eta$ only at its endpoints.  This contradicts the fact that $a$ is not essential.  So $\{a',b\}$, and similarly $\{a',b'\}$, separates $|\G|$.  Thus $A \cup \{b,b'\}$ satisfies (A1).  

For (A2), to avoid trivial cases we may assume by contradiction that $\G$ has a subdivided $K_4$ subgraph $\Lambda$ which contains three vertices of $A \cup \{b,b'\}$, so that for some vertex $a'$ of $A$ in this triple, $a'$ and $b$ lie in different branches of $\Lambda$.   If $b'$ is also on $\Lambda$, then $\Lambda \cup \beta \setminus \{a, a'\}$ is connected, and as before, this contradicts the fact that $a$ is not essential.  Finally, if $b'$ is not on $\Lambda$, then there is another vertex $a''$ of $A$ which is on $\Lambda$.  If $a''$ is between $a'$ and $b$ 
on $\alpha \setminus \beta$, then $\Lambda \cup \alpha \setminus \{a, a''\}$ is connected, and by the same argument as before, this contradicts the fact that $a$ is not essential.  If $a'$ is between $a''$ and $b$ on $\alpha \setminus \beta$, then using $a'$ instead of $a''$ in the previous sentence, we again have a contradiction.  This completes every case, and shows that $A \cup \{b, b'\}$ satisfies (A2).  By (A3) it follows that $b, b' \in A$ in this case as well. 
\end{proof}

\begin{cor}\label{cor:A ess cut pair}   Let $A$ be a set of vertices of $\G$ satisfying properties \emph{(A1)}, \emph{(A2)} and \emph{(A3)}, and such that $\langle A \rangle$ is infinite.  Then $A$ contains an essential cut pair $\{a,b\}$ such that $a$ and $b$ are consecutive in the cyclic order on $A$.
\end{cor}

\begin{proof}  We first show that $A$ contains an essential cut pair.  If $A$ contains an interior vertex of a branch of $\G$, then by Lemma~\ref{lem:A branch}, $A$ contains the endpoints of this branch.  These are an essential cut pair.  Otherwise, $A$ consists only of essential vertices.  Then as $\langle A \rangle$ is infinite  $A$ must contain a pair of non-adjacent essential vertices, which by (A1) are a cut pair.

If $\card(A) = 2$ or $3$, let $\{a,b\}$ be an essential cut pair in $A$.  It is then immediate that $a$ and $b$ are consecutive in any cyclic order on $A$.   Now assume that $\card(A) = n \geq 4$.  Let $\alpha$ be an induced cycle containing all of $A$, and inducing the cyclic order $a_1,\ldots,a_n$ on $A$.  

 We first show that $\alpha$ must contain a vertex which is not in $A$.  If every vertex of $\alpha$ is in $A$, then as $\G$ is not a cycle and has no separating vertices or edges, there is a path $\eta$ which connects a non-adjacent pair $a_i$ and $a_j$ in $A$, so that $\eta$ intersects the cycle $\alpha$ only at $a_i$ and $a_j$.  By slightly extending $\eta$ to the midpoints of edges of $\alpha$ incident to $a_i$ and $a_j$, we obtain a path which contradicts Lemma~\ref{lem:A1A2xy}(2).  Thus 
 $\alpha$ must contain a vertex which is not in $A$.
 
We now have that $\alpha$ contains a vertex $c \notin A$.  Then without loss of generality $c$ lies on the subpath of $\alpha$ between $a_1$ and $a_2$ which has no other elements of $A$.  If $a_1$ is non-essential then by Lemma~\ref{lem:A branch}, $A$ and thus $\alpha$ contains all vertices of the branch on which $a_1$ lies.  Using Lemma~\ref{lem:A branch} again, this contradicts either $c \not \in A$ or there being no vertex of $A$ between $c$ and $a_1$ on $\alpha$.  Thus $a_1$ is essential, and similarly $a_2$ is essential.  By (A1), it now suffices to show that $a_1$ and $a_2$ are not adjacent in $\G$.  If there is an edge $\epsilon$ of $\G$ with endpoints $a_1$ and $a_2$, then since 
$\epsilon$ doesn't separate $\G$, there is a path $\eta$ connecting the two components of 
$\alpha \setminus \{a_1, a_2\}$ which meets $\alpha$ only at its endpoints.  Now $\alpha\cup \eta \cup \epsilon$ is a subdivided $K_4$ which does not contain all elements of $A$ on a single branch.  
This contradicts (A2).
Therefore $a_1$ and $a_2$ are not adjacent in $\G$, and so $\{a_1,a_2\}$ is the desired essential cut pair.
\end{proof}

\begin{remark}\label{rmk:components}
Let $A$ be a set of vertices of $\G$ with $\langle A \rangle$ infinite and satisfying properties (A1),  (A2) and (A3). Write $K_A$ for the Davis complex chamber for the special subgroup $\langle A \rangle$.  Developing this chamber and then imposing the cellulation by big squares gives a subcomplex $\Sigma_A$ of $\Sigma$ with $1$-skeleton the copy of the Cayley graph $\cC_A$ of $A$ which contains the identity.  For example, in Figure~\ref{fig:sim-eg}, if $A = \{a,b,c,d\}$ then $\Sigma_A$ contains big squares with edges labelled by commuting generators, and so $\Sigma_A$ properly contains $\cC_A$, while if $A = \{a_1, a_2, a_3, a_4, a_5\}$ then $K_A$ is a star graph of valence $5$, and $\cC_A = \Sigma_A$ is a tree.  Each coset of $\langle A \rangle$ in $W$ also corresponds to some copy of $\Sigma_A$ and to some copy of $\cC_A$, with $\cC_A$ the $1$-skeleton of $\Sigma_A$. 

Now fix a copy of $\Sigma_A$, and let  $\alpha$ be an induced cycle in $\G$ containing all vertices of $A$.  We define $\cAplus$ to be the set of geodesics in $\cC_A \subseteq \Sigma_A$ which are bicoloured by essential cut pairs $\{a,b\}$ in $A$ so that $a$ and $b$ are consecutive in the cyclic order on $A$.  By Corollary~\ref{cor:A ess cut pair}, the set $\cAplus$ is nonempty.  Let $\gamma$ be a geodesic in $\cAplus$, with $\gamma$ bicoloured by $a$ and $b$.  Then by definition of $\cAplus$, if $A \setminus \{a,b\}$ is nonempty then all vertices of $A \setminus \{a,b\}$ lie in the same component of $\alpha \setminus \{a,b\}$, and thus in the same component of $\G \setminus \{a,b\}$.  It  follows from Lemma~\ref{lem:separation} that $\Sigma_A \setminus \gamma$, if nonempty, is contained in a single component of $\Sigma \setminus \gamma$.   We thus refer to the elements of $\cAplus$ as the \emph{frontier geodesics} of $\Sigma_A$.  Now observe that for any component $U$ of $\Sigma \setminus \Sigma_A$, there is a unique frontier geodesic $\gamma \in \cAplus$ such that $U$ is a component of $\Sigma \setminus \gamma$.  We then say that $\gamma$ is the frontier geodesic \emph{corresponding} to the component $U$ of $\Sigma \setminus \Sigma_A$.  \end{remark}

We finish this section by determining the stabilisers of $\sim$-classes.

\begin{prop}\label{prop:sim stabiliser}   Let $v$ be a $\sim$-class in $T$, and let $A$ be the corresponding set of vertices of $\Gamma$ as in Proposition~\ref{prop:sim}, so that $v = \partial \cC_A \setminus \partial \cA$ for some copy $\cC_A$ of the Cayley graph of~$\langle A \rangle$.
\begin{enumerate}
\item  If $v$ is a $\sim$-pair, let $\gamma$ be an $(a,b)$-bicoloured geodesic such that $v = \partial \gamma = \partial \cC_A$, hence $a,b\in A$.  Then the stabiliser of $v$ is a conjugate of $\langle a,b \rangle$, if there is no vertex $c$ of $\G$ adjacent to both $a$ and $b$, or of $\langle a, b,c \rangle$, if there is such a vertex $c$. 
\item If $v$ is an infinite $\sim$-class, then the stabiliser of $v$ is a conjugate of $\langle A \rangle$.
\end{enumerate}
\end{prop}

\begin{proof}  The proof of (1) is similar to that for stabilisers of $\approx$-pairs in Lemma~\ref{lem:approx stabiliser}.  Now suppose $v$ is an infinite $\sim$-class, so $\langle A \rangle$ is not $2$-ended, and let $\cC_A$ be the copy of the Cayley graph of $\langle A \rangle$ which contains the identity.  It suffices to show that the stabiliser of $v$ is $\langle A \rangle$. 

Suppose that some $g \in W \setminus \langle A \rangle$ stabilises $v$.  Then $g$ stabilises $\partial \cC_A = \partial \Sigma_A$ and the set of endpoints of the frontier geodesics $\cAplus$, since these bound components of $\Sigma \setminus \Sigma_A$.  Let $\gamma$ be an $(a,b)$-bicoloured geodesic in $\cAplus$ which passes through the identity.  Then $g\gamma$ is not in $\cAplus$ since $g \not \in \langle A\rangle$.  Hence there is a geodesic $\gamma' \in \cAplus$ so that $\gamma'$ and $g\gamma$ have the same endpoints.  By Corollary~\ref{cor:parallel}, $\gamma'$ is also $(a,b)$-bicoloured, and $g\gamma =  \gamma' c$ where $c$ is a   (unique) vertex of $\G$ adjacent to both $a$ and $b$.  If there is no such vertex $c$ we are done, otherwise let $w \in \langle A \rangle$ be the label on a shortest path from $e$ to $\gamma'$, so that $\gamma' = w\gamma$.  Then $g\gamma = w\gamma c$.  Since $\gamma$ passes through the identity $\gamma c = c\gamma$, so $g\gamma = wc \gamma$.  As the stabiliser of $\gamma$ is $\langle a, b \rangle$, it follows that $g = wcz$ where $w \in \langle A \rangle$ and $z \in \langle a, b \rangle \leq \langle A \rangle$.  Since we assumed $g \not \in \langle A \rangle$, we have $c \not \in A$.  However $g$, $w$ and $z$ all stabilise $v$, so $c$ must as well.  

Now let $a' \in A \setminus \{a,b\}$ be such that $a'$ does not commute with both $a$ and $b$ (since $\langle A \rangle$ is not $2$-ended, such an $a'$ exists).  Using properties (A1) and (A2), we see that $a'$ cannot be adjacent to $c$.  Consider the geodesic $ca'\gamma$.  Since $a' \gamma \in \cAplus$, $c  \not \in A$ and $c$ stabilises $v$, by similar reasoning to the previous paragraph there is a $\gamma'' \in \cAplus$ with the same endpoints as $ca'\gamma$, and a vertex $x$ of $\Gamma$ which is adjacent to both $a$ and $b$ so that $ca'\gamma = \gamma''x$.  Now $\Gamma$ has no squares, so $x = c$ and thus $ca' \gamma = \gamma'' c$.  As $\gamma'' \in \cAplus$ and is $(a,b)$-bicoloured, we can write $\gamma'' = w' \gamma$ where $w' \in \langle A \rangle$ is the label on a shortest path from $\gamma$ to $\gamma''$.  Then $ca' \gamma = w' \gamma c = w' c \gamma$, hence $c (w')^{-1}ca' \in \langle a, b\rangle$.  It follows that $w'$ contains an instance of $a'$, and every letter in $w'$ commutes with $c$.  But $a'$ does not commute with $c$, so we have obtained a contradiction.  We conclude that the stabiliser of $v$ is $\langle A \rangle$.
\end{proof}

%%%%%%%%%%%%%%%%%%%%%%%%%%%%%%%%%%
\subsection{Construction of certain stars}\label{sec:stars}
%%%%%%%%%%%%%%%%%%%%%%%%%%%%%%%%%%

Recall that a star is a maximal set $X$ of vertices in the pretree $T$ such that given any two vertices in $X$, no vertex in $T$  is between them, in the sense of Definition~\ref{def:between}.  Bowditch shows that stars of size at least 3 are in fact infinite.  In Proposition~\ref{prop:stars} below, we construct certain infinite stars.  We will show in Section~\ref{sec:all-stars} that we have identified all of the stars of size at least 3, and therefore all of the Type 3 vertices.

\begin{prop}[Stars of size at least 3]\label{prop:stars}
Let $B$ be a set of essential vertices of $\G$ satisfying the following properties:
\begin{enumerate}
\item[\emph{(B1)}] if $C = \{c_1, c_2\}$ is any pair of essential vertices of $\G$ (which possibly intersects $B$), then $B \setminus  C$ is contained in a single component of $\G \setminus C$; 
\item[\emph{(B2)}] the set $B$ is maximal with respect to \emph{(B1)}; and
\item[\emph{(B3)}] $\card( B ) \ge 4$.
\end{enumerate}
Then $B$ corresponds to a $W$-orbit of stars of size at least 3 as follows: let $\mathcal{G}_B$ denote the set of $(b_i, b_j)$-bicoloured geodesics in $\mathcal C$ passing through 
$e$, 
where  $\{ b_i, b_j\} \subseteq  B$ is a cut pair. 
Let $\cB$ be the $\langle B \rangle$-orbit of $\mathcal{G}_B$.  Then for each $\eta \in \cB$, the endpoints $\partial \eta$ are contained in a unique $\approx$-pair or $\sim$-class, and this containment induces a map $\Phi:\cB \to T$ so that $\Phi(\cB)$ corresponds to a star of size at least 3 in $T$.
\end{prop}

\begin{example}  
In Figure~\ref{fig:sim-eg},  the set $\{a, d, e\}$ in the left-most graph satisfies (B1) and (B2) 
but not (B3), and since there are only three essential vertices condition (B3) will never be satisfied.  Any subset of $\{a_1, a_2, a_3, a_4, a_5\}$ in the middle graph which contains at least $4$ vertices fails (B1).  The sets $\{u, v, p, q\}$ and  $\{p, q, r, s\}$ in the right-most graph satisfy 
(B1), (B2) and (B3), and each pair of vertices within these sets is a cut pair in $\G$. \end{example}

We begin with some preliminaries for the proof of Proposition~\ref{prop:stars}.

\begin{obs}\label{obs:B}
Let $B$ be a set of essential vertices of $\G$ which satisfies property (B1) of Proposition~\ref{prop:stars}, and  $y$ be an essential vertex of $\G$.  If there exist paths $\eta, \mu$ and $\nu$ in $\G$ from $y$ to $B$ which intersect only at $y$, then $B \cup \{y\}$ satisfies property (B1).  This is because given any essential vertices $c_1, c_2 \neq y$,  at least one of the paths $\eta, \mu$ and $\nu$ misses $\{c_1, c_2\}$, so that $y$ is in the same component as $B \setminus \{c_1, c_2\}$ in $\G \setminus \{c_1, c_2\}$. 
\end{obs}

\begin{lemma}\label{lem:star-sep}
Let $B$ be a set of essential vertices satisfying \emph{(B1)}, \emph{(B2)} and \emph{(B3)}.  Then there exists a cut pair in $\G$ consisting  
of vertices $\{b, b'\}\subset B$. 
\end{lemma}

\begin{proof}
By items (4) and (5) in Standing Assumptions~\ref{assumptions},  we have that $\G$ is not a cycle, and has at least one cut pair.  
Since we may replace each vertex in the pair with a nearest 
essential vertex if necessary, it follows that $\G$ has at least one 
essential cut pair.  
Thus if $B$ contains all the essential vertices of $\G$, the lemma follows.

We may now assume that $\G \setminus B$ contains an essential vertex.  
Let $x$ be an essential vertex in $\G \setminus B$ ``nearest to $B$'' in the sense that there exists 
$b \in B$ and a path $\tau$ in $\G$ between $x$ and $b$ which does not contain any essential vertices in its interior.  
Since $x \notin B$, 
there exists, by property (B2),  a pair of essential vertices which separate $x$ from $B$.  One of the vertices in the pair must be $b$, for otherwise $\tau$ would connect $x$ to $B$. 
Let $c$ be the other vertex, so that $\G \setminus \{b, c\}$
contains disjoint components $\Lambda_x$ and $\Lambda_B$, with $x \in \Lambda_x$ and $B \setminus \{b, c\} \subset \Lambda_B$.  
If $c \in B$, then put $b'=c$ to obtain the 
desired cut pair 
$\{b, b'\} \subset B$.  

If $c \notin B$, 
choose a reduced path $\eta$ from $c$ to $B$ through $\Lambda_B$, which meets $B$ only at its endpoint
$b' \in B$. 
Since $b\notin \Lambda_B$, we know that $b \neq b'$.  We claim that $\{b, b'\}$ is the desired pair.  
If every path from $c$ to $B$ enters through either $b$ or $b'$, then clearly $\{b, b'\}$ separates $\G$ 
(putting $c$ and $B \setminus \{b, b'\}$ in different components).  
Assume, therefore, that there is a reduced path $\nu$ from $c$ to $B$, which intersects $B$ only in its endpoint $b'' \neq b, b'$.  It follows that $\nu$ does not intersect $\Lambda_x$.   

There exists a path $\mu$ from $c$ to $b$ in $\Lambda_x$ which is disjoint from $\eta$ and $\nu$.
Since $c \notin B$, Observation~\ref{obs:B} implies that the paths $\eta$ and $\nu$ must intersect.
Let $y$ be the last point of $\eta \cap \nu$ encountered while 
traveling along $\eta$ from $c$ to $b'$ (with $y$ possibly equal to $c$), so that the segment $\eta_{[yb']}$ meets $\nu$ only at $y$, as in Figure~\ref{fig:star-sep}. 
\begin{figure}[ht]
\begin{center}
\begin{overpic}%[grid, tics=20]
{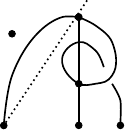}
\put(0, -15){\footnotesize $b$}
\put(60,-15){\footnotesize $b'$}
\put(85,-15){\footnotesize $b''$}
\put(63,41){\footnotesize $y$}
\put(66,89){\footnotesize $c$}
\put(10,80){\footnotesize $x$}
\put(32,95){\footnotesize $\Lambda_x$}
\put(80,95){\footnotesize $\Lambda_B$}
\put(50,15){\footnotesize $\eta$}
\put(-8,15){\footnotesize $\mu$}
\put(97,15){\footnotesize $\nu$}
\end{overpic}
\caption{{\footnotesize }} 
\label{fig:star-sep}
\end{center}
\end{figure}

We now have three paths from $y$ to $B$ which meet only at $y$: $\eta_{[yb']}, \nu_{[yb'']}$, and the concatenation of $\nu_{[yc]}$ with $\mu = \mu_{[cb]}$.  Then by Observation~\ref{obs:B}, we have 
$y \in B$, which contradicts our choice of $\eta$.  It follows that $\{b, b'\}$ separates $\G$. 
\end{proof}

\begin{lemma}\label{lem:cycle-k4}
Let $B$ be a set of essential vertices satisfying \emph{(B1)}, \emph{(B2)} and \emph{(B3)}.  Suppose that $B$ contains three distinct vertices $b_1$, $b_2$ and $b_3$ which lie on an induced cycle $\sigma$ in $\G$.  Then $b_1$, $b_2$ and $b_3$ lie on a subdivided $K_4$ subgraph 
of $\G$.  Moreover, the set $\{b_1, b_2, b_3\}$ intersects at least two branches of this subgraph. 
\end{lemma}

\begin{proof}
By property (B3), $\card(B) \ge 4$, so there is $b_4 \in B$ distinct from $b_1$, $b_2$ and $b_3$.   

If $b_4$ lies on $\sigma$, assume without loss of generality that $b_1$ and $b_3$ separate $\sigma$ into two components, one containing $b_2$ and the other containing $b_4$.  By property (B1), $b_2$ and $b_4$ lie in the same component of $\G \setminus \{b_1, b_3\}$, so there is a path $\eta$ which connects the two components of $\sigma \setminus \{b_1,b_3\}$.  We may assume $\eta$ is reduced and meets $\sigma$ only in a pair of essential vertices $x$ and $y$.  Now $b_1$ and $b_3$ lie in different components of $(\sigma \cup \eta) \setminus \{x, y\}$, but by  property (B1), they lie in the same component of $\G \setminus \{x, y\}$.  Thus there is a reduced path $\tau$ connecting the component of $\sigma \setminus \{x, y\}$ containing $b_1$ with that containing $b_3$.  
Then if $\tau$ and $\eta$ are disjoint, $\sigma \cup \tau \cup \eta$ is a subdivided $K_4$ subgraph containing $b_1$, $b_2$ and $b_3$, and otherwise $\sigma \cup \tau \cup \eta$ contains such a subgraph.  It is clear from the construction that $b_1$, $b_2$, and $b_3$ lie on at least two distinct branches of this subgraph.

On the other hand, if $b_4$ is not on $\sigma$ then since $\G$ is connected, there is at least one (reduced) path $\tau$ connecting $b_4$ to $\sigma$.  If $\tau$ first meets $\sigma$ at $z \neq b_1, b_2, b_3$, then  by Observation~\ref{obs:B} 
and property (B2) 
we conclude that $z\in B$, and apply the previous paragraph with $b_4=z$.  So we may assume that every reduced path from $b_4$ to $\tau$ first meets $\sigma$ at $b_1, b_2$ or $b_3$.  Now using the same argument as in the preceding paragraph, we conclude that for $1 \le i \le 3$, there is a reduced path $\tau_i$ connecting $b_4$ to $b_i$ in the complement of $b_{i+1}$ and $b_{i+2}$ (indices taken mod 3).   It is easy to see that 
$\sigma \cup \tau_1 \cup \tau_2 \cup \tau_3$ either is or contains a subdivided $K_4$ subgraph.  
Again, it is also clear that $b_1$, $b_2$, and $b_3$ lie on at least two distinct branches of this subgraph.
\end{proof}

\begin{lemma}\label{lem:mhf-cap-star}
Let $A$ and $B$ be sets of vertices of $\G$ as in the statements of Propositions~\ref{prop:sim} 
and~\ref{prop:stars}
respectively.  Then $\card(A\cap B) \le 2$.  
\end{lemma}

\begin{proof}
Suppose there are three distinct elements $b_1$, $b_2$ and $b_3$ in $A \cap B$.   By 
Lemma~\ref{lem:A1 cyclic} there is an induced cycle $\alpha$ containing $b_1$, $b_2$ and $b_3$.
Then by Lemma~\ref{lem:cycle-k4}, $b_1$, $b_2$ and $b_3$ lie on at least two distinct branches of a subdivided $K_4$ subgraph 
of $\G$. 
This is a contradiction, since $A$ satisfies property (A2) of Proposition~\ref{prop:sim}. 
\end{proof}

\begin{lemma}\label{lem:z-sep}
If
$B$ is a set of essential vertices satisfying (B1) and (B2) with at least two elements, and if $z$ is an arbitrary vertex in $\G\setminus B$, then (at least) one of the following is true:
\begin{enumerate}
\item  
there exists a pair of essential vertices $\{c_1, c_2\}$ and an element $b\in B$ such that $z$ and $b$ lie in different components of $\G \setminus \{c_1, c_2\}$; or

\item  
the vertex $z$ is not essential and lies on a branch between a pair of vertices $b_1$ and $b_2$ in $B$. 
\end{enumerate}

\end{lemma}

\begin{proof}
If $z$ is essential, then (1) follows easily from (B1) and (B2).  Otherwise $z$ lies on a branch between two essential vertices $c_1$ and $c_2$.  
If $\{c_1, c_2\} \subseteq B$ then (2) holds.  If not, then since $\card(B) \geq 2$, there exists $b \in B $ such that $b \notin  \{c_1, c_2\}$, and (1) holds, because $b$ is necessarily essential, and so cannot lie on the branch between $c_1$ and $c_2$ containing $z$. 
\end{proof}

We are now ready to prove Proposition~\ref{prop:stars}.

\begin{proof}[Proof of Proposition~\ref{prop:stars}]  
Given $\eta\in \cB$, by definition of the sets $\mathcal{G}_B$ and $\cB$ and Lemma~\ref{lem:separation}, the endpoints $\partial \eta$ separate $\partial W$, so that $\partial \eta \subset \partial W(2+)$.  We can thus define a map  $\Phi:\cB \to T$ by letting $\Phi(\eta) $ be the unique  
$\approx$-pair or $\sim$-class containing $\partial \eta$.   This map is well-defined since Bowditch shows that the $\approx$-pairs and $\sim$-classes partition $\partial W(2+)$.

We claim that: 
\begin{itemize}
\item[(a)] 
The map $\Phi:\cB \to T$ satisfies $\card{\Phi^{-1}(v)} \le 2$ for all $v \in T$.
\item[(b)] 
No vertex in $T$ is between any pair of vertices in $\Phi(\cB)$.
\item[(c)] The set $\Phi(\cB)$ is maximal with respect to condition (b).
\end{itemize}
Note that $\mathcal{G}_B$ is nonempty by Lemma~\ref{lem:star-sep}, and since $\card(B) \geq 4$ it follows that $\card{\cB} = \infty$.  
This, together with conditions (a), (b) and (c), implies that $\Phi(\cB)$ corresponds to a star of size at least 3 in $T$.

\bigskip

\emph{Proof of (a):}  Suppose $\Phi(\eta)=v \in T$ where $\eta \in \cB$ is labelled by $b_1$ and $b_2$. 
If $v$ is a $\approx$-pair or a $\sim$-pair, and $\eta' \in \cB$ with $\Phi(\eta') = v$, then $\partial\eta = \partial \eta'$.  As both $\eta$ and $\eta'$ are bicoloured, it follows from Corollary~\ref{cor:parallel} that there is at most one possibility for $\eta'$ distinct from $\eta$.  

Now suppose $v$ is an infinite $\sim$-class, and $\Phi(\eta')=v$.  
By left-multiplying by an element of $\langle B \rangle$ if necessary, we may assume that 
$\eta$ passes through the identity, and that $\eta'= w \gamma $, where $w\in \langle B \rangle$ 
and $\gamma \in \mathcal{G}_B$ passes through the identity.  Since $\partial \eta$ and $\partial\eta' = \partial (w\gamma)$ are part of the same $\sim$-class,  
the supports of $w, \eta_1$ and $\gamma$ are all part of the same set $A$ as in Proposition~\ref{prop:sim} giving rise to this  
$\sim$-class.  Then by Lemma~\ref{lem:mhf-cap-star} the union of these supports contains at most two elements.  It follows that $\eta= \eta'$.   

Thus $\card{\Phi^{-1}(v)} \le 2$ for all $v \in T$.  
\bigskip

\emph{Proof of (b):} 
Let $\eta,\eta'\in \cB$ be such that $\Phi(\eta) = v$, $\Phi(\eta') = v'$ and $v \neq v'$.  Suppose by way of contradiction that there is a vertex $v''$ of $T$ between the vertices $v$ and~$v'$.  

First assume $v''$ is either a $\sim$- or $\approx$-pair. 
Then $v''$ corresponds to $\partial \mu$, where $\mu$ is a 
 $(c_1, c_2)$-bicoloured bi-infinite geodesic in $\cC$,  
so that $\Sigma \setminus \mu$ has at least two components.  
Since $v''$ is between $v=\Phi(\eta)$ and $v' = \Phi(\eta')$, by Remark~\ref{rmk:betweenness} the sets $\partial \eta$ and $\partial \eta'$ are in different components of $\partial W \setminus \partial \mu$.  Thus $\eta$ and $\eta'$ are in different components of $\Sigma \setminus \mu$.  

Assume without loss of generality that $\eta$ passes through $e$, and that $\eta'$ is of the form $w\gamma$, where $w$ is a (possibly empty) reduced word in $B$ labelling a shortest path $\nu$ from $e$ to $\eta'$, and $\gamma$ is a bi-infinite geodesic passing through $e$.  Now $\nu$ necessarily crosses $\mu$, and could potentially travel along it for some distance.  
Let $p$ and $p'$  be the points where it first meets $\mu$ and leaves $\mu$, respectively, and let $u$ (respectively $u'$) be the reduced word labelling the path starting at $p$ (respectively $p'$) and ending at $\eta$ (respectively $\eta'$) along $\nu$.  Then every vertex on $\eta$ (respectively $\eta'$) can be reached from $\mu$ by a path labelled by $u$ (respectively $u'$), followed by an alternating word in $B$.  (This allows the possibility that $u$ or $u'$ is empty.)  By the description of the components of $\Sigma \setminus \mu$ in Lemma~\ref{lem:separation}, we see that all such paths lie in the same component of $\Sigma \setminus \mu$, since the initial letter labelling any such path is in $B$, and by property (B1), 
$B \setminus \{c_1, c_2\}$ lies in a single component of $ \G \setminus \{c_1, c_2\}$.  This is a contradiction.  Thus the vertex $v''$ cannot be either a $\sim$-pair or a $\approx$-pair.

If $v''$ is an infinite $\sim$-class, then it comes from a set $A$ as in Proposition~\ref{prop:sim} so that $\langle A \rangle$ is not $2$-ended.  More precisely, $\cC$ contains a copy $\cC_A$ of the Cayley graph of $\langle A \rangle$ so that $v''$ is the set $\partial \cC_A \setminus \partial \cA = \partial \cC_A \setminus \partial W(3+)$.  As discussed in Remark~\ref{rmk:components}, this copy of $\cC_A$ is contained in a copy of $\Sigma_A$, the subcomplex of $\Sigma$ corresponding to the special subgroup generated by $A$.  Recall from Remark~\ref{rmk:components} that $\cAplus$ denotes the set of frontier geodesics of $\Sigma_A$, and that there is a unique frontier geodesic corresponding to any component of $\Sigma \setminus \Sigma_A$.

We claim that $v''$ being between $v$ and $v'$ implies that either $\eta$ and $\eta' $ are both in
$\cAplus$, or there exists $\mu \in \cAplus$ such that  $\partial \mu$ separates 
$\partial \eta$ from $\partial \eta'$. 
For this, assume that $\eta \notin 
\cAplus$. By \cite[Lemma 3.18]{bowditch} (see Remark~\ref{rmk:betweenness}), $\eta$ then cannot intersect any geodesic in $\cAplus$.  Thus $\eta$ lies entirely in a single component of 
$\Sigma \setminus \Sigma_A$.   If $\eta'$ is also in this component, or $\eta'$ is the frontier geodesic for this component, then no pair of points in 
$\partial \cC_A \setminus \partial W(3+)$
would separate $\partial \eta$ from $\partial \eta'$, contradicting the fact that $v''$ is between $v$ and $v''$.  Thus $\eta$ and $\eta'$ are either in distinct components of $\Sigma \setminus \Sigma_A$, or if $\eta' \in \cAplus$ it corresponds to a different component of $\Sigma \setminus \Sigma_A$ to that containing $\eta$.   Then if $\mu$ is the geodesic in $\cAplus$ corresponding to the component containing $\eta$, it is clear that $\partial \mu$ separates 
$\partial \eta$  from $\partial \eta'$. This proves the claim. 

In the case that there exists $\mu$ in $\cAplus$ such that  $\partial \mu$ separates 
$\partial \eta$ from $\partial \eta'$, the proof proceeds as in the case that $v''$ is in $V_1(\cT)$.  
Otherwise, $\eta$ and $\eta' $ are both in
$\cAplus \subset \cC_A$.  By construction $\eta$ and $\eta'$ are in $\cB\subset \cC_B$.  Since 
$ \cC_A$ (respectively  $\cC_B$) is convex,  the shortest path between $\eta$ and $\eta'$ is labelled by a word in $A$ (respectively $B$).
Thus the support of this path, together with the supports of $\eta$ and $\eta'$, are in $A\cap B$.  But the union of these supports must have at least three elements, which contradicts Lemma~\ref{lem:mhf-cap-star}. 

Thus we have shown that no vertex of $T$ is between $v$ and $v'$.

\bigskip
 
\emph{Proof of (c):}  Consider an arbitrary vertex $v \in T \setminus \Phi(\cB)$.   We wish to show that $\{v \} \cup \Phi(\cB)$ fails (b), 
i.e.~that there is some vertex $v_\eta = \Phi(\eta)$ in $\Phi(\cB)$ such that there exists a vertex $v' \in V_1(\cT) \cup V_2(\cT)$ between $v$ and $v_\eta$.  

Since $v$ is in $T$, there is a bi-infinite geodesic $\mu$ which is bicoloured by a cut pair
$\{x, y\}$ in $\G$, such that $\mu$ separates $\Sigma$ and $\partial \mu \subseteq v \subset \partial W$.
Let $w$ be a (possibly empty) reduced word labelling a shortest path from $e$ to $\mu$, so that every vertex along $
\mu$ in $\cC$ is labelled by a group element consisting of $w$ followed by a word in $x$ and $y$.  
Since $v$ is not in $\Phi(\cB)$, either a letter in $w$ or $x$ or $y$ is not in $B$.  Starting 
from $e$ along the chosen path to $\mu$, let $z$ denote the first letter encountered that is not in $B$.  
Thus either $z$ is some letter used in $w$ or it is one of $x$ and $y$.   Let $u$ denote the 
(possibly empty) subword  of $w$ before the occurrence of $z$.

By Lemma~\ref{lem:z-sep}, there exists an essential cut pair $\{c_1, c_2\}$ (possibly intersecting $B$) which separates $z$ from $B \setminus \{c_1, c_2\}$.   Let $\nu$ be the 
$(c_1, c_2)$-bicoloured geodesic passing through the vertex of $\cC$ labelled $u$ (or $e$ if $u$ is empty).  
By Lemma~\ref{lem:separation},  
$\nu$ separates $\Sigma$, and so $\partial \nu$ is part of the set defining some vertex $v'$ in $T$.  

Now by Lemma~\ref{lem:star-sep}, there exists a cut pair $\{b_1, b_2 \}\subset B$.  We will construct a $(b_1, b_2)$-bicoloured bi-infinite geodesic $\eta$ based at the vertex of $\cC$ labelled $u'$, where $u'$ is chosen so that the length of a shortest path between $\nu$ and $\eta$ is at least $2$, as follows.  
Since $\card(B) \ge 4$, there exist $b_3, b_4 \in B$ distinct from $b_1, b_2$.  If $u$ is empty, then $u' = b_3b_4$, and if $u$ has length $1$, then $u'$ is one of $b_3$ or $b_4$, chosen to be distinct from $ u$ (so that $uu'$ is reduced).  Otherwise, $u'$ is empty (and $\eta$ is based at $e$). 
 It is easy to verify that the shortest path from $\nu$ to $\eta$ passes through $e$ and has 
 $uu'$ as its label. Now $\eta$ defines a vertex $v_\eta = \Phi (\eta)$ corresponding to $\partial \eta$. 
 
 We claim that $v'$ and $v_\eta$ are distinct.  
 Suppose not.  Then if $v' = v_\eta$ is a $\sim$- or $\approx$-pair, the geodesics $\nu$ and $\eta$ are parallel, that is, $\{c_1, c_2\} = \{b_1, b_2\}$ and $uu'$ consists of a single letter, which is a contradiction.   Thus $v' = v_\eta$ is an infinite $\sim$-class, in which case the labels of $\nu, \eta, u, $ and $u'$ are contained in a set $A$ as in Proposition~\ref{prop:sim}.  Now by 
 Lemma~\ref{lem:mhf-cap-star}, we have that $\card(A\cap B) \le 2$, which is a contradiction, since 
 this intersection contains $b_1, b_2, b_3,$ and $ b_4$.  This proves the claim.  

Finally, we observe that $v'$ is between $v_\eta$ (which is in $\Phi(\cB)$)   and $v$.  This 
is a consequence of Definition~\ref{def:between} and Lemma~\ref{lem:separation}, 
 because 
at least one half of $\mu$ can be reached by a path with first letter $z$, and at least one half of $\eta$ can be reached by a path with first letter in $B \setminus\{c_1, c_2\}$.
It follows that $\Phi(\cB)$ is maximal with respect to property (b), as desired.

\medskip

The above shows that the set $\Phi(\cB) \subset V_1(\cT) \cup V_2(\cT)$ defined above yields one star.  By acting on this set on the left by elements of $W$, one obtains a $W$-orbit of stars.  
\end{proof}

%%%%%%%%%%%%%%%%%%%%%%%%%%%%%%%%%%
\subsection{Identification of all stars and their stabilisers}\label{sec:all-stars}
%%%%%%%%%%%%%%%%%%%%%%%%%%%%%%%%%%

In this section we prove that we have identified all of the stars in the JSJ tree $\cT$, and determine their stabilisers.

\begin{prop}\label{prop:all-stars} Every star of size at least 3 in $T$ comes from a set $B$ satisfying conditions \emph{(B1)}, \emph{(B2)} and \emph{(B3)} in the statement of Proposition~\ref{prop:stars}.  
\end{prop}

\begin{proof}
Suppose that $v$ is a vertex of $\cT$ corresponding to a star of size at least 3 in $T$.  Then $v$ is an infinite collection of vertices in $T$.  We now describe how to construct an associated set of geodesics $\cB_v$ by choosing, 
for each vertex in this collection,  a bi-infinite geodesic which is bicoloured by an essential cut pair.

For each $\approx$-pair or $\sim$-pair in the collection $v$, we add to the set $\cB_v$ a (separating) bi-infinite geodesic $\gamma$ whose boundary is this $\approx$-pair or $\sim$-pair. 
We claim that $\gamma$ may be chosen so that the labels of this geodesic are essential vertices.  This is immediate from Lemma~\ref{lem:all approx} for $\approx$-pairs.  For $\sim$-pairs, we have by Corollary~\ref{cor:A ess cut pair} that the corresponding set $A$ contains an essential cut pair.  Now since $\langle A \rangle$ is $2$-ended, there is only one essential cut pair contained in $A$.  If $\card(A) = 2$ the geodesic $\gamma$ must be bicoloured by this cut pair, and if $\card(A) = 3$ then $A = \{a,b,c\}$ with $\{a,b\}$ an essential cut pair and $c$ commuting with both $a$ and $b$; in this case, we may take $\gamma$ to be one of the two geodesics bicoloured by $a$ and $b$ with $\partial \gamma = v$.

Given an infinite $\sim$-class $v'$ in $v$ which comes from a set $A$ as in 
Proposition~\ref{prop:sim}, by Remark~\ref{rmk:components}, for each component of 
$\Sigma\setminus \Sigma_A$ there is a unique corresponding geodesic in the set $\cAplus$.
Now for any $v'' \neq v'$ in the collection $v$, the subset of $\partial W$ corresponding to $v''$ lies in 
a single component of $\Sigma\setminus \Sigma_A$, as shown in the proof of 
Proposition~\ref{prop:stars}(b).  Moreover, there is a unique component $U$ of $\Sigma\setminus \Sigma_A$ which contains the subsets of $\partial W$ corresponding to all of the vertices in the collection $v$ other than $v'$ itself, for otherwise, $v'$ would be between some pair of vertices in $v$, contradicting the fact that $v$ is a star.  Let $\mu_{v'} \in \cAplus$ be the frontier geodesic corresponding to the component $U$.  By the definition of $\cAplus$, 
$\mu$ is bicoloured by essential vertices.  For each such $v'$ in $v$, add $\mu_{v'}$ to the collection 
$\cB_v$. 

Thus for each vertex in the star $v$, the set $\cB_v$ contains a unique separating bi-infinite geodesic bicoloured by a pair of essential vertices.  
We claim that $\cB_v$ arises from the construction given in Proposition~\ref{prop:stars} associated to some set $B$ satisfying (B1), (B2) and (B3).

Let $B$ be the support of the geodesics in the set $\cB_v$, that is the union of all the pairs labelling the geodesics in the set.  We will show that $B$ satisfies (B1), (B2) and (B3).  

If (B1) fails, then there is a pair of vertices $\{c_1, c_2\}$ and $b_1, b_2\in B$ which are in different components of $\G \setminus \{c_1, c_2\}$.  For $i=1, 2, $ let $b_i' \in B$ be such that there is a  
$(b_i, b_i')$-bicoloured geodesic $\gamma_i$ in $\cB_v$. 
Without loss of generality $\gamma_1$ passes through $e$.  Let $w$ be the label on a shortest path from $e$ to $\gamma_2$.  Write $w$ as $w_1 w_2$, where $w_1$ is the maximal initial subword consisting of generators from the same component of $\G \setminus \{ c_1, c_2 \}$ as $b_1$.  (The subwords $w_1$ and/or $w_2$ could be empty.)  Let $\gamma$ be the $(c_1, c_2)$-bicoloured geodesic through the vertex $w_1$.
Then $\gamma$ separates $\Sigma$ by Lemma~\ref{lem:separation}, and since the first letter of $w_2$ (if it is nonempty) or the generator $b_2$ (if $w_2$ is empty) lie in a separate component of $\G \setminus \{c_1,c_2\}$ from $b_1$, we get that at least one endpoint of the geodesic $\gamma_2$ is separated by $\partial \gamma$ from at least one endpoint of the geodesic $\gamma_1$.  Since the endpoints of $\gamma_1$ and $\gamma_2$ are contained in vertices say $v_1$ and $v_2$ of $T$, and $\gamma$ is also contained in a vertex $v_\gamma$ of $T$, this implies that $v_\gamma$ lies between $v_1$ and $v_2$.   (Note that $v_\gamma$ is a distinct vertex from $v_1$ and $v_2$ since our procedure for choosing geodesics in $\cB_v$ involved choosing the geodesic closest to the star in each infinite $\sim$-class.) This contradicts $v_1$ and $v_2$ being in the same star. 
We have shown that no $\{c_1, c_2\}$ as above can exist.

Thus $B$ satisfies property (B1) of Proposition~\ref{prop:stars}.   It is therefore contained in a maximal set say $B'$ satisfying (B1).  Now if we run the procedure from Proposition~\ref{prop:stars} on the set $B$ we will recover all of the pre-tree vertices in $v$.  But $v$ is maximal since it is a star, so $B = B'$ and hence $B$ satisfies property (B2).

To show that (B3) holds, we need to rule out $\card(B) = 2, 3$.
In the first case, suppose $b_1$ and $b_2$ are the only labels of geodesics in $\cB_v$.
Since stars are infinite, there are infinitely many geodesics in $\cB_v$ with these labels.
If $\gamma$ and $\gamma'$ are two such geodesics, they are necessarily disjoint.   
 Let $w$ be the (non-empty) word labelling a shortest path $\nu$ between $\gamma$ and $\gamma'$ and let $z$ be the first letter in $w$.  Assume without loss of generality that $\gamma$  and 
 $\nu$ intersect at $e$.

 We now use Lemma~\ref{lem:z-sep} to define a geodesic $\mu$ passing through some point along $\nu$.  If Case (1) of the lemma holds, then define $\mu$ to be the $(c_1, c_2)$-bicoloured geodesic passing through $e$. (Observe that in this case $b$ from the lemma is necessarily one of $b_1$ and $b_2$.)
If Case (2) of the lemma holds, then let $u$ be the longest subword of $w$ labelled by generators from the branch $\beta$ of $\G$ containing $z$ guaranteed by the lemma.  If $u\neq w$, then define $\mu$ to be the 
$(b_1, b_2)$-bicoloured geodesic through the vertex of $\cC$ labelled $u$.  In both these cases, we arrive at a contradiction by a proof similar to that of Proposition~\ref{prop:stars}(c).

 Finally, if $u=w$, then 
 the vertices of $\beta$ (which include $b_1$, $b_2$ and the labels of $u$) 
 are contained in a set $A$ as in Proposition~\ref{prop:sim}, and it follows that $\gamma$ and $\gamma'$ are either both in $\cA$ (which is a contradiction because then the vertex corresponding to $A$ in $T$ is between the vertices corresponding to $\gamma$ and $\gamma'$) or they are both in 
$\cAplus$ (which is a contradiction, since $\cB_v$ only contains one geodesic from each such set). 
This completes the proof that $\card(B) \neq 2$.

Now suppose $\card(B)=3$, with $B= \{b_1, b_2, b_3\}$.  Since  $B$ comes from bicoloured geodesics which separate $\Sigma$, there are at least two cut pairs in $B$, say $\{ b_1, b_2\} $ and $\{b_2, b_3\}$.  

Suppose that $b_1$ and $b_3$ don't separate $\G$. Starting with a path $\mu$ between them in $\G$, and using the separation properties of $\{b_1, b_2\}$ and $\{b_2, b_3\}$, we can complete $\mu$ to a cycle 
$\sigma$ containing all elements of $B$.  Since $\{b_1, b_3\}$  doesn't separate, there is a path connecting $\mu$ to the other component of $\sigma \setminus \{b_1, b_3\}$ which meets $\mu$ at an essential vertex $c\neq b_1, b_3$, and the other component at a single point.  Then there are three paths from $c$ to $B$ which meet only at $c$, so by Observation~\ref{obs:B}, we have $c \in B$, which is a contradiction.

Thus we may assume that $\{b_1, b_3\}$ also separates $\G$, so that $B$ satisfies (A1).  If $B$ 
also satisfies (A2), then it is contained in a set $A$ which is maximal among all sets satisfying both (A1) and (A2).  That is, $B$ is contained in a set $A$ which satisfies (A1), (A2) and (A3), and therefore defines an infinite $\sim$-class.  However, $\cB_v$ only contains one geodesic from each 
$\sim$-class in $v$, so this is a contradiction. 

Thus $B$ does not satisfy (A2).  If there were no subdivided $K_4$ subgraph containing $B$, then (A2) would be vacuously true.  It follows that there is a subdivided $K_4$ subgraph $\Lambda$ of $\G$ containing $B$, and moreover, that $b_1, b_2$ and $b_3$ lie on at least two distinct branches of $\Lambda$.

Now if there is an essential vertex $x$ of $\Lambda \setminus B$ which  admits three paths to $B$ meeting only at $x$, then by Observation~\ref{obs:B}, the maximality of $B$ is violated, and this is a contradiction.  Thus no essential vertex of  $\Lambda \setminus B$ admits three such paths to $B$. This can only happen if 
$B$ is contained in a single branch of $\Lambda$. This is a contradiction. 

We have shown that the set $B$ satisfies properties (B1), (B2) and (B3) in 
Proposition~\ref{prop:stars}.  It is then easy to verify that the geodesics in $\cB_v$, together with any geodesics parallel to them, form the set $\cB$ defined in the statement of Proposition~\ref{prop:stars}, and that $\Phi(\cB) = \Phi(\cB_v)$.  Hence $v = \Phi(\cB)$ is one of the stars we have already constructed.  This completes the proof.
\end{proof}

\begin{example}  We now describe the stars (if any) for the graphs in Figures~\ref{fig:approx-pairs} and~\ref{fig:sim-eg}, and their adjacencies in $\cT$.  There are no sets of vertices $B$ satisfying conditions (B1), (B2) and (B3) in either graph in Figure~\ref{fig:approx-pairs}, or in the left-hand and central graphs in Figure~\ref{fig:sim-eg}.  Hence by Proposition~\ref{prop:all-stars}, the corresponding JSJ trees have no stars, and so $\cT$ coincides with the pre-tree $T$; compare Example~\ref{eg:sim and approx}, which describes all $\approx$-pairs and $\sim$-classes, and their adjacencies in $T$.  The left-hand (respectively, right-hand) graph in Figure~\ref{fig:approx-pairs} has JSJ tree which is particularly easy to describe: it is biregular with vertices alternating between $\approx$-pairs of valence $6$ (respectively, $4$) and infinite $\sim$-classes.  

On the right of Figure~\ref{fig:sim-eg}, there are two $W$-orbits of stars, corresponding to the sets $\{u,v,p,q\}$ and $\{p,q,r,s\}$.  There are no $\approx$-pairs.  Each infinite $\sim$-class corresponds to a branch of this graph of length three, and so is adjacent in $\cT$ to a Type 1 vertex which is added in to subdivide the edge between this infinite $\sim$-class (which is a Type 2 vertex) and a representative of exactly one of these $W$-orbits of stars (which are Type 3 vertices).  The $\sim$-pairs for the branch between $u$ and $v$ and the branch between $r$ and $s$ are adjacent in $\cT$ to a representative of exactly one $W$-orbit of stars, and the  $\sim$-pairs which correspond to the set $\{p,q\}$ are adjacent in $\cT$ to representatives of both $W$-orbits of stars.  
\end{example}

\begin{remark}\label{rem:lafont}  Having identified all vertices of the JSJ tree $\cT$, we compare with the work of Lafont~\cite{lafont-3d, lafont}.  It is easy to see that the $\approx$-pairs in our setting are analogous to the endpoints of branching geodesics in~\cite{lafont} (which are studied using some results from~\cite{lafont-3d}).  Now suppose $A$ is a set of vertices of $\G$ corresponding to an infinite $\sim$-class such that $A$ consists of all vertices in a single branch of $\G$ of length at least three. Then the $\sim$-classes corresponding to $A$ are analogous to the visual boundaries of chambers in~\cite{lafont}.  The spaces considered in~\cite{lafont} consist entirely of chambers glued together along branching geodesics, and so when $\cT$ consists entirely of ``single branch" infinite $\sim$-classes alternating with $\approx$-pairs, as in the case of the graphs in Figure~\ref{fig:approx-pairs}, the situation is very similar to that in~\cite{lafont}.  However in our setting there may not be any $\approx$-pairs or any infinite $\sim$-classes, and there may also be $\sim$-pairs, infinite $\sim$-classes which do not come from a single branch, and stars.\end{remark}

We finish this section by determining the stabilisers of stars.

\begin{prop}\label{prop:star stabiliser}  Let $v$ be a vertex of $\cT$ corresponding to a star of size at least 3 in $T$, and let $B$ be the corresponding subset of vertices of $\G$ satisfying (B1), (B2) and (B3) in Proposition~\ref{prop:stars}.  Then the stabiliser of $v$ is a conjugate of $\langle B \rangle$.
\end{prop}

\begin{proof}  Let $\mathcal{G}_B$, $\cB$ and $\Phi:\cB \to T$ be as in the statement of Proposition~\ref{prop:stars}.  It suffices to show that the stabiliser of the star $\Phi(\cB)$ is $\langle B \rangle$.  Let $g \in W$ be in this stabiliser.

Let $\eta$ be a geodesic in $\mathcal{G}_B$ and let $v'$ be the $\approx$-pair or $\sim$-class $\Phi(\eta)$.  Since $g$ stabilises $v$, we have that $g(v') = \Phi(\xi)$ for some $\xi \in \cB$.  Note that there are at most two choices for $\xi$, and these choices have the same endpoints, by (a) in the proof of Proposition~\ref{prop:stars}.  We will need the following result.

\begin{lemma}\label{lem:star endpoints}  The geodesics $\xi$ and $g\eta$ have the same endpoints.
\end{lemma}

\begin{proof} If $v'$ is an $\approx$-pair or a $\sim$-pair then $v' = \partial \eta$ and $g(v') = \partial \xi$, hence $\partial \xi = g(\partial \eta) = \partial (g\eta)$ as required. Now suppose $v'$ is an infinite $\sim$-class and that $\xi$ and $g\eta$ do not have the same endpoints.  By Lemma~\ref{lem:parallel}, since $\xi$ and $g\eta$ are bicoloured, the sets $\partial \xi$ and $\partial (g\eta)$ are disjoint.  

Let $A$ be the set of vertices of $\G$ corresponding to $v'$ as in Proposition~\ref{prop:sim}, so that $v' = \partial \cC_A \setminus \partial \cA$ for some copy $\cC_A$ of the Cayley graph of $\langle A \rangle$.  
Then by similar arguments to those in the proof of Proposition~\ref{prop:all-stars}, $\eta$ is the frontier geodesic of $\Sigma_A \supseteq \cC_A$ corresponding to the component of $\Sigma \setminus \Sigma_A$ which contains $\cB \setminus \{\eta\}$.  Similarly, writing $\Sigma_A'$ for the copy of $\Sigma_A$ for the infinite $\sim$-class $g(v')$, we have that $\xi$ is the frontier geodesic of $\Sigma_A'$ corresponding to the component of $\Sigma \setminus \Sigma_A'$ which contains $\cB \setminus \{\xi\}$.  By Proposition~\ref{prop:sim stabiliser} the stabiliser of $v'$ is $\langle A \rangle$ and the stabiliser of $g(v')$ is $g\langle A \rangle g^{-1}$.  It follows that $g$ takes the frontier geodesic $\eta$ of $\Sigma_A$ to a frontier geodesic $g\eta$ of $\Sigma_A'$.  The endpoints of $g\eta$ are in $g(v')$, so $g\eta$ is the frontier geodesic of $\Sigma_A'$ corresponding to the component of $\Sigma \setminus \Sigma_A'$ which contains $g\cB \setminus \{g\eta\}$.

Now $v$ is a star of size at least 3, so there is a $v'' \in v$ with $v''$ distinct from both $v'$ and $g(v')$.  Note that $g(v'') \neq g(v')$.  Let $\mu \in \cB$ be such that $\Phi(\mu) = v''$.  Then $\mu$ and $g\eta$ lie in distinct components of $\Sigma \setminus \xi$.  Now the endpoints of $g\mu$ are contained in $g(v'')$, and $g(v'')$ is in the star $v$.  However $g\mu$ lies in $g\cB \setminus \{g\eta\}$, so $g\mu$ and $\mu$ are in distinct  components of $\Sigma \setminus \Sigma_A'$.  This means that $g(v')$ separates $g(v'')$ from $v''$ in $T$, which contradicts $v$ being a star.  Thus $\xi$ and $g\eta$ have the same endpoints as required.
\end{proof}

Assume now that there is a $g \in W \setminus \langle B \rangle$ which stabilises~$v$.  
By Lemma~\ref{lem:star-sep}, there exists a cut pair $\{b_i,b_j \} \subset B$.  Let $\eta$
be the (unique) geodesic in $\mathcal{G}_B$ which is $(b_i,b_j)$-bicoloured. 
By similar arguments to those in the proof of Proposition~\ref{prop:sim stabiliser}, since $g \not \in \langle B \rangle$ there is a (unique) vertex $x$ of $\G \setminus B$ so that $x$ is adjacent to both $b_i$ and $b_j$, and  $x$ stabilises $v$.  

Since $\card(B) \geq 4$, there is an element $b_k \in B \setminus \{b_i,b_j\}$.  Consider the geodesic $b_k \eta \in \cB$.   The group element $x$ stabilises $v$ but $x \not \in \langle B \rangle$, so we have by the same arguments as for $g$ and $\eta$ above that the geodesic $x b_k \eta$, which is not in $\cB$, has the same endpoints as some geodesic, say $\xi'$, which is in $\cB$.  Corollary~\ref{cor:parallel} then implies that the geodesic $\xi'$ is $(b_i,b_j)$-bicoloured and $x b_k \eta = \xi' x$.  Now since $\xi'$ is $(b_i,b_j)$-bicoloured and is in $\cB$, we have $\xi' = h \eta$ for some $h \in \langle B \rangle$.  Thus  $x b_k \eta$ equals $h \eta x =  h x \eta$ for some $h \in \langle B \rangle$.  
As the stabiliser of $\eta$ is $\langle b_i, b_j \rangle$, it follows that $b_k^{-1} x^{-1} h x = b_k x h x\in \langle b_i, b_j \rangle$.  Now $x$ and $b_k$ are not in $\langle b_i,b_j\rangle$, so this means $x$ commutes with $b_k$.  Thus $x$ is adjacent to all three elements of $\{b_i,b_j,b_k\}$.  Hence by Observation~\ref{obs:B}, $x \in B$ as well.  This contradicts $x \in \G \setminus B$, and so completes the proof.
\end{proof}

\subsection{The tree $\cT$}\label{sec:summary}

We now list all vertices and edges in $\cT$, and describe their stabilisers.  

\begin{thm}\label{thm:JSJ}  The vertices of $\cT$ and their stabilisers are as follows.

\begin{enumerate}
\item The $W$-orbits in $V_1(\cT)$ are in bijective correspondence with the following four families.
\begin{enumerate}

\item ($\approx$-pairs, valence $k \ge 3$)
Pairs $\{a, b\}$ such that $\G \setminus \{a, b\}$ has $k \ge 3$ components and $\G \setminus \{ a,b\}$ has no component consisting of a single vertex. The stabilisers of these vertices are the conjugates of either $\langle a, b \rangle$, if there is no vertex $c$ of $\G$ adjacent to both $a$ and $b$, or of $\langle a, b, c \rangle$, if there is such a vertex $c$.

\item ($\approx$-pairs, valence $2(k-1) \ge 4$)
Pairs $\{a, b\}$ such that $\G \setminus \{a, b\}$ has $k \ge 3$ components and $\G \setminus \{ a,b\}$ has a component consisting of a single vertex $c$. The stabilisers of these vertices are the conjugates of $\langle a, b, c \rangle$.

\item ($\sim$-classes of size 2, valence 2)
Sets $A$ satisfying properties \emph{(A1)}, \emph{(A2)} and \emph{(A3)} in 
Proposition~\ref{prop:sim} with $\langle A \rangle$ being $2$-ended and which are not as in (1)(a) or (1)(b).  The  stabilisers of these vertices are the conjugates of either $\langle A \rangle$, if $\card(A) = 3$ or $\card(A) = 2$ and there is no vertex $c$ of $\G$ adjacent to both elements of $A$, or of $\langle A \cup \{c\}\rangle$, if $\card(A) = 2$ and there is such a vertex $c$.

\item Valence 2 vertices added between Type 2 and Type 3 vertices (if any).
The stabilisers of these vertices are the intersections of stabilisers of their endpoints. 

\end{enumerate}

\item The $W$-orbits in $V_2(\cT)$ are in bijective correspondence with sets $A$ satisfying properties \emph{(A1)}, \emph{(A2)} and \emph{(A3)} in 
Proposition~\ref{prop:sim} with $\langle A \rangle$ infinite but not $2$-ended.   
The stabilisers of these vertices are conjugates of $\langle A \rangle$. 

\item The $W$-orbits in $V_3(\cT)$ are in bijective correspondence with sets $B$ satisfying properties \emph{(B1)}, \emph{(B2)} and \emph{(B3)} in 
Proposition~\ref{prop:stars}.  
The stabilisers of these vertices are conjugates of $\langle B \rangle$. 

\end{enumerate}

The edges of $\cT$ are as described in Section~\ref{sec:edges}.  Furthermore, there is an edge 
between vertices $v$ and $v'$ if and only if the corresponding stabilisers intersect. 
\end{thm}

\begin{proof}  The identification of all vertices in $\cT$ follows from the construction of the JSJ tree together with Corollary~\ref{cor:approx} and Lemma~\ref{lem:all approx} for the $\approx$-pairs, Propositions~\ref{prop:sim} and~\ref{prop:all 2+} for the $\sim$-classes, and Propositions~\ref{prop:stars} and~\ref{prop:all-stars} for the stars of size at least 3.  We determined the stabilisers of $\approx$-pairs in Lemma~\ref{lem:approx stabiliser}, of $\sim$-classes in Proposition~\ref{prop:sim stabiliser} and of stars in Proposition~\ref{prop:star stabiliser}.
Finally the stabilisers of any valence $2$ vertices added between Type 2 vertices and stars are the intersections of the stabilisers of the endpoints, since these vertices are just subdividing edges of $\cT$.
\end{proof}

See Figure~\ref{fig:quotients} for an illustration of the quotient graph of groups obtained by the action of $W$ on $\cT$, for the examples in Figure~\ref{fig:approx-pairs}.

\begin{figure}[ht]
\begin{center}
\begin{overpic}[width=0.9\textwidth]{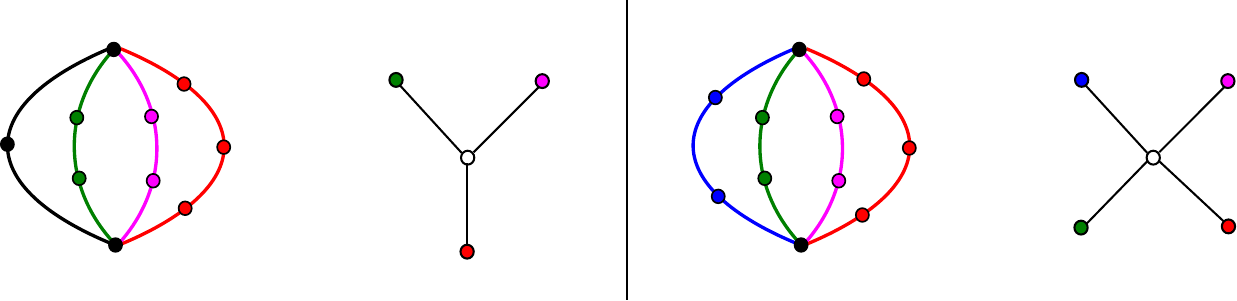}
% left
\put(9,22){$a$}
\put(9,1){$b$}
\put(-2,12){\footnotesize{$c$}}
\put(3,15){\footnotesize{\textcolor{ForestGreen}{$x_1$}}}
\put(3,9){\footnotesize{\textcolor{ForestGreen}{$x_2$}}}
\put(9.5,15){\footnotesize{\textcolor{magenta}{$y_1$}}}
\put(9.5,10){\footnotesize{\textcolor{magenta}{$y_2$}}}
\put(16,18){\footnotesize{\textcolor{red}{$z_1$}}}
\put(19,12){\footnotesize{\textcolor{red}{$z_2$}}}
\put(16,7){\footnotesize{\textcolor{red}{$z_3$}}}
\put(29,11){\footnotesize{$\langle a, b,c \rangle$}}
\put(41,20){\footnotesize{$\langle a, b, y_i \rangle$}}
\put(34,1){\footnotesize{$\langle a, b, z_i \rangle$}}
\put(27,20){\footnotesize{$\langle a, b, x_i \rangle$}}
% right
\put(64,22){$a$}
\put(64,1){$b$}
\put(56,18){\footnotesize{\textcolor{blue}{$s_1$}}}
\put(55,8){\footnotesize{\textcolor{blue}{$s_2$}}}
\put(59.5,16){\footnotesize{\textcolor{ForestGreen}{$t_1$}}}
\put(59,10){\footnotesize{\textcolor{ForestGreen}{$t_2$}}}
\put(64.5,15){\footnotesize{\textcolor{magenta}{$u_1$}}}
\put(64.5,10){\footnotesize{\textcolor{magenta}{$u_2$}}}
\put(71,18){\footnotesize{\textcolor{red}{$v_1$}}}
\put(74.5,12){\footnotesize{\textcolor{red}{$v_2$}}}
\put(71,6.5){\footnotesize{\textcolor{red}{$v_3$}}}
\put(87,11){\footnotesize{$\langle a, b \rangle$}}
\put(96,20){\footnotesize{$\langle a, b, u_i \rangle$}}
\put(96,3){\footnotesize{$\langle a, b, v_i \rangle$}}
\put(82,20){\footnotesize{$\langle a, b, s_i \rangle$}}
\put(82,3){\footnotesize{$\langle a, b, t_i \rangle$}}
\end{overpic}
\caption{{\footnotesize The quotient graphs of groups obtained via the action of $W = W_\G$ on its JSJ tree $\cT = \cT_\G$, for the graphs $\G$ from Figure~\ref{fig:approx-pairs}.  Each edge group is the intersection of the adjacent vertex groups, so all edge groups are $\langle a, b \rangle$.}} 
\label{fig:quotients}
\end{center}
\end{figure}
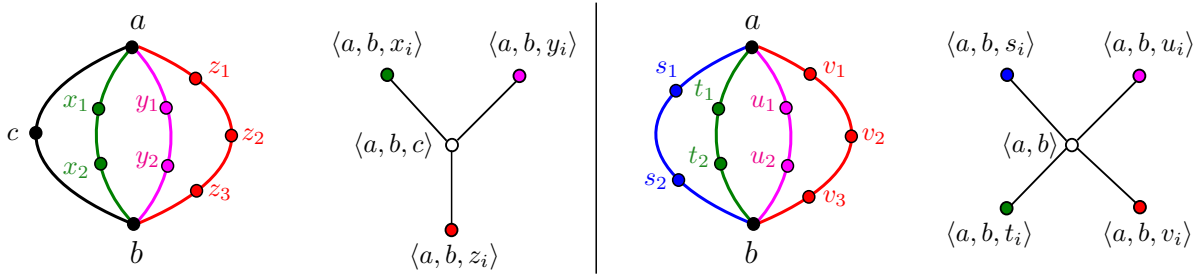

%%%%%%%%%%%%%%%%%%%%%%%%%%%
%%%%%%%%%%%%%%%%%%%%%%%%%%%
\section{Applications}\label{sec:applications}
%%%%%%%%%%%%%%%%%%%%%%%%%%%
%%%%%%%%%%%%%%%%%%%%%%%%%%%

In this section we give some applications of our main result.  In Section~\ref{sec:k4} we prove Theorem~\ref{thm:complete-inv} and in Section~\ref{sec:qi} we provide a precise statement for and proof of Theorem~\ref{thm:qi-rigid}.

%%%%%%%%%%%%%%%%%%%%%%%%%%%%%%%%%%
\subsection{A collection for which $\cT$ is a complete invariant}\label{sec:k4}
%%%%%%%%%%%%%%%%%%%%%%%%%%%%%%%%%%
In~\cite{malone}, Malone 
uses the quasi-isometries of ``fattened trees" from Behrstock--Neumann~\cite{behrstock-neumann} to construct 
 a quasi-isometry between any pair of ``geometric amalgams of free groups" which have the same Bowditch JSJ tree.  Cashen and Martin observe that this argument generalises to any pair of $1$-ended hyperbolic groups which are not cocompact Fuchsian, so long as their corresponding JSJ trees do not have any stars, in Theorem 3.9 of~\cite{cashen-martin}.  
In light of this,  to establish Theorem~\ref{thm:complete-inv} it is enough to prove the following. 

\begin{prop}\label{prop:k4}
Let $\G$ be a graph satisfying Standing Assumptions~\ref{assumptions}.  Then the JSJ tree $\cT = \cT_\G$ has no stars if and only if 
$\G$ has no subdivided $K_4$ subgraphs.  
\end{prop}

\begin{proof}
First we show that if $\G$ has a subdivided $K_4$ subgraph $\Lambda$ then there is a star in $\cT$.  
Let $b_1, b_2, b_3, b_4$ be the essential vertices of the $K_4$, considered as vertices in 
$\Lambda$.  If $c_1, c_2$ is an arbitrary pair of essential vertices of $\G$, then the subgraph $\Lambda \setminus \{c_1, c_2\}$ of 
$\G  \setminus \{c_1, c_2\}$ is a (subdivided) $K_4$ with at most two points missing, and is therefore connected.  
It follows that $\{b_1, b_2, b_3, b_4 \}\setminus \{c_1, c_2\}$ is contained in a single component of 
$\G  \setminus \{c_1, c_2\}$.
Thus $\{b_1, b_2, b_3, b_4 \}$
 satisfies properties (B1) and (B3) of Proposition~\ref{prop:stars}.  It is therefore contained in a set satisfying (B2), yielding an orbit of stars.  

Now suppose $\cT$ has a star, which comes from a set of essential vertices $B$ as in 
Proposition~\ref{prop:stars}.  By Lemma~\ref{lem:cycle-k4} it suffices to show that some three elements of $B$ lie on an induced cycle in $\G$.   By Lemma~\ref{lem:star-sep}, there is a cut pair $\{b_1, b_2 \}\subseteq B$.
Then we may choose $b_3 \neq b_1, b_2$ in $B$, and reduced paths $\tau$ and $\mu$
such that $\tau$ connects $b_1$ and $b_2$ in the component of $\G \setminus \{b_1, b_2\}$ which does not contain $b_3$, and $\mu$
connects $b_2$ and $b_3$ and misses $b_1$.  By construction,  $\tau \cap \mu =\{b_2\}$.
Now let $\eta$ be a reduced path from $b_3$ to $b_1$ which misses $b_2$.  Clearly $\eta \cap \tau= \{b_1\}$.   If $\eta\cap \mu= \{b_3\}$ then $\tau \cup \mu \cup \eta$ is an induced cycle containing $b_1, b_2$, and $b_3$. 
If $\eta\cap \mu$ contains points other than $b_3$, 
then let $x$ be the first point of $\eta \cap \mu$ that $\eta$ meets, starting from $b_1$, so that $\eta_{[b_1x]} \cap \mu= \{x\}$.  Then $x$ is connected to 
$b_1, b_2$ and $b_3$ by the disjoint paths 
$\eta_{[b_1,x]}$,
$\mu_{[b_2,x]}$ and $\mu_{[b_3,x]}$.  Now
by 
Observation~\ref{obs:B},
$x \in B$.  Then $\sigma = \eta_{[b_1,x]} \cup 
\mu_{[b_2,x]}\cup \tau$ is an induced  cycle containing $b_1, b_2, x \in B$, as desired. 
\end{proof}

%%%%%%%%%%%%%%%%%%%%%%%%%%%%%%%%%%%%%
\subsection{Quasi-isometric rigidity}\label{sec:qi}
%%%%%%%%%%%%%%%%%%%%%%%%%%%%%%%%%%%%%

In this section we state and prove Theorem~\ref{thm:fuchsian-racg}, which gives a precise version of Theorem~\ref{thm:qi-rigid} of the introduction. Theorem~\ref{thm:fuchsian-racg} follows by combining the work of Tukia~\cite{tukia}, 
Gabai~\cite{gabai}, and Casson--Jungreis~\cite{casson-jungreis} and 
Theorem~\ref{thm:fuchsian-all} of the Appendix, which is well-known to experts.  
Below we give a direct proof of the theorem.

Recall the definition of a Fuchsian group from the introduction.  Note that if $G$ is Fuchsian then so is $G \times F$ for any finite group $F$.  Denote by $\Lambda_n$ the $n$-cycle of length $n \geq 5$ and put $W_n = W_{\Lambda_n}$.  Then $W_n$ is the group generated by reflections in the sides of a right-angled hyperbolic $n$-gon.  

A \emph{generalised $\Theta$ graph} is defined as follows.    
Let $k \geq 3$ and $n_1,\dots,n_k \ge 0$ be integers.  The graph $
\Theta(0,0,\ldots,0)$ is the graph with two essential vertices $a$ and $b$ each of valence $k$, and $k$ edges $e_1$, $e_2$, \ldots, $e_k$ connecting $a$ and $b$.  
Then the graph $\Theta(n_1,n_2,\ldots,n_k)$ is obtained by, for each $1 \leq i \leq k$, subdividing the edge $e_i$ into $n_i + 1$ edges by inserting $n_i$ new vertices of valence $2$.  For example, the graphs in Figure~\ref{fig:approx-pairs} are $\Theta(1,2,2,3)$ and $\Theta(2,2,2,3)$.

\begin{thm}\label{thm:fuchsian-racg}
Let $W_\G$ be $2$-dimensional.
The following are equivalent:
\begin{enumerate}
\item $W_\G$ is cocompact Fuchsian.

\item $W_\G$ is quasi-isometric to $W_n$ for some $n\ge 5$.

\item $\G = \Lambda_n$ for some $n \ge 5$.   
\end{enumerate}
\end{thm}

\begin{proof} The implications $(1) \implies (2)$, $(3) \implies (1)$ and $(3) \implies (2)$ are obvious.  In order to establish $(2) \implies (3)$ and so complete the proof, we will use the following lemma, whose proof is an elementary exercise in graph theory. 

\begin{lemma}\label{lem:theta subgraph}  Suppose $\G$ is a triangle-free simplicial graph which has no separating vertices or edges.  If  $\Gamma$ is not a cycle graph or a single edge then $\Gamma$ contains an induced subgraph $\Theta(n_1,n_2,n_3)$ with $n_1, n_2, n_3 \geq 1$.
\qed
\end{lemma}

Now suppose that $W_\G$ is 2-dimensional and quasi-isometric to $W_n$ for some $n\geq 5$.   Then $W_\G$ is $1$-ended, so $\G$ has no separating vertices or edges.  Also $\G$ is not a single edge, as then $W_\G$ would be finite.  Assume that $\G$ is not a cycle graph.  Then by Lemma \ref{lem:theta subgraph}, $\G$ contains an induced subgraph $\Theta = \Theta(n_1,n_2,n_3)$.  If more than one $n_i$ is $1$, then $\Theta$ and hence $\G$ contains a square, thus $W_{\G}$ is not hyperbolic. This is a contradiction, so we may assume $n_2, n_3 \ge 2$.  

Since $W_\Theta$ is a $1$-ended special subgroup of $W_\G$, its visual boundary $\partial W_\Theta$ is a closed, connected subset of $\partial W_\G$.  Now $\partial W_\G$ is homeomorphic to $S^1$, so $\partial W_\Theta$ is homeomorphic to either $S^1$ or to a closed interval in $S^1$.   Let $a$ and $b$ be the two essential vertices of $\Theta$ and let $\gamma$ be an 
$(a, b)$-bicoloured geodesic in the Cayley graph of $W_\Theta$.  Since $\Theta \setminus \{a,b\}$ has three components, by Corollary~\ref{cor:approx} the endpoints of $\gamma$ are in $\partial W_\Theta(3+)$.  Thus $\partial W_\Theta$ cannot be homeomorphic to either~$S^1$ or an interval in $S^1$.  Hence $\G$ is a cycle graph $\Lambda_n$ with $n \geq 5$.
\end{proof}

%%%%%%%%%%%%%%%%%%%%%
%%%%%%%%%%%%%%%%%%%%%
\appendix
%%%%%%%%%%%%%%%%%%%%%
%%%%%%%%%%%%%%%%%%%%%

%%%%%%%%%%%%%%%%%%%%%%%%%%%%%%%%%%%%%%%%%%
\section{Characterisation of cocompact Fuchsian Coxeter groups}\label{sec:fuchsian-all}
%%%%%%%%%%%%%%%%%%%%%%%%%%%%%%%%%%%%%%%%%%

In this appendix we prove Theorem~\ref{thm:fuchsian-all}, which characterises cocompact Fuchsian groups among all Coxeter groups.   
The proof of Theorem~\ref{thm:fuchsian-all} uses results from~\cite{davis-book} which were suggested to us by an anonymous referee.

  In this appendix only, we consider general Coxeter groups.  Recall that a \emph{Coxeter group} is a group with  presentation
\begin{equation}\label{eq:Coxeter}  W = \langle S \mid s_i^2 = 1\, \forall \, s_i \in S, (s_is_j)^{m_{ij}} = 1 \, \forall \mbox{ distinct } s_i, s_j \in S \rangle\end{equation}
where $S = \{ s_i \}_{i=1}^n$ is a finite set and $m_{ij} = m_{ji} \in \{ 2,3,4,\dots\} \cup \{ \infty \}$, with $m_{ij} = \infty$ meaning that there is no relation between $s_i$ and $s_j$.  The pair $(W,S)$ is called a \emph{Coxeter system}.  The \emph{nerve} of a Coxeter system $(W,S)$ is the simplicial complex $N(W,S)$ with vertex set $S$ and a $k$-simplex $\sigma_T$ with vertex set $T \subset S$ for each subset $T$ of $S$ such that $\langle T \rangle$ is finite and $|T| = k + 1$.  The Coxeter system $(W,S)$ is \emph{right-angled} if each $m_{ij} \in  \{2,\infty\}$, and if $W = W_\G$ is right-angled then its defining graph $\G$ is precisely the $1$-skeleton of the nerve $N(W,S)$.  Moreover, $\G$ is triangle-free if and only if $\G$ is equal to the nerve of $W_\G$.

\begin{example}\label{ex:polygon} Let $P$ be a compact convex hyperbolic polygon with edges labelled cyclically $e_1, \dots, e_n$, so that for $1 \leq i \leq n$ the dihedral angle between $e_i$ and $e_{i+1}$ is $\pi/m_{i,i+1}$  with $m_{i,i+1}$ an integer $\geq 2$.  Put $m_{i+1,i} = m_{i,i+1}$ and if edges $e_i$ and $e_j$ are nonadjacent, put $m_{ij} = m_{ji} = \infty$.  Let $S = \{ s_i \}_{i=1}^n$ where $s_i$ is the reflection of $\H^2$ in the geodesic containing $e_i$.  Let $W = W(P)$ be  the group generated by this set of reflections $S$.  Then $W$ is cocompact Fuchsian, and is a Coxeter group with presentation as in \eqref{eq:Coxeter}.
\end{example}

A Coxeter group $W$ is a \emph{hyperbolic polygon reflection group} if $W = W(P)$ for some hyperbolic polygon $P$ as in Example~\ref{ex:polygon}.  Following Definition~10.6.2 of \cite{davis-book} we say that a Coxeter system $(W,S)$ is of \emph{type} $\mbox{HM}^2$ if its nerve is a generalised homology $1$-sphere (see Definition 10.4.5 of~\cite{davis-book} for the definition of this latter term.) 

\begin{thm}\label{thm:fuchsian-all}
A Coxeter group $W$ is cocompact Fuchsian if and only if $W$ is either a hyperbolic polygon reflection group, or the direct product of a hyperbolic polygon reflection group with a finite Coxeter group.
\end{thm}

\begin{proof}[Proof of Theorem~\ref{thm:fuchsian-all}.]
Suppose that a Coxeter group $W$ is cocompact Fuchsian.  Then $W$ has a finite index torsion-free subgroup which is the fundamental group of a closed hyperbolic surface.  Hence by~\cite[Theorem 10.9.2]{davis-book}, either $(W,S)$ is of type $\mbox{HM}^2$, or $(W,S)$ has a direct product decomposition $(W,S) = (W_0,S_0) \times (W_1,S_1)$ where $W_1$ is a finite (nontrivial) Coxeter group and $(W_0,S_0)$ is of type $\mbox{HM}^2$.  It thus suffices to show that a Coxeter group of type $\mbox{HM}^2$ must be a hyperbolic polygon reflection group.  Now as remarked on p. 194 of~\cite{davis-book}, every generalised homology $1$-sphere is homeomorphic to the circle $S^1$.  Therefore the nerve of $(W,S)$ is an $n$-cycle for some $n \geq 3$.  Since $W$ is hyperbolic, it follows that $W$ is a hyperbolic polygon reflection group as required.
\end{proof}

\begin{cor}\label{cor:fuchsian-2dim}  Let $\G$ be a finite, simplicial, triangle-free graph.  Then the associated right-angled Coxeter group $W_\G$ is cocompact Fuchsian if and only if $\G$ is an $n$-cycle with $n \geq 5$.
\end{cor}

Corollary~\ref{cor:fuchsian-2dim} 
provides an alternative proof of $(1) \implies (3)$ in 
Theorem~\ref{thm:fuchsian-racg}.

%%%%%%%%%%%%%%%%%%%%%%%%%%%%%%%%%%%%%%%%%%
\section{Quasi-isometry classification for complete graphs\\ {\small Christopher Cashen, Pallavi Dani, and Anne Thomas}}\label{app:Kn}
%%%%%%%%%%%%%%%%%%%%%%%%%%%%%%%%%%%%%%%%%%

In this appendix
we use techniques from Behrstock and Neumann~\cite{behrstock-neumann} and Cashen and Macura~\cite{cashen-macura} to give the quasi-isometric classification of all right-angled Coxeter groups whose defining graphs are obtained by sufficiently subdividing a complete graph on at least $4$ vertices.  In particular, this classification shows that Bowditch's JSJ tree is not a complete quasi-isometry invariant in general (i.e.~for all graphs satisfying Standing 
Assumptions~\ref{assumptions}). 

We define a graph to be \emph{$3$-convex} if every path between a pair of essential vertices has at least three edges.   Recall that $K_n$ denotes the complete graph on $n$ vertices.

\begin{thm}\label{thm:Kn}
For each $n\ge 3$, let $\hat K_{n+1}$ be a 3-convex graph which is a subdivided copy of $K_{n+1}$, and define $W_{n+1}$ to be the right-angled Coxeter group $W_{n+1}= W_{\hat K_{n+1}}$.  
\begin{enumerate} 
\item For all $n \geq 3$, all $W_{n+1}$ have isomorphic JSJ trees.
\item For any two such subdivisions $\hat K_{n+1}$ and $\hat K_{n+1}'$, the corresponding groups $W_{n+1} = W_{\hat K_{n+1}}$ and $W'_{n+1} = W_{\hat K_{n+1}'}$ are quasi-isometric. 
\item $W_{m+1}$ is not quasi-isometric to $W_{n+1}$ when $m \neq n$.
\end{enumerate}
\end{thm}

\begin{proof}
We first obtain a description of the JSJ tree $\cT := \cT_{n+1}$ of any $W_{n+1}$, using 
Theorem~\ref{thm:JSJ}.  
Let $x_0,\dots, x_n$ be the essential vertices of $\hat K_{n+1}$.
 For each $0 \le i < j \le n$, there is one 
$W_{n+1}$-orbit in $V_2(\cT)$ corresponding to the branch of $\hat K_{n+1}$ between $x_i$ and 
$x_j$. Also, there is a single $W_{n+1}$-orbit in $V_3(\cT)$, corresponding to the set $\{x_0, \dots, x_n\}$.  There 
are no $\approx$-pairs or $\sim$-classes of size 2.  Thus $V_1(\cT)$ contains only the valence two vertices 
described in (1)(d) of Theorem~\ref{thm:JSJ}.
 It follows that $\cT$ is a tripartite tree in which each vertex in $V_2(\cT)$ and in $V_3(\cT)$ is 
connected to countably infinitely many vertices in $V_1(\cT)$, and each vertex in $V_1(\cT)$ is 
 connected to one vertex in $V_2(\cT)$ and one vertex in $V_3(\cT)$.  It is evident from this description that there is a type-preserving isomorphism from $\cT_{n+1}$ to $\cT_{m+1}$, for all $m,n \ge 3$.
In other words, the JSJ trees of all $W_{n+1}$ are isomorphic for all $n\ge 3$, and so (1) holds.

For (2), the type-preserving isomorphism between the two JSJ trees $\cT_{W_{n+1}}$ and $\cT_{W'_{n+1}}$ induces, by techniques of Behrstock--Neumann~\cite{behrstock-neumann},  a quasi-isometry between the Davis complexes $\Sigma$ and $\Sigma'$ for $W_{n+1}$ and $W'_{n+1}$, respectively.  Thus $W_{n+1}$ and $W'_{n+1}$ are quasi-isometric.  The map $\Sigma \to \Sigma'$ sends subcomplexes corresponding to Type 2 vertex stabilisers (that is, branches) quasi-isometrically to subcomplexes corresponding to Type 2 vertex stabilisers.  The restriction of this map to any such subcomplex can be extended to a quasi-isometry $\Sigma \to \Sigma'$ because of the bijections between the branches and the essential vertices of the two subdivisions $\hat K_{n+1}$ and $\hat K_{n+1}'$.

For (3), by Theorem 3.37 the stabilisers of the vertices in $V_3(\cT_{n+1})$ are conjugates of the special subgroup \[R_{n+1}:=\langle x_0,\dots, x_n \mid x_i^2=1\rangle\] of $W_{n+1}$.
Also, the stabilisers of the vertices in $V_1(\cT_{n+1})$ are
  conjugates of the infinite dihedral groups $D_{i,j}:=\langle
  x_i,x_j\rangle$ for $0\leq i<j\leq n$;
these distinguished subgroups are the maximal 2-ended subgroups that
give a splitting of $W_{n+1}$ and that are \emph{universally
  elliptic}, in the terminology of~\cite{guirardel-levitt}.
By work of Papasoglu~\cite{papasoglu}, a quasi-isometry $\phi: W_{n+1}\to W_{m+1}$ must, up to
  bounded distance, take $R_{n+1}$ quasi-isometrically to a 
Type 3
vertex
stabiliser in the image. Since there is only one $W_{m+1}$-orbit of
such vertices, we may assume, by postcomposing with an element of
$W_{m+1}$, that this is the vertex stabilised by $R_{m+1}$.
Furthermore, $\phi$ gives a bijection between 
 the distinguished
$2$-ended subgroups
over which $W_{n+1}$ splits and 
 the distinguished
$2$-ended subgroups over which
$W_{m+1}$ splits. We know that such subgroups are exactly the
conjugates of the $D_{i,j}$, so, in particular, $\phi|_{R_{n+1}}$ gives a quasi-isometry from $R_{n+1}$
  to $R_{m+1}$ that takes the collection $\mathcal{L}_{n+1}$ of
  $R_{n+1}$-conjugates of the various $D_{i,j}$ in
  $R_{n+1}$ bijectively to the collection $\mathcal{L}_{m+1}$ of $R_{m+1}$-conjugates of
  the $D_{i,j}$ in $R_{m+1}$.

The group $R_{n+1}$ has an index two free subgroup $F_n:=\langle x_0x_1,\dots,
x_0x_n\rangle$, with $R_{n+1}=F_n\amalg F_nx_0$ (note that $x_i = (x_0 x_i)^{-1}x_0$ for $1 \leq i \leq n$).
The map $\bar{\iota}_n: R_{n+1}\to F_n$ that forgets the
  coset representative is a quasi-isometry inverse to the inclusion
  $\iota_n: F_n\into R_{n+1}$.
Let $y_i:=x_0x_i$ for each $1\leq i\leq n$. 
Notice that for each $0 \leq i < j \leq n$ the dihedral group $D_{i,j} = \langle x_i,x_j\rangle$ contains an index two infinite cyclic
subgroup $\langle x_ix_j\rangle$.
In the case $i>0$ we have $\langle x_ix_j\rangle=\langle y_i^{-1}y_j\rangle < F_n$.
Furthermore, each $R_{n+1}$-conjugate of $\langle y_i^{-1}y_j\rangle$
is of the form $f^{-1} \langle y_i^{-1}y_j\rangle f$ or $x_0f^{-1}
\langle y_i^{-1}y_j\rangle fx_0$, where $f\in F_n$.
Both of these are, in fact, $F_n$-conjugates of $\langle
y_i^{-1}y_j\rangle$. 
For the former this is clear. For the latter, note that for
$f':=x_0fx_0\in F_n$ we have $x_0f^{-1}y_i^{-1}y_j
fx_0=(f')^{-1}x_0x_ix_jx_0f'=(f')^{-1}y_iy_j^{-1}f'=(y_j^{-1}f')^{-1}(y_i^{-1}y_j)^{-1}(y_j^{-1}f')$.
Similarly, when $i=0$ we have $\langle x_ix_j\rangle=\langle
y_j\rangle<F_n$ and $x_0f^{-1}y_jfx_0=(f')^{-1}y_j^{-1}f'$.
Therefore, $\iota_n$ and $\bar{\iota}_n$  induce inverse bijections
  between $\mathcal{L}_{n+1}$ and the collection of $F_n$-conjugates of the cyclic subgroups
$\langle y_i\rangle$ for $1\leq i \leq n$ and $\langle
y_i^{-1}y_j\rangle$ for $1\leq i<j\leq n$.
We may then denote the latter collection by
$\bar{\iota}_n(\mathcal{L}_{n+1})$.
This means that the quasi-isometry $\phi$ induces a quasi-isometry
$\bar{\iota}_m\circ\phi|_{R_{n+1}}\circ\iota_n:F_n\to F_m$ that takes
$\bar{\iota}_n(\mathcal{L}_{n+1})$ bijectively to
$\bar{\iota}_m(\mathcal{L}_{m+1})$.
We use techniques of \cite{cashen-macura} to show that such a
quasi-isometry does not exist when $n\neq
m$, as follows.

Such a quasi-isometry extends to a
  homeomorphism
  $\partial(\bar{\iota}_m\circ\phi|_{R_{n+1}}\circ\iota_n):\partial
  F_n\to\partial F_m$.
Each element of  $\bar{\iota}_n(\mathcal{L}_{n+1})$, being a cyclic
subgroup of $F_n$, has two boundary points in $\partial F_n$, and the
boundary pairs for distinct elements of
$\bar{\iota}_n(\mathcal{L}_{n+1})$ are disjoint. 
Since $\bar{\iota}_m\circ\phi|_{R_{n+1}}\circ\iota_n$ takes
$\bar{\iota}_n(\mathcal{L}_{n+1})$ bijectively to
$\bar{\iota}_m(\mathcal{L}_{m+1})$, its extension to the boundary
takes the collection of boundary pairs of $\partial F_n$ coming from
$\bar{\iota}_n(\mathcal{L}_{n+1})$ bijectively to the collection of
boundary pairs of $\partial F_m$ coming from $\bar{\iota}_m(\mathcal{L}_{m+1})$.
Thus, $\partial(\bar{\iota}_m\circ\phi|_{R_{n+1}}\circ\iota_n)$ descends to a
homeomorphism $\partial F_n /\sim\to\partial F_m/\sim$ between
`decomposition spaces', where the decomposition space $\partial F_n
/\sim$ is the space obtained from $\partial F_n$ by, for each element
of $\bar{\iota}_n(\mathcal{L}_{n+1})$, collapsing its pair of boundary
points to a single point.

The topology of the decomposition space is related to Whitehead graphs. 
See~\cite[Section 3.1]{cashen-macura} for the definition of the Whitehead graph for a set of words.
For $W_{n+1}$, we are interested in 
the Whitehead graph for the set of words $\{y_i,y_j^{-1}y_k\mid 1\leq
i\leq n,\, 1\leq j<k\leq n\}$ in $F_n$.
This is a graph with $2n$ vertices labelled by the $y_i$ and
$y_i^{-1}$. 
For each $i$ there is an edge between (the vertices labelled) $y_i$ and $y_i^{-1}$, and for
each $j\neq k$ there is an edge between $y_j$ and $y_k$ and an edge
between $y_j^{-1}$ and $y_k^{-1}$.
In this graph the cut sets of size $n$ are of two types: the $n$ edges
incident to some vertex or the $n$ edges connecting the positive
generators to their inverses. One sees inductively that there are no smaller cut sets. 
Furthermore, by the same kind of argument as in
\cite[Section~6.2]{cashen-macura}, every cut set of size at most $n$ in the
decomposition space belongs to the $F_n$-orbit of one corresponding to
a cut set that is visible in the Whitehead graph. 
In particular, the smallest cut sets in the decomposition space have
size exactly $n$. The size of the smallest cut set is a homeomorphism
invariant, so the topology of the decomposition space gives an
obstruction to the existence of a quasi-isometry between $W_{n+1}$ and
$W_{m+1}$ when $n\neq m$.
\end{proof}

\bibliographystyle{siam}
\bibliography{refs}

\begin{thebibliography}{10}

\bibitem{abddy}
{\sc A.~Abrams, N.~Brady, P.~Dani, M.~Duchin, and R.~Young}, {\em Pushing
  fillings in right-angled {A}rtin groups}, J. Lond. Math. Soc. (2), 87 (2013),
  pp.~663--688.

\bibitem{behrstock-counterexample}
{\sc J.~Behrstock}, {\em A counterexample to questions about boundaries,
  stability, and commensurability},  (2017).
\newblock arxiv.org:1705.03984.

\bibitem{behrstock-charney}
{\sc J.~Behrstock and R.~Charney}, {\em Divergence and quasimorphisms of
  right-angled {A}rtin groups}, Math. Ann., 352 (2012), pp.~339--356.

\bibitem{behrstock-drutu}
{\sc J.~Behrstock and C.~Dru{\c{t}}u}, {\em Divergence, thick groups, and short
  conjugators}, Illinois J. Math., 58 (2014), pp.~939--980.

\bibitem{behrstock-drutu-mosher}
{\sc J.~Behrstock, C.~Dru{\c{t}}u, and L.~Mosher}, {\em Thick metric spaces,
  relative hyperbolicity, and quasi-isometric rigidity}, Math. Ann., 344
  (2009), pp.~543--595.

\bibitem{bfrhs}
{\sc J.~Behrstock, V.~Falgas-Ravry, and T.~Hagen, Mark F.~and~Susse}, {\em
  Global structural properties of random graphs},  (2015).
\newblock arXiv:1505.01913.

\bibitem{behrstock-hagen-sisto-caprace}
{\sc J.~Behrstock, M.~Hagen, and A.~Sisto}, {\em Thickness, relative
  hyperbolicity, and randomness in {C}oxeter groups}, Algebr. Geom. Topol., 17
  (2017), pp.~705--740.
\newblock With an appendix written jointly with Pierre-Emmanuel Caprace.

\bibitem{behrstock-jan-neumann}
{\sc J.~A. Behrstock, T.~Januszkiewicz, and W.~D. Neumann}, {\em
  Quasi-isometric classification of some high dimensional right-angled {A}rtin
  groups}, Groups Geom. Dyn., 4 (2010), pp.~681--692.

\bibitem{behrstock-neumann}
{\sc J.~A. Behrstock and W.~D. Neumann}, {\em Quasi-isometric classification of
  graph manifold groups}, Duke Math. J., 141 (2008), pp.~217--240.

\bibitem{bestvina-kleiner-sageev-RAAG}
{\sc M.~Bestvina, B.~Kleiner, and M.~Sageev}, {\em The asymptotic geometry of
  right-angled {A}rtin groups. {I}}, Geom. Topol., 12 (2008), pp.~1653--1699.

\bibitem{bowditch}
{\sc B.~H. Bowditch}, {\em Cut points and canonical splittings of hyperbolic
  groups}, Acta Math., 180 (1998), pp.~145--186.

\bibitem{bowditch-treelike}
{\sc B.~H. Bowditch}, {\em Treelike structures arising from continua and
  convergence groups}, Mem. Amer. Math. Soc., 139 (1999), pp.~viii+86.

\bibitem{cashen-macura}
{\sc C.~H. Cashen and N.~s. Macura}, {\em Line patterns in free groups}, Geom.
  Topol., 15 (2011), pp.~1419--1475.

\bibitem{cashen-martin}
{\sc C.~H. Cashen and A.~Martin}, {\em Quasi-isometries between groups with
  two-ended splittings}, Math. Proc. Cambridge Philos. Soc., 162 (2017),
  pp.~249--291.

\bibitem{casson-jungreis}
{\sc A.~Casson and D.~Jungreis}, {\em Convergence groups and {S}eifert fibered
  {$3$}-manifolds}, Invent. Math., 118 (1994), pp.~441--456.

\bibitem{charney-sultan}
{\sc R.~Charney and H.~Sultan}, {\em Contracting boundaries of {$\rm {CAT}(0)$}
  spaces}, J. Topol., 8 (2015), pp.~93--117.

\bibitem{crisp-paoluzzi}
{\sc J.~Crisp and L.~Paoluzzi}, {\em Commensurability classification of a
  family of right-angled {C}oxeter groups}, Proc. Amer. Math. Soc., 136 (2008),
  pp.~2343--2349.

\bibitem{dani-thomas}
{\sc P.~Dani and A.~Thomas}, {\em Divergence in right-angled {C}oxeter groups},
  Trans. Amer. Math. Soc., 367 (2015), pp.~3549--3577.

\bibitem{davis-book}
{\sc M.~W. Davis}, {\em The geometry and topology of {C}oxeter groups}, vol.~32
  of London Mathematical Society Monographs Series, Princeton University Press,
  Princeton, NJ, 2008.

\bibitem{davis-jan}
{\sc M.~W. Davis and T.~Januszkiewicz}, {\em Right-angled {A}rtin groups are
  commensurable with right-angled {C}oxeter groups}, J. Pure Appl. Algebra, 153
  (2000), pp.~229--235.

\bibitem{drutu}
{\sc C.~Dru{\c{t}}u}, {\em Relatively hyperbolic groups: geometry and
  quasi-isometric invariance}, Comment. Math. Helv., 84 (2009), pp.~503--546.

\bibitem{gabai}
{\sc D.~Gabai}, {\em Convergence groups are {F}uchsian groups}, Ann. of Math.
  (2), 136 (1992), pp.~447--510.

\bibitem{gersten-quadratic}
{\sc S.~M. Gersten}, {\em Quadratic divergence of geodesics in {${\rm CAT}(0)$}
  spaces}, Geom. Funct. Anal., 4 (1994), pp.~37--51.

\bibitem{guirardel-levitt}
{\sc V.~Guirardel and G.~Levitt}, {\em {JSJ} decompositions of groups},
  (2016).
\newblock arxiv.org:1602.05139.

\bibitem{huang}
{\sc J.~Huang}, {\em Quasi-isometric classification of right-angled {A}rtin
  groups, {I}: {T}he finite out case}, Geom. Topol., 21 (2017), pp.~3467--3537.

\bibitem{lafont-3d}
{\sc J.-F. Lafont}, {\em Rigidity result for certain three-dimensional singular
  spaces and their fundamental groups}, Geom. Dedicata, 109 (2004),
  pp.~197--219.

\bibitem{lafont}
\leavevmode\vrule height 2pt depth -1.6pt width 23pt, {\em Diagram rigidity for
  geometric amalgamations of free groups}, J. Pure Appl. Algebra, 209 (2007),
  pp.~771--780.

\bibitem{levcovitz1}
{\sc I.~Levcovitz}, {\em Divergence of $\rm {CAT}(0)$ cube complexes and
  {C}oxeter groups},  (2016).
\newblock arXiv:1611.04378.

\bibitem{levcovitz2}
\leavevmode\vrule height 2pt depth -1.6pt width 23pt, {\em A quasi-isometry
  invariant and thickness bounds for right-angled {C}oxeter groups},  (2017).
\newblock arXiv:1705.06416.

\bibitem{malone}
{\sc W.~Malone}, {\em Topics in geometric group theory}, ProQuest LLC, Ann
  Arbor, MI, 2010.
\newblock Thesis (Ph.D.)--The University of Utah.

\bibitem{martin-swiatkowski}
{\sc A.~Martin and J.~{{\'S}}wi{\polhk{a}}tkowski}, {\em Infinitely-ended
  hyperbolic groups with homeomorphic {G}romov boundaries}, J. Group Theory, 18
  (2015), pp.~273--289.

\bibitem{mihalik-Tschantz}
{\sc M.~Mihalik and S.~Tschantz}, {\em Visual decompositions of {C}oxeter
  groups}, Groups Geom. Dyn., 3 (2009), pp.~173--198.

\bibitem{papasoglu}
{\sc P.~Papasoglu}, {\em Quasi-isometry invariance of group splittings}, Ann.
  of Math. (2), 161 (2005), pp.~759--830.

\bibitem{papasoglu-whyte}
{\sc P.~Papasoglu and K.~Whyte}, {\em Quasi-isometries between groups with
  infinitely many ends}, Comment. Math. Helv., 77 (2002), pp.~133--144.

\bibitem{rips-sela}
{\sc E.~Rips and Z.~Sela}, {\em Cyclic splittings of finitely presented groups
  and the canonical {JSJ} decomposition}, Ann. of Math. (2), 146 (1997),
  pp.~53--109.

\bibitem{sela}
{\sc Z.~Sela}, {\em Structure and rigidity in ({G}romov) hyperbolic groups and
  discrete groups in rank {$1$} {L}ie groups. {II}}, Geom. Funct. Anal., 7
  (1997), pp.~561--593.

\bibitem{stark}
{\sc E.~Stark}, {\em Abstract commensurability and quasi-isometry
  classification of hyperbolic surface group amalgams}, Geom. Dedicata, 186
  (2017), pp.~39--74.

\bibitem{swarup}
{\sc G.~A. Swarup}, {\em On the cut point conjecture}, Electron. Res. Announc.
  Amer. Math. Soc., 2 (1996), pp.~98--100 (electronic).

\bibitem{swiatkowski-manifolds}
{\sc J.~{{\'S}}wi{\polhk{a}}tkowski}, {\em Trees of manifolds as boundaries of
  spaces and groups},  (2013).
\newblock arxiv:1304.5067.

\bibitem{swiatkowski-metric}
\leavevmode\vrule height 2pt depth -1.6pt width 23pt, {\em Trees of metric
  compacta and trees of manifolds},  (2013).
\newblock arxiv:1304.5064.

\bibitem{swiatkowski-dense}
\leavevmode\vrule height 2pt depth -1.6pt width 23pt, {\em The dense amalgam of
  metric compacta and topological characterization of boundaries of free
  products of groups}, Groups Geom. Dyn., 10 (2016), pp.~407--471.

\bibitem{swiatkowski-sierpinski}
\leavevmode\vrule height 2pt depth -1.6pt width 23pt, {\em Hyperbolic {C}oxeter
  groups with {S}ierpi{\'n}ski carpet boundary}, Bull. Lond. Math. Soc., 48
  (2016), pp.~708--716.

\bibitem{tran}
{\sc H.~C. Tran}, {\em Divergence spectra and morse boundaries of relatively
  hyperbolic groups},  (2016).
\newblock arXiv:1611.05005.

\bibitem{tukia}
{\sc P.~Tukia}, {\em Homeomorphic conjugates of {F}uchsian groups}, J. Reine
  Angew. Math., 391 (1988), pp.~1--54.

\end{thebibliography}

\end{document}